\documentclass[11pt]{amsart}
\usepackage{amssymb,amsmath,amsthm,amsfonts,mathrsfs}
\usepackage[hmargin=3cm,vmargin=3.5cm]{geometry}
\usepackage{color}
\usepackage{graphicx}
\usepackage{amscd}
\usepackage{stmaryrd}  
\usepackage{mathtools}
\usepackage{multirow}
\usepackage{bookmark}

\usepackage{tikz}

\newtheorem{theorem}{Theorem}[section]
\newtheorem{lemma}[theorem]{Lemma}
\newtheorem{remark}[theorem]{Remark}

\newtheorem{definition}[theorem]{Definition}
\newtheorem{conjecture}[theorem]{Conjecture}
\newtheorem{proposition}[theorem]{Proposition}
\newtheorem{prop}[theorem]{Proposition}

\newtheorem{corollary}{Corollary}

\newtheorem{example}[theorem]{Example}

\let\oldtocsection=\tocsection
\let\oldtocsubsection=\tocsubsection
\renewcommand{\tocsection}[2]{\hspace{0em}\oldtocsection{#1}{#2}}
\renewcommand{\tocsubsection}[2]{\hspace{1em}\oldtocsubsection{#1}{#2}}

\usepackage{chngcntr}
\counterwithin{figure}{subsection}

\title[Two-dimensional topological theories]{Two-dimensional topological theories, rational functions and their tensor envelopes}

\author{Mikhail Khovanov}
 \address{Department of Mathematics, Columbia University, New York, NY 10027, USA}
 \email{\href{mailto:khovanov@math.columbia.edu}{khovanov@math.columbia.edu}}
 \author{Victor Ostrik}
 \address{Department of Mathematics, University of Oregon, Eugene, OR 97403, USA
 \newline
 \indent Laboratory of Algebraic Geometry,
National Research University Higher School of Economics, Moscow, Russia}
 \email{\href{mailto:ostrik@uoregon.edu}{ostrik@uoregon.edu}}
 \author{Yakov Kononov}
 \address{Department of Mathematics, Columbia University, New York, NY 10027, USA,
 \newline
 \indent  Kharkevich Institute for Information Transmission Problems, Moscow, 127994, Russia}
 \email{\href{mailto:ya.kononoff@gmail.com}{ya.kononoff@gmail.com}}

\date{December 1, 2020}

\begin{document}

\begin{abstract}
 We study generalized Deligne categories and related tensor envelopes for the universal two-dimensional cobordism theories described by rational functions, recently defined by Sazdanovic and one of the authors.
\end{abstract}

\def\R{\mathbb R}
\def\Q{\mathbb Q}
\def\Z{\mathbb Z}
\def\N{\mathbb N}
\def\C{\mathbb C}
\def\S{\mathbb S}
\def\CP{\mathbb P}
\renewcommand\SS{\ensuremath{\mathbb{S}}}
\def\l{\lbrace}
\def\r{\rbrace}
\def\o{\otimes}
\def\lra{\longrightarrow}
\def\Ext{\mathrm{Ext}}
\def\mc{\mathcal}
\def\mf{\mathfrak}
\def\uFr{\underline{\mathrm{Fr}}}
\def\mcC{\mathcal{C}}
\def\RB{\mathrm{RB}}

\def\mfgl{\mathfrak{gl}}
\def\mfglN{\mathfrak{gl}_N}
\def\Cat{\mathrm{Cat}}
\def\ovF{\overline{F}}
\def\ovb{\overline{b}}
\def\tr{{\mathrm{tr}}}
\def\det{\mathrm{det}}
\def\Mantwo{\mathrm{Man}_2}
\def\tral{\tr_{\alpha}}
\def\one{\mathbf{1}}   
\def\nchar{\mathrm{char}\,}  
\def\uRep{\underline{\mathrm{Rep}}}

\newcommand{\myrightleftarrows}[1]{\mathrel{\substack{\xrightarrow{#1} \\[-.9ex] \xleftarrow{#1}}}}

\def\lra{\longrightarrow}
\def\kk{\mathbf{k}}  

\def\gdim{\mathrm{gdim}}  
\def\rk{\mathrm{rk}}
\def\mmat{\mathrm{Mat}}

\def\Pa{\mathrm{Pa}}   
\def\Cob{\mathrm{Cob}}
\def\Cobtwo{\Cob_2}   
\def\Kob{\mathrm{Kob}}
\def\PCobn{\mathrm{PCob}}  
\def\Cobal{\Cob_{\alpha}}  
\def\Cobalp{\Cob_{\alpha}'}
\def\Kobal{\Kob_{\alpha}}   
\def\PCob{\mathrm{PCob}}
\def\PKob{\mathrm{PKob}}
\def\PCobal{\PCob_{\alpha}}
\def\PKobal{\PKob_{\alpha}}
\def\KPobal{\mathrm{PKob}_{\alpha}}
\def\Kar{\mathrm{Kar}}   
\def\delcat{\mathrm{Rep}(S_t)}   
\def\delcatn{\mathrm{Rep}(S_n)}   
\def\delcatk{\mathrm{Rep}(S_k)}   
\def\delcatbart{\underline{\mathrm{Rep}}(S_t)}

\def\vcobal{\mathrm{VCob}_{\alpha}} 
\def\scobal{\mathrm{SCob}_{\alpha}} 
\def\dcobal{\mathrm{DCob}_{\alpha}} 
\def\dcobalb{\mathrm{DCob}_{\alpha}^{\bullet}}
\def\udcobal{\underline{\mathrm{DCob}}_{\,\alpha}}  
\def\cobal{\mathrm{Cob}_{\alpha}} 

\def\vcob{\mathrm{VCob}}
\def\scob{\mathrm{SCob}}
\def\dcob{\mathrm{DCob}}
\def\udcob{\underline{\mathrm{DCob}}}

\def\dmod{\mathrm{-mod}}   
\def\pmod{\mathrm{-pmod}}    

\newcommand{\brak}[1]{\ensuremath{\left\langle #1\right\rangle}}
\newcommand{\oplusop}[1]{{\mathop{\oplus}\limits_{#1}}}
\newcommand{\ang}[1]{\langle #1 \rangle }
\newcommand{\angf}[1]{\langle #1 \rangle }
\newcommand{\bbn}[1]{\mathbb{B}^{#1}}

\newcommand{\pseries}[1]{\kk\llbracket #1 \rrbracket}
\newcommand{\rseries}[1]{R\llbracket #1 \rrbracket}
\newcommand{\ovfi}[1]{\overline{F_{#1}}}
\newcommand{\delcatv}[1]{\mathrm{Rep}(S_{#1})}
\newcommand{\delcatbar}[1]{\underline{\mathrm{Rep}}(S_{#1})}
\newcommand{\undrep}{\underline{\mathrm{Rep}}}
\newcommand{\mfg}{\mathfrak{g}}


\newcommand{\cC}{{\mathcal C}} 
\newcommand{\cA}{{\mathcal A}}
\newcommand{\cZ}{{\mathcal Z}}
\newcommand{\fA}{{\mathfrak A}}
\newcommand{\cD}{{\mathcal D}}
\newcommand{\M}{{\mathcal M}}
\newcommand{\be}{{\bf 1}}
\newcommand{\bl}{{\bf s}}
\newcommand{\BZ}{{\mathbb Z}}
\newcommand{\BN}{{\mathbb N}}
\newcommand{\BQ}{{\mathbb Q}}
\newcommand{\BC}{{\mathbb C}}
\newcommand{\eps}{{\varepsilon}}
\newcommand{\sq}{$\square$}
\newcommand{\bi}{\bar \imath}
\newcommand{\bj}{\bar \jmath}
\newcommand{\Ve}{\mbox{Vec}}
\newcommand{\sVec}{\mbox{sVec}}
\newcommand{\Rep}{\mathrm{Rep}}
\newcommand{\FP}{{\rm FPdim}}
\newcommand{\Fun}{\mbox{Fun}}
\newcommand{\Ver}{\mbox{Ver}}
\newcommand{\Hom}{\mathrm{Hom}}
\newcommand{\Mod}{\mbox{Mod}}
\newcommand{\Bimod}{\mbox{Bimod}}
\newcommand{\Ind}{\mbox{Ind}}
\newcommand{\id}{\mbox{id}}
\newcommand{\inv}{\mbox{inv}}
\newcommand{\ot}{\otimes}
\newcommand{\Id}{\mbox{Id}}
\newcommand{\End}{\mbox{End}}
\newcommand{\iHom}{\underline{\mbox{Hom}}}
\newcommand{\PD}{\mathbb{PD}}
\newcommand{\PS}{\mathbb{PS}}

\def\YK#1{{\color{blue}[YK: #1]}}
\def\MK#1{{\color{red}[MK: #1]}}%
\def\VO#1{{\color{green}[VO: #1]}}%

\maketitle
\tableofcontents

\section{Introduction}

Throughout the paper we work over a field $\kk$, occasionally specializing to a characteristic zero field.

The Deligne category $\delcat$, where $t$ ranges over elements of $\kk$, interpolates between categories of finite-dimensional representations $\Rep(S_n)$ of the symmetric group $S_n$ over $\kk$, for various  $n$~\cite{D1,CO}. It is  a rigid symmetric monoidal Karoubi-closed $\kk$-linear category that depends on a parameter  $t\in \kk$. The Deligne category is equipped with a natural symmetric  trace  form $\tr$ that allows to form the ideal $J_t\in \delcat$ of  negligible morphisms. A morphism $x:a\to b$ is  negligible  if for  any  $y:b\to a$ the trace   $\tr(yx)=0$. The  ideal $J_t$ is non-trivial only when $t=n$ is a non-negative integer, and the quotient  category
\begin{equation}
    \delcatbar{n}\ := \ \delcatn/J_n
\end{equation}
 is equivalent to the  category of symmetric  group  representations in characteristic $0$, with some modifications needed in characteristic $p$~\cite{H}. For the other values of $t$ the ideal $J_t$ is  trivial and $\delcatbar{t}$ is equivalent  to  $\delcat$.

\vspace{0.07in}

As observed by Comes~\cite{C},  there is a functor from the category $\Cobtwo$ of oriented two-dimensional cobordisms between  one-manifolds to the Deligne category. Modifications of this functor, coupled with the universal construction of two-dimensional topological theories~\cite{BHMV,Kh1}, lead to generalizations   of the  Deligne category $\delcat$ and of its quotient $\delcatbar{t}$ by the negligible ideal~\cite{KS}.

Objects in $\Cobtwo$ are  non-negative integers $n\in \Z_+$ and morphisms from $n$  to  $m$ are oriented two-dimensional cobordisms from the union of $n$  circles to the union of $m$ circles, up to rel boundary  diffeomorphisms~\cite{KS}.

\vspace{0.07in}

Working over $\kk$, choose a rational function and its power series expansion   (the series of $\alpha$)
\begin{equation}\label{eq_Z_rat}
    Z_{\alpha}(T) = \frac{P(T)}{Q(T)} =\sum_{n\ge 0}\alpha_n T^n, \
    \alpha_n \in \kk, \  \alpha=(\alpha_0,\alpha_1,\dots) ,
\end{equation}
where polynomials $P(T),Q(T)\in \kk[T]$ are relatively prime and
\begin{equation}\label{eq_NMK}
    N = \deg P(T) , \ M = \deg Q(T), \ K:=\max(N+1,M).
\end{equation}
We may also write $P_{\alpha}(T)$ and $Q_{\alpha}(T)$ in lieu of $P(T)$ and $Q(T)$ to emphasize dependence on $\alpha$, and normalize so that
\begin{equation}\label{eq_Z_rat_9}
    Z_{\alpha}(T) = \frac{P_{\alpha}(T)}{Q_{\alpha}(T)} =\sum_{n\ge 0}\alpha_n T^n, \ \ (P_{\alpha}(T),Q_{\alpha}(T))=1, \ \ Q_{\alpha}(0)=1.
\end{equation}
With this normalization, $P_{\alpha}(T)$ and $Q_{\alpha}(T)$ are uniquely determined by $\alpha$. We may refer to $\alpha$ as a (rational) sequence and to $Z_{\alpha}(T)$ as its associated series. To keep track of $\alpha$, we may also denote
\begin{equation}\label{eq_NMK_alpha}
    N_{\alpha} = \deg P_{\alpha}(T) , \ M_{\alpha} = \deg Q_{\alpha}(T), \ K_{\alpha}:=\max(N_{\alpha}+1,M_{\alpha})
\end{equation}
in place of (\ref{eq_NMK}).

One can "linearize" the category $\Cobtwo$ of two-dimensional cobordisms using the sequence $\alpha$.  To do that, first allow linear combinations of cobordisms with coefficients in $\kk$ and also evaluate a closed connected oriented surface  of genus $g$, when it is a component of a cobordism, to $\alpha_g\in \kk$. This results  in the $\kk$-linear tensor category $\vcobal$, which has the  same  objects  $n\in \Z_+$ as  $\Cobtwo$ (this category is denoted $\Cobalp$ in~\cite{KS}).

In the notation $\vcobal$ letter $V$  stands for  \emph{viewable} or  $\emph{visible}$. A 2D cobordism  $x$ is called \emph{viewable} or \emph{visible} if it has no closed components, that is, any connected component of $x$  has  non-empty boundary.
The above evaluation of closed components allows to reduce a morphism from $n$ to $m$ to a linear combination  of viewable  cobordisms. Thus,
morphisms in $\vcobal$ are  $\kk$-linear combinations of viewable cobordisms, with  the composition of  morphisms given by composition of cobordisms and the above \emph{$\alpha$-evaluation} applied to all closed components of the composition.  Note that  hom spaces in $\vcobal$ are infinite-dimensional, since a component of a cobordism can have any number  of handles.

\vspace{0.07in}

A further reduction in the size of the  category is given by considering the ideal $J_{\alpha}\subset \vcobal$ of negligible morphisms, relative to the trace form $\tral$ associated with $\alpha$ and forming the quotient category
\begin{equation}
    \cobal := \vcobal/J_{\alpha}.
\end{equation}
In this paper we  call $\cobal$ the \emph{gligible quotient} of  $\vcobal$. We choose this  terminology over more cumbersome \emph{non-negligible quotient} and over  \emph{radical quotient}, for the latter may be  somewhat ambiguous.

The trace form is given on a cobordism $y$ from $n$ to $n$ by closing it via $n$ annuli connecting $n$ top and $n$ bottom circles of the boundary  of  $y$ into a closed oriented surface $\widehat{y}$ and  applying $\alpha$,
\begin{equation}
  \tral(y) := \alpha(\widehat{y}).
\end{equation}
It is shown in~\cite{KS} that the hom spaces in $\cobal$ are finite-dimensional over $\kk$ iff the generating function $Z_{\alpha}(T)$ is  rational, see formula (\ref{eq_Z_rat}), although the category $\cobal$ for any sequence $\alpha$.

Notation
\begin{equation}
    A_{\alpha}(n) := \Hom_{\Cobal}(0,n)
\end{equation}
is used in~\cite{Kh1} to denote the  \emph{state space} of $n$ circles in the  theory associated  to $\alpha$. Vector space $A_{\alpha}(n)$ can be  described as the quotient of the $\kk$-vector space with a basis $\{[S]\}_S$ of viewable cobordisms $S$ with $n$ boundary circles modulo skein relations that hold in the evaluation $\alpha$ for any closure of these $n$ circles by a cobordism $T$  on the other size, so that  $TS$ are closed surfaces and evaluations $\alpha(TS)$ make sense as elements of $\kk$. These state spaces are finite-dimensional for rational sequences $\alpha$.

\vspace{0.07in}

Category  $\cobal$  is a tensor $\kk$-linear category with  objects $n\in \Z_+$ and finite-dimensional hom spaces (recall the assumption that $\alpha$ is rational).
One can form the additive Karoubi  closure
\begin{equation}
    \udcobal := \Kar(\cobal^{\oplus})
\end{equation}
by first allowing formal finite direct sums of objects in $\cobal$ and extending morphisms correspondingly to get the finite additive closure $\cobal^{\oplus}$, then adding idempotents to get a Karoubi-closed  category.

Category $\udcobal$ is the  analogue of the gligible quotient of the Deligne category. It is a $\kk$-linear additive idempotent-complete rigid symmetric monoidal category with finite-dimensional hom spaces. It  carries  nondegenerate bilinear pairings   on its hom spaces
\[
\Hom(a,b) \otimes \Hom(b,a) \lra \kk.
\]
This category is the  analogue of the category $\delcatbar{t}$ above given by modding out the  Deligne category $\delcat$ by the ideal of negligible morphisms.

\vspace{0.07in}

To recover the analogue of the Deligne category itself, one needs to insert an intermediate category $\scobal$ (\emph{skein cobordisms}) into  the chain of categories and functors below, in between categories $\vcobal$ and  $\cobal$:
\begin{equation}\label{eq_four_cat}
    \Cobtwo  \lra \vcobal  \lra \scobal \lra \cobal \lra \udcobal.
\end{equation}

In this additional intermediate step, instead of modding out by all negligible  morphisms, one can first form a skein relation (the \emph{handle relation})  quotient  of $\vcobal$ which  reduces a component with $K$ handles ($K$ is given in formula (\ref{eq_NMK}))
to a linear  combination of components with  fewer handles, using the relation
\begin{equation}\label{eq_with_K}
     x^K - b_1 \, x^{K-1}+ b_2 \, x^{K-2}- \ldots + (-1)^M b_M \, x^{K-M} =0,
 \end{equation}
where $x$ denotes addition of a handle to a component, and $x^m$ stands for adding $m$ handles.

Polynomial in the  left hand side of (\ref{eq_with_K}) is the  minimal  polynomial of the handle endomorphism $x$ of  the circle in the category $\cobal$. We  call it \emph{the handle polynomial} of $\alpha$ and denote
\begin{equation}
    \label{eq_with_K_2}
    U_{\alpha}(x) \ := \  x^K - b_1 \, x^{K-1}+ b_2 \, x^{K-2}- \ldots + (-1)^M b_M \, x^{K-M}.
\end{equation}
The handle polynomial is monic. We use $x$  to denote the handle operator, use formal variable $T$ in the  power series, and switch between $x$ and $T$ as convenient throughout the paper.

 Coefficients $b_i$ of the handle polynomial are the  coefficients of  the polynomial in the denominator of (\ref{eq_Z_rat}) in the reverse order:
\begin{eqnarray} \label{eq_K}
  Q(T) & = & 1 - b_1 T +b_2 T^2  + \ldots + (-1)^M b_M T^M, \ b_i \in \kk,
 \end{eqnarray}
 also  see~\cite[Section 2.4]{Kh1}. We have
 \begin{equation}
     U_{\alpha}(x) = x^{K} Q(1/x) = x^{K-M}(x^M - b_1 \, x^{M-1}+ b_2 \, x^{K-2}- \ldots + (-1)^M b_M ).
 \end{equation}

 Recall that denominator $Q(T)$ is normalized to have constant term $1$, so that  $Q(0)=1$. Since $Q(T)$ is a denominator of the power series $Z_{\alpha}(T)$, it has a non-zero constant term, which is rescaled to $1$. Note that  changing $Z_{\alpha}(T)$ to $\lambda  Z_{\alpha}(T)$ for  $\lambda\in\kk^{\ast}$ does not change the above skein relation, but does change the  evaluation and the resulting categories $\vcobal, \cobal$, and $\udcobal$.

 When $P(T)/Q(T)$ in (\ref{eq_Z_rat}) is a \emph{proper} fraction, that is $\deg P(T)< \deg Q(T)$, then $K=M$ and $U_{\alpha}(0)=(-1)^Mb_M \not= 0$.

 We denote by $\scobal$ the quotient of $\vcobal$ by the handle relation (\ref{eq_with_K}).
 This category  is denoted $\PCobal$ in~\cite{KS}.
 The  quotient of $\scobal$ by the ideal of negligible morphisms (the  gligible quotient) is naturally isomorphic to $\cobal$ (isomorphic and not only equivalent, since objects of these categories are non-negative integers).

 We see that $\scobal$ has a place in  (\ref{eq_four_cat}) as an intermediate category between the two categories   in the middle.
 The additive Karoubi envelope of $\scobal$ is denoted by  $\dcobal$. It is the analogue of the  Deligne  category.
 There is a natural equivalence between the category obtained from $\scobal$ by first forming the additive Karoubi closure and  then  modding out by negligible morphisms (two consecutive right arrows then the down arrow in the square below) and  the category obtained from $\scobal$ by first forming  the quotient by negligible morphisms and  then passing  to the additive Karoubi envelope (down arrow followed by two right arrows in the square below).

 We summarize the resulting collection of categories and  functors between them in the  following diagram, with the square commutative.

\begin{equation} \label{eq_seq_cd_1}
\begin{CD}
\Cobtwo  @>>> \vcobal @>>> \scobal  @>>>  \scobal^{\oplus}  @>>> \dcobal  \\
@.    @.   @VVV @VVV   @VVV  \\
 @.     @.  \cobal @>>>  \cobal^{\oplus} @>>> \udcobal
\end{CD}
\end{equation}
The four rightmost  categories are additive, the three categories to the left of them are  $\kk$-linear and pre-additive (category $\Cobtwo$ is neither pre-additive nor $\kk$-linear). All eight categories are rigid symmetric monoidal. The six categories on the  right each have finite-dimensional hom  spaces. The bottom three categories are gligible quotients of the respective categories above them (that is, quotients by the ideals of negligible morphisms), and their hom spaces carry non-degenerate bilinear forms. The table below provides brief summaries for most of these categories.

\begin{center}
    \begin{tabular}{|l|c|}
    \hline
       Notation  &  Category \\
       \hline
        $\Cobtwo$ & Oriented 2D cobordisms \\
        \hline
        $\vcobal$ & Viewable cobordisms; evaluate closed components via $\alpha$ \\
        \hline
        \multirow{2}{4em}{$\scobal$} & "Skein" category; quotient of $\vcobal$ \\
        & by the handle relation (\ref{eq_with_K}) \\
        \hline
        \multirow{3}{4em}{$\cobal$} & "Gligible" quotient of $\scobal$ \\
        & by the kernels of trace forms \\ & (equivalently, by the negligible ideal) \\
        \hline
        \multirow{2}{4em}{$\dcobal$} & Deligne category; additive Karoubi \\
        & completion of the skein category $\scobal$ \\
        \hline
       \multirow{2}{4em}{$\udcobal$}  & Gligible quotient of the Deligne category; \\
       & equivalent to the additive Karoubi completion of $\cobal$ \\
        \hline
    \end{tabular}
\end{center}

Category  $\dcobal$ is the analogue of the Deligne category $\delcat$ and specializes to
it when the sequence $\alpha$ is constant,
\begin{equation}
    \alpha(t)=(t,t,\dots ), \ \  Z_{\alpha(t)}=\frac{t}{1-T}, \ t\in \kk.
\end{equation}

Category $\udcobal$ is the analogue of the quotient $\delcatbar{t}$ of $\delcat$ by  negligible morphisms. It specializes to  $\delcatbar{t}$ when  $\alpha$  is the constant  sequence $\alpha(t)$.

In this paper we study generalized Deligne  categories  $\dcobal$, their quotients $\udcobal$ as well as categories $\scobal$ and $\cobal$ for other rational series $\alpha$. We refer to these categories as \emph{tensor envelopes} of $\alpha$.

\vspace{0.07in}

For particular key rational generating functions $Z_{\alpha}(T)$ we establish or recall the connection between tensor envelopes of $\alpha$ and the known representation categories:
\begin{itemize}
    \item Generating function $\beta/(1-\lambda T)$ relates to the Deligne category of representations of symmetric group  $\Rep(S_t)$, $t=\beta\gamma$, see~\cite{KS} and Section~\ref{subsec_deligne}. For these series $\alpha$ the category $\dcobal$ is equivalent to $\Rep(S_t)$, inducing an equivalence of gligible quotient categories as well, $\udcobal \cong \uRep(S_t)$.
    \item Tensor envelopes for the constant generating function $Z(T)=\beta$, $\beta\in\kk$,  relate to the representation category of the Lie algebra $osp(1|2)$, see Section~\ref{sec_constant_gen} and Theorem~\ref{thm_dcob_C}.
    \item Categories for the linear generating function $\beta_0+\beta_1 T$ relate to the Deligne category $\Rep(O_t)$ for the orthogonal group, also known as the  unoriented Brauer category, and to its gligible quotients,  Section~\ref{subsec_linear}.
\end{itemize}

\vspace{0.07in}

Below is a brief summary, section by section, of the constructions  and results in the paper.

\begin{itemize}
    \item In Section~\ref{sec_general} we discuss basic properties  of tensor envelopes.
    \begin{itemize}
        \item Section~\ref{subsec-scaling} points out that the scaling $Z(T)\lra \lambda^{-1}Z(\lambda  T)$ for an invertible  $\lambda=\mu^2$ does not change the categories we consider.
        \item In Section~\ref{subset_frob_obj} we explain that  any commutative Frobenius algebra object  in a pre-additive  tensor  category gives rise  to a power  series  $\alpha$ with coefficients  in the commutative  ring $\End(\one)$  of endomorphisms  of the unit  object.
        \item In Section~\ref{subsec_universal} we recall the  universal property of $\Cobal$ and   Proposition~\ref{prop_assume}.
        \item Section~\ref{subsec_direct_sums} studies direct sum decompositions of commutative Frobenius algebra objects mirroring partial fraction  decompositions  of their rational generating series.
    \end{itemize}
    \item Section~\ref{sec_abel_real} contains key semisimplicity and abelian realization criteria for the tensor envelopes of $\alpha$, including Theorems~\ref{thm_semi_real} and~\ref{thm_both}. In particular, we classify series $\alpha$ with the semisimple category $\udcobal$.
    \item Section~\ref{sec_end_one} reviews properties of the endomorphism ring of the one-circle object in categories $\scobal$ and $\cobal$.
    \item  Section~\ref{sec_constant_gen} describes the structure of the gligible quotient category $\cobal$ for  the constant function (series $\alpha=(\beta,0,0,\dots)$). Theorem~\ref{thm_catalan} states that the  dimension of the state space $A(n)$ of $n$ circles for  this function is the Catalan number, for $\kk$ of characteristic $0$. A monoidal equivalence between the Karoubi envelope $\udcobal$ of $\cobal$ and a suitable category of finite-dimensional representations of  the Lie superalgebra $osp(1|2)$ is established in Section~\ref{subset_osp}.
    \item  Section~\ref{subsec_deligne} studies Gram determinants of a natural spanning set of surfaces for the function $\beta/(1-\gamma T)$, where tensor envelopes correspond to the Deligne category. These are rank one theories. Determinant computations for various rank two theories are given in Section~\ref{subsec_denom_two}.
    \item Section~\ref{sec_poly_gen} considers the case of a polynomial generating function, beyond the constant function case studied in Section~\ref{sec_constant_gen}. When the function is linear, associated tensor envelopes can be expressed via the unoriented Brauer category and its gligible quotient, due to the presence of a commutative Frobenius object in the Brauer category with a linear generating function, see Section~\ref{subsec_linear}. Section~\ref{subsec_deg_2_3} provides numerical data for the Gram determinants in categories when the generating function is a polynomial of degree two or three. Section~\ref{subsec_any_deg} considers arbitrary degree polynomials. A conjectural basis in the state space of $n$ circles for the theory is proposed there, and some properties of the Gram determinant for that set of vectors is established.
    \item In Section~\ref{sec_frac_genus} we explain how to enrich category $\Cobtwo$ of two-dimensional oriented cobordisms by adding codimension two defects (dots). Presence of the handle cobordism allows one to add relations intertwining the handle cobordism with dot decorations. Going from less general to more general examples, dots may be viewed as fractional handles, elements of a commutative monoid, or elements of a commutative algebra. Theory developed in the rest of this paper should extend to least some of these generalizations.
\end{itemize}

\vspace{0.1in}

{\bf Acknowledgments.} The  authors are grateful to Pavel Etingof for a valuable email  discussion. The authors would like to thank Vera Serganova and Noah Snyder for discussions of the Frobenius algebra described in Section~\ref{subset_osp}.
M.K. was partially supported by the NSF grant  DMS-1807425 while working on this paper.
Y.K. was partially supported by the NSF FRG grant DMS-1564497.
The work of V.~O. was partially supported by the HSE University Basic Research Program, Russian Academic Excellence Project '5-100' and by the NSF grant DMS-1702251.

\section{Properties of  \texorpdfstring{$\alpha$}{a}-theories}  \label{sec_general}

\subsection{Scaling by  invertible elements} \label{subsec-scaling}

Consider a  theory $\alpha$ over $\kk$ with generating function $Z_{\alpha}(T)$ and state spaces $A_{\alpha}(k)$ of $k$ circles. Choose an invertible element  $\mu\in \kk^{\ast}$, denote  $\lambda=\mu^2$, and change the sequence  $\alpha=(\alpha_0,\alpha_1,\dots)$  to
\begin{equation}
    \alpha'=(\lambda^{-1}\alpha_0,\alpha_1,\lambda \alpha_2, \lambda^2 \alpha_3, \dots ),
\end{equation}
that is, $(\alpha')_n = \lambda^{n-1}\alpha_n$.
The generating function for  $\alpha'$ is
\begin{equation}
 Z_{\alpha'}(T) = \lambda^{-1}Z_{\alpha}(\lambda T) =
 \lambda^{-1}\alpha_0+ \alpha_1 T +\lambda \alpha_2 T^2+ \lambda^2 \alpha_3 T^3  + \dots
\end{equation}
Note that $Z_{\alpha}(T)$ and $Z_{\alpha'}(T)$ have  the same linear term $\alpha_1$.

Consider the $\kk$-vector space $\uFr(k)$ with a basis $\{[S]\}_S$ given by surfaces $S$ without closed components and with $\partial S\cong \sqcup_k \SS^1$, one for each diffeomorphism class rel boundary of such surfaces.
Sequence $\alpha$ determines a $\kk$-bilinear symmetric form on $\uFr(k)$ with the pairing $(,)_{\alpha}$ given on generators by
\begin{equation}
    ([S_1],[S_2])_{\alpha}  = \alpha((-S_1)\sqcup S_2)
\end{equation}
and extended by linearity, where  $(-S_1)\sqcup S_2$ is the closed surface given by gluing $S_1$ and $S_2$ along the common boundary. Recall that the state space
\begin{equation}\label{eq_A_alpha}
    A_{\alpha}(k)  \ := \ \uFr(k)/\mathrm{ker}((,)_{\alpha})
\end{equation}
is the quotient of the free module by the kernel of this form.

Alternatively, consider  the bilinear form on $\uFr(k)$ given by $\alpha'$. The quotient of $\uFr(k)$ by the kernel of this form  is the state space for $\alpha'$:
\begin{equation}
    A_{\alpha'}(k)  \ := \ \uFr(k)/\mathrm{ker}((,)_{\alpha'}).
\end{equation}
Introduce the  $\kk$-linear map
\begin{equation} \label{eq_phi_S}
\phi:\uFr(k)\lra \uFr(k), \ \  \phi([S]) = \mu^{-\chi(S)}[S].
\end{equation}
This map scales  $[S]$ by  $\mu^{-\chi(S)}$,  where  $\chi(S)$ is the Euler characteristic of $S$. It  intertwines the bilinear forms on these spaces
\begin{equation}
    ([S_1],[S_2])_{\alpha'} = (\phi([S_1]),\phi([S_2]))_{\alpha}
\end{equation}
and induces an isomorphism of vector spaces
\begin{equation}
    \phi \ : \ A_{\alpha'}(k) \lra A_{\alpha} (k),
\end{equation}
also denoted $\phi$. This  isomorphism intertwines nondegenerate $R$-valued bilinear forms $(,)_{\alpha'}$ and $(,)_{\alpha}$ on these spaces and shows that $\alpha$ and $\alpha'$ define equivalent topological theories as defined in~\cite{Kh1}.

On the level of categories, the scaling map $\phi$ in (\ref{eq_phi_S})  induces $\kk$-linear isomorphisms between the morphism spaces in $\vcobal$ and $\vcob_{\alpha'}$
\begin{equation*}
\Hom_{\vcob_{\alpha'}}(n,m) \stackrel{\cong}{\lra}\Hom_{\vcobal}(n,m)
\end{equation*}
compatible with  the composition in these categories and leading to an isomorphism of categories  $\vcob_{\alpha'}\cong \vcobal$. This isomorphism is compatible with the various quotient and Karoubi envelope categories that follow and leads to isomorphisms or equivalences  of the corresponding categories, including isomorphisms
 $\scob_{\alpha'}\cong \scobal, $ $\mathrm{Cob}_{\alpha'}\cong \cobal$ and equivalences
 $\dcob_{\alpha'}\cong \dcobal$, $\udcob_{\alpha'}\cong \udcobal$.
We see that scaling by $\lambda=\mu^2$, $\mu\in \kk^{\ast}$,  gives isomorphic theories and isomorphic or equivalent associated categories. This scaling changes the handle relation by rescaling the handle.

\vspace{0.1in}

\subsection{Commutative Frobenius  algebra objects in symmetric  monoidal categories}\label{subset_frob_obj}

\quad

Let $\cC$ be a symmetric monoidal category. Let $A=(A,m,\iota)\in \cC$ be a commutative algebra object in $\cC$, i.e. an object $A\in \cC$ equipped with associative and commutative multiplication $m: A\otimes A\to A$ such that $\iota: \be \to A$ satisfies the unit axiom, see e.g~\cite[Sections 7.8.1 and 8.8.1]{EGNO}.
We say that $A$ is a  {\em commutative Frobenius monoid} in $\cC$ if it is equipped with a morphism $\epsilon: A\to \be$ such
that the composition $b: A\otimes A\xrightarrow{m} A\xrightarrow{\epsilon} \be$ is a non-degenerate  pairing, i.e. there exists a morphism $c: \be \to A\otimes A$ such that the morphisms $b$ and $c$ satisfy the axioms
of evaluation and coevaluation maps, see e.g. \cite[2.10.1]{EGNO}. Equivalently, the object $A^*$ exists
and the morphism $A\to A^*$ which is the image of $b$ under the natural isomorphism
$\Hom(A\ot A,\be)\simeq \Hom(A,A^*)$ is an isomorphism.  We will often identify $A$ and $A^*$
using this morphism.
For example we define the comultiplication morphism $\Delta: A\to A\otimes A$ as dual to the multiplication morphism. Clearly $\Delta$ is coassociative.
It is easy to see that $\Delta$ is a morphism of $A\times A^{op}$-objects. One shows that the morphism $c$ equals to the composition of $\iota$ and $\Delta$.

Given a Frobenius monoid $A\in \cC$ and $n\in \BZ_{\ge 0}$ we get a morphism $a_n: \be \xrightarrow{\iota} A\xrightarrow{\Delta_n} A^{\otimes n}\xrightarrow{m_n} A\xrightarrow{\epsilon} \be$
where $\Delta_n: A \to A^{\otimes n}$ is $n-$fold comultiplication and
$m_n: A^{\otimes n}\to A$ is $n-$fold multiplication.  Thus $a_0=\epsilon \, \iota$ is the composition of $\epsilon$ and $\iota$, and
$a_1=\epsilon\, m \, \Delta \, \iota =c \, b$ is the composition of $c$ and $b$, that is the dimension of $A$.

Equivalently, consider the \emph{handle} endomorphism
\begin{equation}
    x \ : \ A\xrightarrow{\Delta} A\otimes A \xrightarrow{m} A,
\end{equation}
which has  a  topological interpretation as a tube with a handle  on it.
Iterating this morphism yields $x^n$, a tube with $n$ handles. The  map $a_n$ can be  written as
\begin{equation}
    a_n  = \epsilon \, x^n \, \iota .
\end{equation}
Elements  $a_n$ are endomorphisms  of the  unit object $\be$ of $\cC$.

We distinguish between the \emph{handle endomorphism} $x$  above and the \emph{handle morphism}. The latter is the morphism $\be\lra  A$ given by  $x\iota$. Handle morphism corresponds to a one-punctured torus, with the puncture circle on the target  on the morphism,  while handle endomorphism corresponds to the  twice-punctured torus, with one  circle  on both the  source and target cobordisms.

From here on we
assume that $\cC$ is a $\kk-$linear symmetric monoidal category and the canonical map $\kk\to \End(\be)$
is an isomorphism. Then $a_n=\alpha_n\id_{\be}$ for some $\alpha_n\in \kk$.
 The sequence $\alpha=(\alpha_n)_{n\in \BZ_{\ge 0}}$ will be
called \emph{$\alpha-$evaluation} or just the  \emph{evaluation} of $A$.

\begin{example} \label{ortho example}
Let $\cC$ be an additive $\kk-$linear symmetric monoidal category and let $V\in \cC$ be an object equipped with a non-degenerate symmetric pairing $V\otimes V\to \be$. We define a $\BZ-$graded commutative Frobenius algebra $A=A(V)$ as follows: $A=A_0\oplus A_1\oplus A_2$ where $A_0=A_1=\be$ and $A_1=V$; $A_0$ is the image of the unit morphism, the multiplication $A_1\otimes A_1\to A_2$ is given by the symmetric pairing, and the linear form $\epsilon : A\to \be$ factors through
the projection $A\to A_0\oplus A_2$ and is nonzero when restricted to $A_2$. It is easy to verify that
the $\alpha-$evaluation of the algebra $A(V)$ has $\alpha_i=0$ for $i>1$; also $\alpha_1=\dim(A)=2+\dim(V)$. Parameter $\alpha_0$ is the composition of $\epsilon$ and the unit morphism and can be chosen to be any element of $\kk$. The generating function is then $Z_{\alpha}(T)=\alpha_0+(2+\dim(V))T$.  Possible dimensions $\dim(V)$ of such objects  $V$ depend on  $\cC$. For instance, when $\cC$ is the category of $\kk$-vector spaces, these dimensions belong to the image of $\Z_+\in \kk$.  When $\cC$ is the unoriented Brauer category $\Rep(O_t) $ with the parameter $t\in\kk$, $\dim (V)=t$ for the standard generator $V$ of $\cC$, see Section~\ref{subsec_linear}, for instance, and references there.
\end{example}

\vspace{0.1in}

\emph{Remark:} One can informally compare this setup with the problem of reconstructing or understanding a system from observable data on it. Here one can imagine that the system consists  of an object  $A\in \cC$, handle  endomorphism $x$ of $A$ (and, more generally, endomorphisms of $A$ associated to arbitrary cobordisms from a circle to itself). Object $A$ is unknown to us, but  we  can  observe  values of closed  cobordisms, which are $\alpha_n$ for a connected genus $n$ cobordism.  Then the universal pairing construction of~\cite{Kh1,KS} in dimension two (and its counterpart~\cite{BHMV} in three dimensions) consists of recovering a minimal model for $X$ and $\cC$ from the closed cobordism data. This toy  example in two dimensions can be compared to more complicated  reconstructions in control theory. We probe category $\cC$ via evaluations of closed cobordisms, which allow us to fully reconstruct it, in the universal pairing setup.

\vspace{0.1in}

\begin{example}\label{example_1}
Given two commutative Frobenius algebra objects $A_1,A_2$ in $\cC$, their sum $A_1\oplus A_2$ is  naturally a  commutative Frobenius algebra in $\cC$. If sequences $\alpha$  and  $\beta$ are  evaluations of $A_1$ and $A_2$, respectively, the  evaluation of $A_1\oplus  A_2$ is the sequence $\alpha+\beta=(\alpha_n+\beta_n)_{n\in \Z_+}.$
\end{example}

\begin{example}\label{example_2}
Hadamard product of power series $\alpha$ and $\beta$ is the series $\alpha\beta$ with $(\alpha\beta)_n=\alpha_n\beta_n$, that is, we multiply the two series term-wise. Hadamard product of rational power series is rational~\cite{LT}.
The tensor product $A_1\otimes A_2$ of commutative Frobenius algebras in $\cC$ is naturally a commutative Frobenius algebra in $\cC$.  The evaluation of $A_1\otimes A_2$ is the  Hadamard product of evaluations of $A_1$ and  $A_2$.

If $\nchar \kk=p$, the  $p$-th tensor power $A^{\otimes p}$ of  a commutative  Frobenius algebra $A$ has  evaluation $\alpha^p$ equal to the application of the Frobenius endomorphism of $\kk$ to each term of $\alpha$.
\end{example}

\vspace{0.1in}

\subsection{Universal property} \label{subsec_universal}

It is well known (see e.g. \cite[Theorem 0.1]{SP}) that the category $\Cobtwo$ has the following universal property: for a symmetric category $\cC$ an evaluation of a tensor functor on the circle object gives an
equivalence of categories
$$\{ \mbox{tensor functors}\; \Cobtwo \to \cC \} \to \{ \mbox{commutative Frobenius algebras in}\; \cC \} .$$
One deduces easily a similar universal property of $\kk\Cobtwo$ where the categories $\cC$ and functors are assumed
to be $\kk-$linear. Likewise,category $\vcobal$ has the following universal property: for an
$\kk-$linear symmetric category $\cC$ an evaluation at the circle object gives an
equivalence of categories:
$$\{ \kk-\mbox{linear tensor functors}\; \vcobal \to \cC \} \to
\left\{
\begin{tabular}{c}
commutative Frobenius algebras in $\cC$ \\   with evaluation $\alpha$
\end{tabular}
\right\}
$$

We pick an inverse equivalence of categories and for a commutative Frobenius algebra $A\in \cC$ we
will denote by $F_A$ the corresponding tensor functor,
\begin{equation} \label{eq_F_A}
    F_A \colon \vcobal \to \cC.
\end{equation}

A sequence $\alpha$ is called \emph{linearly recurrent} or \emph{homogeneously linearly recurrent} if $\alpha_{k+n+1}=a_n\alpha_{k+n}+a_{n-1}\alpha_{k+n-1}\dots + a_1 \alpha_{k+1}$ for all $k\ge N$ for some $N$ and fixed $a_1,\dots, a_n$, see~\cite{EPSW}. In this paper we refer to such $\alpha$ as recurrent sequences.

Assume that the sequence $\alpha$ is  recurrent. Functor $F_A$
factors through the category $\scobal$ if and only if $F_A$ annihilates the handle polynomial in (\ref{eq_with_K_2}).
If the category $\cC$ is Karoubian the functor $F_A$ extends uniquely to the
category $\dcobal$.

We will often use the following result, see \cite[Lemma 2.6]{BEEO}, specialized to $\udcobal$:

\begin{proposition} \label{prop_assume}
Assume that the category $\cC$ is a $\kk$-linear additive Karoubian nondegenerate symmetric monoidal category with finite-dimensional $\hom$ spaces and the functor $F_A$ satisfies the following  properties:

1) Any indecomposable object of $\cC$ is a direct summand of $F_A(n)$ for some object $n$  in $\vcobal$.

2) The functor $F_A$ is full (i.e. surjective on $\Hom$'s).

Then the functor $F_A$ induces an equivalence $\udcobal \simeq \cC$.
\end{proposition}

Here we say that $\cC$ is \emph{nondegenerate} if any negligible morphism is the zero morphism between some objects.

\subsection{Direct sums decompositions} \label{subsec_direct_sums}
Let $A\in \cC$ be a commutative Frobenius algebra object in a $\kk-$linear symmetric monoidal category $\cC$.
Then the multiplication in $A$ induces a  commutative algebra structure on the vector space $A_1:=\Hom(\be, A)$; this algebra acts on $A$ via left (equivalently, right) multiplications, so we
get a natural injective homomorphism
\begin{equation}\label{eq_phi}
\phi: A_1\to \End_{\cC}(A).
\end{equation}
Let $x_0\in A_1$ be the handle morphism; its image $\phi(x_0)$ is the handle endomorphism of $A$.
Let $A_0\subset A_1$ be the unital subalgebra generated by $x_0$.

We assume that $A_0$ is finite dimensional. Thus $x_0$ is annihilated by a nonzero polynomial. We let $U(T)\in \kk [T]$ be the minimal polynomial of $x_0$, which is  assumed monic. It factors
\begin{equation} \label{eq_a}
U(T)=T^a \underline{U}(T),
\end{equation}
with $\underline{U}(0)\not= 0$ and $a\ge 0$.

We recall that the idempotents $e\in A_0$ are naturally
labeled by factorizations $U(T)=U_1(T)U_2(T)$ such that the factors $U_1(T)$ and $U_2(T)$ are relatively prime.
Namely given such a factorization we can find $a(T), b(T)\in \kk [T]$ such that
$$a(T)U_1(T)+b(T)U_2(T)=1,$$
and then
\begin{equation}\label{eq_e_def}
  e\ =\ a(x_0)U_1(x_0) \ \in A_0
\end{equation}
is an idempotent. Conversely for an idempotent $e\in A_0$ let us choose
a polynomial $s(T)$ such that $e=s(x_0)$; then setting $U_1(T)=\mathrm{gcd}(U(T),s(T))$ and $U_2(T)=\mathrm{gcd}(U(T), 1-s(T))$
we get a factorization as above.

We furthermore assume that the category $\cC$ is Karoubian. Let $e\in A_0$ be an idempotent; then it is easy
to see that the image of $\phi(e)$, see (\ref{eq_phi}), is a Frobenius subalgebra of $A$ in $\cC$; moreover there is a decomposition
$$A=\phi(e)A\oplus \phi(1-e)A$$
of $A$ as a direct sum of its Frobenius subalgebras (direct  sum as objects in $\cC$ and  direct  product as  algebras). Note that the unit  elements of these subalgebras become  idempotents in $A$.

Let us compute $\alpha-$invariants of subalgebras $\phi(e)A$ and $\phi(1-e)A$ in terms of the $\alpha-$invariant of $A$. Recall that the generating function of $A$ can be written as a rational function
\begin{equation}
Z(T)=\frac{P(T)}{Q(T)}
\end{equation}
with $ Q(0)=1$,  where $Q(T)$ is the  polynomial given by
\begin{equation}\label{eq_p_and_P} Q(T)=T^{\, d} \, U(T^{-1}),
\ \ \   d=\deg(U(T)).
\end{equation}
Here $d$ is  the degree of $U(T)$, and $Q(T)$ is the reverse polynomial of $\underline{U}(T)$. The orders of the coefficients of $Q(T)$ and $U(T)$ are reversed. Note that $Q(0)=1$ since $U(T)$ is monic.

\begin{example} Let $U(T)=T^5+9T^4-6T^3$ then  $d=5$, $ \underline{U}(T)=T^2 +  9 T - 6$ and $Q(T)=-6T^2+9T+1$.
\end{example}

In particular, any factorization
\begin{equation}\label{eq_fac_1}
U(T)=U_1(T)U_2(T)
\end{equation}
for $U(T)$ as above into two relatively prime monic polynomials
induces a factorization
\begin{equation}
Q(T)=Q_1(T)Q_2(T),
\end{equation}
where $Q_1(T)$ and $Q_2(T)$ are determined from $U_1(T)$ and $U_2(T)$, respectively, in
the same way as $Q(T)$ is determined by $U(T)$, via relation (\ref{eq_p_and_P}).

Since polynomials $U_1(T)$ and $U_2(T)$ are relatively prime, at most one of them is divisible by $T$.
Thus we can and will assume that $U_2(T)$ is not divisible by $T$. Polynomial $Q(T)$ is divisible by $T$ iff  $a\not=0$ in formula (\ref{eq_a}).

There is  a unique partial fraction decomposition
\begin{equation}
Z(T) = \frac{v_1(T)}{Q_1(T)} + \frac{v_2(T)}{Q_2(T)},
\end{equation}
where $\deg(v_1)<\deg(Q_1)$. Denote the terms on the right hand side by  $Z_1(T)$ and $Z_2(T)$, respectively, and write
$$Z(T)=Z_1(T)+Z_2(T)$$
where $v_1(T)=Z_1(T)Q_1(T)$ and $v_2(T)=Z_2(T)Q_2(T)$ are polynomials and \begin{equation*}
    \deg v_1(T)=\deg (Z_1(T)Q_1(T))<\deg Q_1(T).
\end{equation*}
Thus, $Z_1(T)$ is a proper fraction, but $Z_2(T)$ may not be proper. Recall idempotent $e\in $ defined by (\ref{eq_e_def}).

\begin{proposition} \label{prop_not_above}  In the notations above, the generating functions of commutative Frobenius algebra objects $\phi(e)A$ and $\phi(1-e)A$ are
$Z_1(T)$ and $Z_2(T)$ respectively.
\end{proposition}

\begin{proof} Let $'x_0$ and $''x_0$ be the handle endomorphisms of the algebras $\phi(e)A$ and $\phi(1-e)A$, so that $'x_0=ex_0$ and $''x_0=(1-e)x_0$. The  evaluation series $\alpha'$ of $\phi(e)A$ has $n$-the coefficient
\begin{equation}\label{eq_implicit}
\alpha'_n= \epsilon(('x_0)^n)=\epsilon(ex_0^n)=\epsilon(s(x_0)x_0^n),
\end{equation}
where $s(T)$ is a polynomial such that $e=s(x_0)$. This shows
that the evaluation series $\beta$ and $\gamma$  of $\phi(e)A$ and $\phi(1-e)A$ are uniquely determined by
the factorization (\ref{eq_fac_1}) of $U(T)$. The expression (\ref{eq_implicit}) is  somewhat implicit since we don't write down a formula for $s(T)$.

\vspace{0.1in}

Pick commutative Frobenius algebras $A'\in \cC'$ and $A''\in \cC''$ with generating functions
$Z_1(T)$ and $Z_2(T)$, and such that their handle endomorphisms $x_0'$ and $x_0''$ have minimal polynomials $U_1(T)$ and $U_2(T)$ respectively. Here  $\cC'$ and  $\cC''$ are $\kk$-linear symmetric monoidal categories.  Such algebras and categories exists in view of~\cite{KS,Kh1}.

Consider commutative Frobenius algebra
\begin{equation}\label{eq_comm_fr}
A'\boxtimes \be \oplus \be \boxtimes A''\in \cC'\boxtimes \cC'',
\end{equation}
where $\cC'\boxtimes \cC''$ is the external tensor product of $\cC'$ and $\cC''$, see~\cite[Section 2.2]{O}. This is the \emph{naive} tensor product of $\kk$-linear monoidal categories.

A more sophisticated exterior tensor product was defined by Deligne for abelian monoidal categories, subject to additional assumptions, see~\cite{EGNO} and references therein, but it is not used here.

Objects  of the naive tensor product $\cC'\boxtimes \cC''$ are finite direct sums of external tensor products $V'\boxtimes V''$ of objects $V'$ and $V''$ of $\cC'$ and $\cC''$. The tensor product  $\cC'\boxtimes \cC''$ is additive but $\cC'$, $\cC''$ do not have to be additive, only $\kk$-linear.

The generating function
of this algebra is $Z_1(T)+Z_2(T)=Z(T)$ and its handle endomorphism $\tilde x_0$ is $x_0'\oplus x_0''$, where  $x_0'$ and $x_0''$ are handle endomorphisms of $A'$ and $A''$, respectively.
Hence, for any polynomial $a(T)$, we have $a(\tilde x_0)=a(x_0')\oplus a(x_0'')$. Since the polynomials
$U_1(T)$ and $U_2(T)$ are relatively prime, the minimal polynomial of $\tilde x_0$ is $U_1(T)U_2(T)=U(T)$.
Moreover it is clear that the idempotents determined by the factorization $U(T)=U_1(T)U_2(T)$ are precisely
 the unit elements of $A'\boxtimes \be$ and $\be \boxtimes A''$. Thus, the $\alpha-$invariants of the algebras
$A'\boxtimes \be$ and $\be \boxtimes A''$ can be computed via formula (\ref{eq_implicit}) applied to them. The result follows.
\end{proof}

\begin{example} Assume that $\nchar \kk \ne 2$ and the generating function of $A$ is
\begin{equation}
    Z_{\alpha}(T)=\frac{T^3+1}{1-3T+2T^2}.
\end{equation}
Then the handle polynomial  $U(T)=U_{\alpha}(T)=T^2(T^2-3T+2)=T^2(T-1)(T-2)$. Note that $T^2-3T+2$ is the reciprocal if $1-3T+2T^2$. The degree of $U(T)$ equals $K_{\alpha}=\max(\deg P_{\alpha}+1,\deg Q_{\alpha})=\max(3+1,2)=4$, see (\ref{eq_NMK_alpha}).

Consider the factorization (\ref{eq_fac_1}) with
$Q_1(T)=T^2(T-1)$ and $Q_2(T)=T-2$. Then, see formula (\ref{eq_e_def}),  $e=\frac14(x_0^3-x_0^2)$ and
$$Z_1(T)=\frac{9/4}{1-2T},\; \; Z_2(T)=\frac{-T^2/2-T/4-5/4}{1-T}.$$
\end{example}

 Rational series $\beta$ and $\gamma$ give rise to commutative Frobenius objects $A_{\beta}$ and $A_{\gamma}$ in the skein categories $\scob_{\beta}$ and $\scob_{\gamma}$, respectively. Consider the tensor product category
\begin{equation}
\scob_{\beta,\gamma} \ := \ \scob_{\beta}\boxtimes \scob{\gamma}
\end{equation}
with the Frobenius object
\begin{equation}\label{eq_obj_frob}
    A_{\beta,\gamma}\ := \ A_{\beta}\boxtimes \be \oplus \be \boxtimes A_{\gamma},
\end{equation}
see also (\ref{eq_comm_fr}). Let
 \begin{equation}\label{eq_Z_two}
     Z_{\beta}(T)=\frac{P_{\beta}(T)}{Q_{\beta}(T)}, \ \
     Z_{\gamma}(T)=\frac{P_{\gamma}(T)}{Q_{\gamma}(T)}
 \end{equation}
 be the standard presentations of rational series for $\beta$ and $\gamma$, see formulas (\ref{eq_Z_rat_9}), (\ref{eq_K}), with $Q_{\beta}(0)=Q_{\gamma}(0)=1$ and co-prime numerators and denominators in each of the two fractions.
 Polynomials  $U_{\beta}(x)$ and $U_{\gamma}(x)$  describe handle skein relations for series $\beta$ and $\gamma$, respectively. They are reciprocal polynomials of $Q_{\beta}(x)$ and $Q_{\gamma}(x)$, respectively, scaled by suitable powers of $x$ when the fractions are not proper.

 \begin{lemma}
 The handle polynomial of the Frobenius object $A_{\beta,\gamma}$ in $\scob_{\beta,\gamma}$  is
 \begin{equation}\label{eq_def_U_b_g}
 U_{\beta,\gamma}(x):=\mathrm{lcm}(U_{\beta}(x),U_{\gamma}(x)),
 \end{equation}
 the least common  multiple of $U_{\beta}(x)$ and $U_{\gamma}(x)$.
 \end{lemma}
 \begin{proof} The handle endomorphism of $A_{\beta,\gamma}$ is the sum of handle endomorphisms of its direct summands $A_{\beta}\boxtimes \be$ and $\be\boxtimes A_{\gamma}$.
 \end{proof}

 To understand the handle polynomial for $\beta+\gamma$, we convert the series for $\beta$ and $\gamma$ into sums of  proper fractions and polynomial terms:
 \begin{eqnarray*}
 Z_{\beta}(T) &= & \frac{\overline{P}_{\beta}(T)}{Q_{\beta}(T)} + R_{\beta}(T), \ \ \deg \overline{P}_{\beta}(T) < \deg  Q_{\beta}(T),  \\
 Z_{\gamma}(T) &= & \frac{\overline{P}_{\gamma}(T)}{Q_{\gamma}(T)} + R_{\gamma}(T), \ \ \deg \overline{P}_{\gamma}(T) < \deg  Q_{\gamma}(T).
 \end{eqnarray*}
 The handle polynomials for $\beta$ and $\gamma$ are the reciprocals of $Q_{\beta}(T)$ and $Q_{\gamma}(T)$ multiplied by $T$ to the exponent the degree of $R_{\beta}(T)$ and $R_{\gamma}(T)$, respectively.

From the corresponding decomposition for the series of $\beta+\gamma$,
 \begin{equation}
 Z_{\beta+\gamma}(T) \ = \  Z_{\beta}(T)+ Z_{\gamma}(T) \ = \  \frac{\overline{P}_{\beta+\gamma}(T)}{Q_{\beta+\gamma}(T)} + R_{\beta}(T)+R_{\gamma}(T),  \end{equation}
with the reduced fraction
 \begin{equation}
  \frac{\overline{P}_{\beta+\gamma}(T)}{Q_{\beta+\gamma}(T)}  \ = \  \frac{\overline{P}_{\beta}(T)Q_{\gamma}(T)+\overline{P}_{\gamma}(T)Q_{\beta}(T)}{Q_{\beta}(T)Q_{\gamma}(T)} ,
  \end{equation}
one sees that $Q_{\beta+\gamma}(T)$ is a divisor of $\mathrm{lcm}(Q_{\beta}(T),Q_{\gamma}(T))$ and $R_{\beta+\gamma}(T)=R_{\beta}(T)+R_{\gamma}(T)$.
  Consequently, the handle polynomial $U_{\beta+\gamma}(x)$ is the reciprocal of a divisor of $\mathrm{lcm}(Q_{\beta}(x),Q_{\gamma}(x))$ times a power of $x$ of degree  $\deg(R_{\beta}+R_{\gamma})\le \max(\deg R_{\beta},\deg R_{\gamma})$.

\begin{corollary} The handle  polynomial $U_{\beta+\gamma}(x)$ of $\beta+\gamma$ is a divisor of the polynomial $U_{\beta,\gamma}(x)$ in (\ref{eq_def_U_b_g}).
\end{corollary}

\begin{definition}
A pair $(\beta,\gamma)$ of rational sequences is called \emph{regular} if $U_{\beta+\gamma}(x)=U_{\beta,\gamma}(x)$.
\end{definition}

\begin{prop} $(\beta,\gamma)$ is regular (and
 $U_{\beta,\gamma}(x)=U_{\beta+\gamma}(x)$) iff
 there is a functor
 \begin{equation}\label{eq_want_func}
     F_{\beta,\gamma}^S \ : \ \scob_{\beta+\gamma} \lra \scob_{\beta}\boxtimes \scob{\gamma}
 \end{equation}
 taking the circle object $A_{\beta+\gamma}$ of $\scob_{\beta+\gamma}$  to the object $A_{\beta,\gamma}$, see (\ref{eq_obj_frob}), and the Frobenius structure of $A_{\beta+\gamma}$ to that of $A_{\beta,\gamma}$. In particular, the handle endomorphism of $A_{\beta+\gamma}$ must go to that of $A_{\beta,\gamma}$.
 \end{prop}

 \begin{proof}
 The handle polynomial of $A_{\beta+\gamma}$ is $U_{\beta+\gamma}(x)$, while that of $A_{\beta,\gamma}$ is $U_{\beta,\gamma}(x)$. For the  functor to exist, one needs $U_{\beta,\gamma}(x_0)=0$,
 where $x_0$ is the handle endomorphism of $A_{\beta+\gamma}$.
 \end{proof}

If $U_{\beta+\gamma}(x)$ is a proper divisor of $U_{\beta,\gamma}(x)$, then the handle endomorphism of $A_{\beta+\gamma}$ satisfies a stronger relation than that of the handle endomorphism of $A_{\beta,\gamma}$, and such a functor cannot be set up.

Note that the category $\scob_{\beta+\gamma}$ in (\ref{eq_want_func}) is not additive, while the target category  is additive. To remedy that, one can first pass to finite additive closures of these categories to get an additive functor
\begin{equation}\label{eq_want_func_add}
     F_{\beta,\gamma}^{\oplus} \ : \ \scob_{\beta+\gamma}^{\oplus} \lra \scob_{\beta}^{\oplus}\boxtimes \scob_{\gamma}^{\oplus}\cong \scob_{\beta}\boxtimes \scob{\gamma}.
 \end{equation}

  \begin{prop}\label{prop_equal}  If
  at least one of the fractions in (\ref{eq_Z_two}) is proper and $Q_{\beta}(T),Q_{\gamma}(T)$ are relatively prime then the pair  $(\beta,\gamma)$ is regular, so that  $U_{\beta+\gamma}(x)=U_{\beta,\gamma}(x)$.
 \end{prop}

 \begin{proof}
  A fraction $P(T)/Q(T)$ is proper if $\deg P(T)< \deg Q(T)$. This is equivalent to the condition that the handle polynomial $U(x)$ for this rational series is not divisible by $x$, that  is, $U(0)\not= 0$. Proposition follows by considering partial fraction decompositions for series $Z_{\beta}(T)$ and $Z_{\gamma}(T)$.
 Then
 \begin{equation}
 U_{\beta,\gamma}(x)=\mathrm{lcm}(U_{\beta}(x),U_{\gamma}(x))= U_{\beta}(x)U_{\gamma}(x)=U_{\beta+\gamma}(x).
 \end{equation}
 \end{proof}

\begin{remark} The implication in the above proposition goes only one way, as one can see by taking $\beta=\gamma$ when $\nchar \kk\not=2$. The pair $(\beta,\beta)$ is regular then.
\end{remark}

Proposition~\ref{prop_equal} gives a  sufficient condition for the functor $F_{\beta,\gamma}^S$ in  (\ref{eq_want_func}) to exist.

 \vspace{0.1in}

 It will follow from Proposition~\ref{prop_upon} that for any regular  pair $(\beta,\gamma)$ functors $F_{\beta,\gamma}^S$ and $F_{\beta,\gamma}^{\oplus}$ are fully faithful.

 \begin{prop} For a regular $(\beta,\gamma)$, functor $F^S_{\beta,\gamma}$, see  (\ref{eq_want_func}), induces
 a fully faithful functor
 \begin{equation}\label{eq_ff_functor}
     \Cob_{\beta+\gamma} \lra \Cob_{\beta}\boxtimes \Cob{\gamma}
 \end{equation}
 \end{prop}

 \begin{proof} Observe that the category $\Cob_{\beta}\boxtimes \Cob{\gamma}$ is nondegenerate. Indeed let $X\boxtimes Y$
 and $Z\boxtimes T$ be some objects of $\Cob_{\beta}\boxtimes \Cob{\gamma}$. The trace form $\Hom(X\boxtimes Y, Z\boxtimes T)\times \Hom(Z\boxtimes T, X\boxtimes Y)\to \kk$ is the
 tensor product of the trace forms $\Hom(X,Z)\times \Hom(Z,X)\to \kk$ and $\Hom(Y,T)\times \Hom(T,Y)\to \kk$.
 Since the tensor product of non-degenerate pairings is
 non-degenerate, we see that $\Hom(X\boxtimes Y,Z\boxtimes T)$
 has no nonzero negligible morphisms and the result follows.

 The functor $F^S_{\beta,\gamma}$ induces a full functor
 $\scob_{\beta+\gamma} \lra \Cob_{\beta}\boxtimes \Cob{\gamma}$. Since the category
 $\Cob_{\beta}\boxtimes \Cob{\gamma}$ is nondegenerate,
 this functor factors through $\Cob_{\beta+\gamma}$ and gives rise to the fully faithful functor in (\ref{eq_ff_functor}) .

 \end{proof}

\vspace{0.1in}

Passing to additive Karoubi envelopes results in an equivalence of categories:

\begin{proposition} \label{prop_upon} For a regular pair $(\beta,\gamma)$ and  upon passing to additive Karoubi envelopes, functor $F^S_{\beta,\gamma}$ induces an  equivalence of tensor categories
\begin{equation}
F^D_{\beta,\gamma}\ :\ \dcob_{\beta+\gamma} \simeq \dcob_\beta \boxtimes \dcob_\gamma.
\end{equation}
Further passage to gligible quotients produces a tensor equivalence
\begin{equation}
\underline{F}_{\beta,\gamma}\ :\ \udcob_{\beta+\gamma} \simeq \udcob_\beta \boxtimes \udcob_\gamma.
\end{equation}
\end{proposition}
Note that the tensor products $\boxtimes$ above are still the  naive tensor products of additive $\kk$-linear tensor  categories.

\begin{proof} The category $\scob_\beta \boxtimes \scob_\gamma$ has a commutative Frobenius algebra $A_{\beta,\gamma}$, see (\ref{eq_obj_frob}), so there is a tensor functor
$\scobal \to \scob_\beta \boxtimes \scob_\gamma$ sending $A_\alpha$ to $A_{\beta+\gamma}$ by the universal property from Section~\ref{subsec_universal}, with $\alpha=\beta+\gamma$.

Let $U(T)$ be the polynomial representing the handle skein relation in the category $\scobal$.
We have a factorization $U(T)=U_1(T)U_2(T)$ corresponding to factorization of the denominator of $Z(T)$
into product of the denominators of $Z_1(T)$ and $Z_2(T)$ (in particular we assume that $U_1(T)$ is not
divisible by $T$). Thus we have a corresponding idempotent $e\in \Hom(\be, A_\alpha)$ and decomposition $A_\alpha =\phi(e)A_\alpha \oplus \phi(1-e)A_\alpha$ where the generating functions
of the algebras $\phi(e)A_\alpha$ and $\phi(1-e)A_\alpha$ are precisely $Z_1(T)$ and $Z_2(T)$.

By the universal property there are tensor functors $\scob_\beta \to \dcobal$ and $\scob_\gamma \to \dcobal$ from the skein  categories to the Deligne category (additive Karoubi closure of the skein category) for $\alpha$ sending $A_\beta$ to $\phi(e)A_\alpha$ and $A_\gamma$ to $\phi(1-e)A_\alpha$.
Thus by the universal property of the external tensor product, see e.g. \cite[2.2]{O}, there is a tensor
functor $\scob_\beta \boxtimes \scob_\gamma \to \dcobal$ sending $A_\beta \boxtimes \be$ to $\phi(e)A_\alpha$ and $\be \boxtimes A_\gamma$ to $\phi(1-e)A_\alpha$.

Passing to the additive Karoubi closure of the source category gives a tensor functor
\[ F  \ : \ \dcob_\beta \boxtimes \dcob_\gamma \lra \dcobal .
\]
The composition of the above tensor functors
\begin{equation*}
\dcobal \stackrel{F^D_{\beta,\gamma}}{\lra} \dcob_\beta \boxtimes \dcob_\gamma \stackrel{F}{\lra} \dcobal
\end{equation*}
sends $A_\alpha$ to itself and thus
is isomorphic to the identity functor. Similarly, the composition
\begin{equation*}
\dcob_\beta \boxtimes \dcob_\gamma \stackrel{F}{\lra} \dcobal \stackrel{F^D_{\beta,\gamma}}{\lra}  \dcob_\beta \boxtimes \dcob_\gamma
\end{equation*}
sends $A_\beta \boxtimes \be$ and
$\be \boxtimes A_\gamma$ to themselves and thus is also isomorphic to the identity functor.

These functors intertwine the Frobenius structures of $A_{\alpha}$ and $A_{\beta,\gamma}$, so the isomorphisms are that of tensor functors.
This completes the proof.
\end{proof}

\begin{remark} A similar argument can be applied in a slightly more general situation where the polynomial relation
$U(x)=0$ is replaced by the polynomial relation $\widetilde U(x)=0$, with $U(t)$ a factor of the polynomial $\widetilde U(t)$. This allows to generalize the skein category $\scobal$ to a category $\widetilde{S}\cobal$ that maps onto $\scobal$. Notice that such a  skein relation is still compatible with evaluation $\alpha$.

Considering these categories $\widetilde{S}\cobal$ with  handle skein relations of a fixed degree $n$ gives a family of tensor categories that depend on $2n$ parameters, that is, the coefficients of the polynomial $\widetilde{U}(T)=T^n+\mathrm{l.o.t}$ and the evaluations $\alpha_0=\alpha(1),\dots, \alpha_{n-1}=\alpha(x^{n-1})$. This is a flat family of tensor categories, in  a suitable sense.
\end{remark}

\vspace{0.1in}

Starting with $Z_{\alpha}(T)$ as in (\ref{eq_Z_rat_9}), let us extract the polynomial term by writing
\begin{equation}
    Z_{\alpha}(T) = \frac{P(T)}{Q(T)}=  \frac{\overline{P}(T)}{Q(T)} + R(T), \ \  R(T)\in \kk[T], \ \  \deg \overline{P}(T) <\deg Q(T),
\end{equation}
so that $\overline{P}(T)/Q(T)$ is a proper fraction.
Factor $Q(T)$ over $\kk$ into
\begin{equation}
    Q(T) = Q_1(T)\dots Q_{\ell}(T),
\end{equation}
where each factor is a power of an irreducible polynomial over $\kk$, the factors are mutually coprime, $(Q_i(T),Q_j(T))=1$ for $i\not=j$ and $Q_i(0)=1$ for all $i$. Each $Q_i(T)$ is a power of an irreducible polynomial over $\kk$, with distinct polynomials for different $i$. Now form the partial fraction decomposition
\begin{equation}\label{eq_pf_decomp}
    Z_{\alpha}(T) = \sum_{i=1}^{\ell} \frac{P_i(T)}{Q_i(T)} + R(T), \ \ \deg P_i(T) < \deg Q_i (T),
\end{equation}
with $Q_i(T)$ and $R(T)$ as above. Denote by $\alpha[i]$ the sequence associated to the rational function $P_i(T)/Q_i(T)$, so that $Z_{\alpha[i]}(T)=P_i(T)/Q_i(T)$,  and by $\alpha[0]$ the sequence of coefficients of $R(T)$, so that $Z_{\alpha[0]}(T)=R(T)$.
There is a functor
\begin{equation}
F_{\alpha}^S\ : \ \scobal \lra \scob_{\alpha[0]} \boxtimes \scob_{\alpha[1]} \boxtimes \dots \boxtimes  \scob_{\alpha[\ell]}  = \boxtimes_{i=0}^{\ell}\, \scob_{\alpha[i]}
\end{equation}
taking the circle object $A_{\alpha}$ to the direct sum of objects  $\be^{\otimes (i-1)}\otimes A_{\alpha[i]}\otimes \be^{\otimes (\ell-i)}$, for $i=0,\dots, \ell$.

This functor induces an additive functor with the same target category  from the additive  closure of the source category:
\begin{equation}
F_{\alpha}^{\oplus}\ : \  \scobal \lra \boxtimes_{i=0}^{\ell}\, \scob^{\oplus}_{\alpha[i]}\cong \boxtimes_{i=0}^{\ell}\, \scob_{\alpha[i]},
\end{equation}
as well as functors
\begin{eqnarray}
\label{eq_eq_FD1}
F_{\alpha}^D & : &  \dcobal\lra \boxtimes_{i=0}^{\ell}\, \dcob_{\alpha[i]} \\
\label{eq_eq_FD2}
\underline{F}_{\,\alpha} & : &  \udcobal\lra \boxtimes_{i=0}^{\ell}\, \udcob_{\alpha[i]}
\end{eqnarray}

\begin{prop} Functors $F_{\alpha}^D$ and $\underline{F}_{\,\alpha}$ are equivalences of categories for any rational $\alpha$ over a field $\kk$.
\end{prop}

\begin{proof} This follows by iteratively applying the previous proposition.
\end{proof}

\section{Abelian realizations}\label{sec_abel_real}

Let $\alpha=\{\alpha_i, i\in \BZ_{\ge 0}\}$ be a sequence of elements of $\kk$. We say that a
Frobenius algebra $A$ in a symmetric monoidal category $\cC$ defined over some field extension $L$ of $\kk$ (not necessarily a  finite extension) is a {\em realization} of $\alpha$ if the evaluation
of $A$  is $\alpha$, see~\cite{Kh1}. Category  $\cC$ is  then an $L$-linear category. We say that the realization of $\alpha$ is {\em finite} if the $\Hom$ spaces in
$\cC$ are finite dimensional over $L$. Sequence $\alpha$ is called \emph{recognizable} if it admits a finite  realization.  The following result closely mirrors the one in \cite{Kh1}.

\begin{theorem} \label{finite realization thm}
A sequence $\alpha$ admits a finite realization if and only if it is  recurrent.
\end{theorem}

\begin{proof} Let $x=m\circ \Delta$ be the handle endomorphism of $A$. Note that $\tr(x^n)=\alpha_{n}\in \kk$ for $n\in \BZ_{\ge 0}$.
If $\Hom_\cC(A,A)$ is finite dimensional over $L$ then there exists a nonzero polynomial
$U(X)\in L[X]$ such that $U(x)=0$. This implies $x^iU(x)=0$ for any $i\in \BZ_{\ge 0}$. Computing the traces
of all terms in this relation we get a recurrent relation with constant coefficients in $L$ satisfied by $\alpha_i$
for $i \gg 0$. Eventual recurrence property can be written as vanishing of suitable  Hankel  determinants, see references in~\cite{Kh1}, and computing a determinant made of $\alpha_i$'s gives the same answer in $L$ and $\kk$.  Thus, the sequence $\alpha$ is  recurrent over $\kk$.

Conversely, assume $\alpha$ is recurrent. Then the category $\cC=\scobal$ or its additive Karoubi closure  $\dcobal$, see diagram (\ref{eq_seq_cd_1}),  with Frobenius object
$A$ given by one circle is a finite realization of $\alpha$.
\end{proof}

We say that a realization of $\alpha$ is {\em abelian} if $\cC$ is a symmetric abelian tensor category in the sense
of \cite[4.1.1]{EGNO} (the hom spaces are finite-dimensional, all objects have finite length, $\End_{\cC}(\be)=L$ and $\cC$ is rigid).

In particular, an abelian realization is finite in the above sense.

\begin{theorem} \label{thm_semi_real}
A sequence $\alpha$ admits an abelian realization if and only if the category $\udcobal$
is semisimple.
\end{theorem}

\begin{proof} Assume that the category $\udcobal$ is semisimple. Then the object $A\in \udcobal$ gives
an abelian realization of $\alpha$.
Conversely, the existence of an abelian realization implies that the quotient of $\dcobal$ by the negligible morphisms is semisimple, see~\cite[Theorem 1]{AK}.
\end{proof}

\begin{remark} \label{remark_field_ext}
(i) The above theorem shows that if a sequence $\alpha$ admits an abelian realization
over a field extension $L\supset \kk$ then it also admits an abelian realization over $\kk$.

(ii) Assume that a sequence $\alpha$ admits an abelian realization over field $\kk$ and let $L\supset \kk$ be a finite separable extension of $\kk$. Then
the sequence $\alpha$ admits an abelian realization over $L$. This follows from the construction of scalars extension of a tensor category, see \cite[5.3]{D2}. It is not clear, though, what can happen when $L\supset \kk$ is inseparable.
\end{remark}

The following result gives necessary conditions for a sequence $\alpha$ to admit an abelian realizations
in terms of its generating function $Z(T)=Z_{\alpha}(T)$, see (\ref{eq_Z_rat}).
Below, in Theorem~\ref{thm_sufficient}, it is shown that these conditions are also sufficient.

\begin{theorem} \label{necess_thm}
Assume that a sequence $\alpha$ admits an abelian realization. Then

(1) The generating function $Z(T)$ is rational, so $Z(T)=\frac{P(T)}{Q(T)}$ where $P(T), Q(T)\in \kk [T]$
are relatively prime.

(2) The denominator $Q(T)$ is separable, i.e., it has no multiple roots in an algebraic closure of $\kk$.

(3) $\deg P(T)\le \deg Q(T)+1$.

(4) Assume that {\em $\nchar \kk =p>0$}. Then all the residues of the form $Z(T)\frac{dT}{T^2}$ (computed over the algebraic closure $\overline{\kk}$)  lie in the prime subfield
${\mathbb F}_p\subset \kk$.
\end{theorem}

Note that conditions (2) and (3) say that the form $Z(T)\frac{dT}{T^2}$ has no poles of order $\ge 2$ (including at
$T=\infty$) except, possibly, the point $T=0$.

\begin{proof}
Statement (1) is implied by Theorem \ref{finite realization thm} as any abelian realization is finite.

To prove (2) let us consider the morphism $x=m\circ \Delta \in \End_\cC(A)$ as in the proof
of Theorem \ref{finite realization thm}. Let $p\in \kk [X]$ be a nonzero polynomial such that $p(x)=0$
(this polynomial exists since the $\Hom$ spaces in the category $\cC$ are finite dimensional). Let
$p_0$ be the product of all irreducible factors of $p$, each appearing with multiplicity 1. Then for any
$i\in \BZ_{\ge 0}$ the endomorphism
$x^ip_0(x)$ is a nilpotent element of $\End_\cC(A)$ as some power of $p_0$ is divisible by $p$.
Thus we have $\tr(x^ip_0(x))=0$ which gives a linear recurrent relation with constant coefficients
for $\alpha_n$ with $n$ sufficiently large. This relation implies that the generating function $Z(T)$ can
be written as a fraction with denominator $p_0(T)$. Thus we proved that the factorization of
$Q(T)$ into irreducible factors is square free.

A  related property is that in a rigid abelian $\kk$-linear tensor category the trace of a nilpotent endomorphism  is  zero, see e.g. \cite[Corollaire 3.6]{D1}.

\vspace{0.1in}

We still have to show that each irreducible factor of $Q(T)$ is separable in the case $\nchar \kk =p>0$. Observe that the tensor power $A^{\otimes p}\in \cC$ is a commutative Frobenius algebra object that gives an abelian realization
of the sequence $\mathrm{Fr}(\alpha)=\{ \alpha_0^p, \alpha_1^p, \ldots \}$, see Example~\ref{example_2}. The generating function
of the sequence Fr$(\alpha)$ is Fr$(Z(T))$ where Fr$(Z(T))$ is obtained from $Z(T)$ by applying
the Frobenius endomorphism $\lambda \mapsto \lambda^p$ to all the coefficients of $Z(T)$.
Now for the sake of contradiction assume that one of the
irreducible factors of $Q(T)$ is not separable. We recall that a nonseparable irreducible polynomial
is of the form $r(T)=\sum_{i=0}^kc_iT^{pi}$. Thus, one of the factors of Fr$(Q(T))$ is
Fr$(r(T))=\sum_{i=0}^kc_i^pT^{pi}=(\sum_{i=0}^kc_iT^{i})^p$ and Fr$(Q(T))$ is not square free.
This is a contradiction (note that the polynomials Fr$(P(T))$ and Fr$(Q(T))$ are relatively prime).
Thus (2) is proved.

Recall that the relation $\tr(x^ip_0(x))=0$ holds for any $i\in \BZ_{\ge 0}$, where the polynomial
$p_0=p_0(T)$ is square free. In particular the multiplicity of factor $T$ in $p_0(T)$ is $\le 1$. It follows
that the sequence $\alpha_2, \alpha_3, \ldots$ satisfies a linear recurrent relation with constant
coefficients, which implies (3).

Let us prove (4).
We can assume that the denominator $Q(T)=\prod_{i=1}^r(1-\gamma_iT)$ for distinct nonzero constants $\gamma_i\in L$, where $L$ is a finite separable field extension of $\kk$.
Let $Z(T)=\frac{\beta_i}{1-\gamma_iT}+Z'$ where $\beta_i\in L$ and $Z'$ has no poles at $T={\gamma_i}^{-1}$, $1\le i\le r$. Let us consider an abelian realization of $\alpha$
over $L$, see Remark \ref{remark_field_ext}(2). For a suitable idempotent $e$ the algebra $\phi(e)A$ will have the generating function $\frac{\beta_i}{1-\gamma_iT}$, see Section \ref{subsec_direct_sums}.
Thus $\beta_i\gamma_i=\dim(\phi(e)A)$ must be an element of the prime subfield
${\mathbb F}_p\subset L$, see \cite[Lemma 2.2]{EHO}. Observe that $-\beta_i\gamma_i$ is precisely the residue of the 1-form $Z(T)\frac{dT}{T^2}$ at
$T={\gamma_i}^{-1}$. Thus the statement (4) is proved for all finite nonzero poles of $Z(T)$. The residue at $T=0$ is $\alpha_1=\dim(A)$, and we can
apply \cite[Lemma 2.2]{EHO} again. Finally in the remaining case $T=\infty$ we use the Residue Theorem, which  holds in characteristic  $p$ as well~\cite[Corollary 2.5.4]{G}.
\end{proof}

\begin{example}
Sequence $\alpha =(1,2,3,4,5,\ldots)$ describing the function $Z(T)=1/(1-T)^2$  does not admit an abelian realization over any field, see condition (2) of the above theorem. More explicitly, the handle endomorphism $x$ satisfies
$(x-1)^2=0$ in $\Cobal$ for this $\alpha$. However,  $\tr(x-1)=\alpha_2-\alpha_1=3-2=1\ne 0$. This is a contradiction:
in any abelian category trace of a  nilpotent endomorphism is zero~\cite[Proposition~4.7.5]{EGNO}.
\end{example}

If $\kk$ is algebraically closed,  decomposition (\ref{eq_pf_decomp}) can be refined to
\begin{equation}\label{eq_pf_decomp_2}
    Z_{\alpha}(T) = \sum_{i=1}^{\ell} \frac{P_i(T)}{(1-\gamma_i T)^{m_i}} + R(T), \ \ \deg P_i(T) < m_i,
\end{equation}
with distinct $\gamma_1,\dots, \gamma_{\ell}\in \kk$ and polynomial $R(T)$. Conditions (2)-(4) of Theorem~\ref{necess_thm} translate to
\begin{itemize}
\item $m_i =1$ for all $i$, $1\le i\le \ell$.
\item $\deg R(T)\le 1$, that is, the polynomial $R(T)$ is at most linear, $R(T)=r_0+ r_1 T$, and $r_1\in \mathbb{F}_p$ if $\nchar \kk=p$.
\item Due to $m_i=1$ we restrict to a constant polynomial $P_i(T)=p_i\in\kk$, simplifying the residue to
\begin{equation}
\mathrm{res}_{\gamma^{-1}_i}\left(\frac{p_i dT}{(1-\gamma_i T)T^2}\right) =- p_i\gamma_i.
\end{equation}
Thus, condition (4) can be rewritten as $p_i\gamma_i\in {\mathbb F}_p$.
\end{itemize}

Summarizing, over an algebraically closed $\kk$, a rational function $Z(T)$ admits an abelian realization iff
\begin{equation}\label{eq_pf_decomp_3}
    Z_{\alpha}(T) = \sum_{i=1}^{\ell} \frac{p_i}{1-\gamma_i T} + r_0+r_1T, \ \ p_i\gamma_i\in {\mathbb F}_p, \  r_1\in {\mathbb F}_p,
\end{equation}
for distinct $\gamma_1,\dots, \gamma_{\ell}$.

In characteristic $0$, one can  use Remark~\ref{remark_field_ext}(i) to pass from $\kk$ to its algebraic closure. In characteristic $p$, irreducible inseparable factors in the denominator also constitute an obstruction to existence of an abelian realization, by condition (2) of Theorem~\ref{necess_thm}.
\begin{theorem} \label{thm_sufficient}
Assume that a sequence $\alpha$ satisfies  conditions (1)--(4) from Theorem
\ref{necess_thm}. Then $\alpha$ admits an abelian realization.
\end{theorem}

\begin{proof} By Remark \ref{remark_field_ext} (i) we can and will assume that the field $\kk$ is algebraically closed.

We start by giving abelian realizations for some special sequences.

(1) Assume $Z(T)=\alpha_0+\alpha_1T$ and $\nchar \kk =p>0$ with $\alpha_1\in {\mathbb F}_p$.
Then we can choose $\cC =\mbox{Vec}_\kk$ and use Example \ref{ortho example} with vector space $V$ of suitable dimension.

(2) Assume $Z(T)=\alpha_0+\alpha_1T$ and $\nchar \kk =0$. Again we use Example \ref{ortho example}; however in all cases when $\dim(V)\not \in \BZ_{\ge 0}$ we use the abelian specialization of the  Deligne category
$\cC =\Rep(O_t)$ (see e.g.~\cite[9]{D1}) with $t=\dim(V)$.

(3) Assume $Z(T)=\frac{\beta}{1-\gamma T}$ with $\beta \gamma =1$.
We choose $\cC =\mbox{Vec}_\kk$ and $A=\kk$ such that $\epsilon(1)=\beta$.

(4) Assume $Z(T)=\frac{\beta}{1-\gamma T}$ and $\nchar \kk =p>0$, with
$\beta \gamma \in {\mathbb F}_p\setminus \{ 0\}$. We choose $\cC =\mbox{Vec}_\kk$ and take $A$
to be a direct sum of several copies of the algebra from (3).

(5) Assume $Z(T)=\frac{\beta}{1-\gamma T}$ and $\nchar \kk =0$, with
$t=\beta \gamma \ne 0$. We take $A$ to be the standard Frobenius algebra in the semisimple quotient
of the Deligne category $\cC =\Rep(S_t)$, see~\cite[Th\'eor\`emes 2.18, 6.2]{D1}.

Any sequence $\alpha$ satisfying the conditions (1)--(4) from Theorem
\ref{necess_thm} is a sum of sequences considered in (1), (2), (4), (5) above. Thus the following
result completes the proof of the Theorem.
\end{proof}

The last two theorems together are equivalent to the following result.

\begin{theorem} \label{thm_both}
A sequence $\alpha$ over a field $\kk$ admits an abelian realization if and only if it satisfies  conditions (1)--(4) in Theorem
\ref{necess_thm}.
\end{theorem}

\begin{lemma} Assume that sequences $\alpha'$ and $\alpha''$ admit abelian realizations over an algebraically closed field $\kk$. Then the sequence $\alpha'+\alpha''$ also admits an abelian realization
over $\kk$.
\end{lemma}

\begin{proof} By Theorem~\ref{thm_semi_real} there are semisimple categories $\cC'$ and $\cC''$ with Frobenius algebras $A'\in \cC'$ and $A''\in \cC''$ giving the realizations of $\alpha'$ and $\alpha''$. Then the algebra
\begin{equation*}
A'\boxtimes \be \oplus \be \boxtimes A''\in \cC'\boxtimes \cC''
\end{equation*}
gives a realization of $\alpha$ in a semisimple (and hence abelian) category $\cC'\boxtimes \cC''$.
\end{proof}

\begin{remark} The proof of Theorem \ref{thm_sufficient} shows that in the case of algebraically closed field $\kk$ of positive characteristic the sequence $\alpha$ admits an abelian realization if and only if it admits a realization with $\cC =\mbox{Vec}_\kk$.
\end{remark}

\begin{example}
(a) Let $\alpha =(1,1,2,3,5,8,\ldots)$ be the Fibonacci sequence, with the generating function $Z(T)=1/(1-T-T^2)$. Then $\alpha$ admits
an abelian realization in characteristic zero.
It admits an abelian realization in positive characteristic $p$ if $p\ne 5$ and $5$ is
a quadratic residue modulo $p$, i.e., $p=2$ or $p\equiv \pm 1 (\mathrm{mod}\ 5)$. Indeed, in characteristic $5$ the denominator is $(1+2T)^2$, hence has a multiple root and is not separable. If $5$ is not a quadratic residue modulo $p$, the differential form residue of  condition $(4)$ in either of the two roots of the denominator does not lie in the prime subfield.

(b) Let $\beta =(-1,2,1,3,4,7,11,\ldots)$ be the (shifted) Lucas sequence. It satisfies the Fibonacci relation $\beta_{n+2}=\beta_{n+1}+\beta_n$ for $n\ge 0$ but with  a
different initial condition. The generating
function $Z(T)=\frac{\phi^{-1}}{1-\phi T}+\frac{\bar \phi^{-1}}{1-\bar \phi T}$ where $\phi=\frac{1+\sqrt{5}}2$
is the golden ratio and $\bar \phi=\frac{1-\sqrt{5}}2$  its
Galois conjugate. Both residues of the one-form in
Theorem \ref{necess_thm} (4) equal 1 in this case.
Thus, $\beta$ admits an abelian realization in any characteristic.
\end{example}

\vspace{0.1in}

\section{Endomorphisms of object 1} \label{sec_end_one}

\subsection{Algebras  \texorpdfstring{$B_{S}$}{B S} and  \texorpdfstring{$B$}{B}.}
$\quad$
\vspace{0.1in}

Define the \emph{rank} $K$ of a rational theory $\alpha$ by  formula (\ref{eq_NMK}). Rank  is the maximum of the degree of the numerator $P(T)$ of $Z(T)$ plus one and the  degree of the denominator $Q(T)$. For a theory of rank $K$, elements $1,x,\dots, x^{K-1}$ of  $A_{\alpha}(1)$ are linearly independent and there is a  linear relation (\ref{eq_with_K}), reproduced below
\begin{equation}\label{eq_with_K_5}
    U_{\alpha}(x):= x^K - b_1 \, x^{K-1}+ b_2 \, x^{K-2}- \ldots + (-1)^M b_M \, x^{K-M} =0,
 \end{equation}
 where $b_i$'s are the coefficients of the  denominator $Q(T)$, see (\ref{eq_K}), normalized so that $Q(0)=1$, and  $p_{\alpha}(x)=x^K Q(x^{-1})$.

 Recall that category  $\Cobal$ is the quotient of the skein category $\scobal$ by the ideal of negligible morphisms. In the category $\scobal$ we evaluate closed components via $\alpha$ and reduce $K$ handles on  a connected component via (\ref{eq_with_K_2}). In $\Cobal$ we further mod out by all negligible morphisms.

 Consider the  endomorphism algebras
 \begin{eqnarray}
     B_{S} & := & \End_{\scobal}(1)\cong \Hom_{\scobal}(0,2), \\
     B & = & \End_{\Cobal}(1)\cong  A_{\alpha}(2).
 \end{eqnarray}
Both algebras are commutative unital $\kk$-algebras, under the pants cobordism multiplication. Isomorphisms on the right are  those of $\kk$-vector spaces, given by
moving the bottom circle of a $(1,1)$-cobordism to the top. Algebra $B$ is the quotient of $B_{S}$ by the two-sided ideal $J_{neg}$ of negligible endomorphisms,
\begin{equation}\label{eq_B_quotient}
    B \cong B_{S}/J_{neg}.
\end{equation}

Elements $u, x$ in Figure~\ref{fig4_1} generate the algebra $B_{S}$.

\begin{figure}[!htb]
\begin{center}
\begin{tikzpicture}
[scale=0.4]
\draw [yscale=0.3] (0,0) circle(1);
\draw [dashed, yscale=0.3] (1,-8) arc(0:180:1);
\draw [yscale=0.3] (-1,-8) arc(180:360:1);
\draw (-1,0) to [out=270, in=180] (0,-1) to [out=0, in=270] (1,0);
\draw (-1,-2.4) to [out=90, in=180] (0,-1.4) to [out=0, in=90] (1,-2.4);
\node at (-2,-1.2) {$u=$};
\end{tikzpicture}
$ \ \ \ $
\begin{tikzpicture}
[scale=0.4]
\draw [yscale=0.3] (0,0) circle(1);
\draw [dashed, yscale=0.3] (1,-8) arc(0:180:1);
\draw [yscale=0.3] (-1,-8) arc(180:360:1);
\draw (-1,0) -- (-1,-2.4);
\draw (1,0) -- (1,-2.4);
\draw [fill=black] (0,-1.2) circle(0.05);
\node at (-2,-1.2) {$x=$};
\node at (2,-1.2) {$:=$};
\draw [yscale=0.3] (5,0) circle(1);
\draw [dashed, yscale=0.3] (6,-8) arc(0:180:1);
\draw [yscale=0.3] (4,-8) arc(180:360:1);
\draw (4,0) to[out=270, in=90] (3,-1.2) to [out=270, in=90] (4,-2.4);
\draw (6,0) to[out=270, in=90] (7,-1.2) to [out=270, in=90] (6,-2.4);
\draw (4.,-1.2) to [out=290, in=250] (6.,-1.2);
\draw (4.3,-1.5) to [out=70, in=110] (5.7,-1.5);
\end{tikzpicture}
\caption{\label{fig4_1} Generators $u$ and $x$ of $B_{S}$.}
\end{center}
\end{figure}
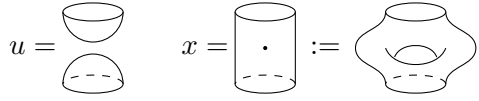

It is easy to write down a basis in each hom
space  of the category $\scobal$, see~\cite{KS}. A basis in $B=\End_{\scobal}(1)$ is given by the set of tube cobordisms with at most $K-1$ dots on them and the cup-cap  cobordisms $u$ decorated by at most $K-1$ dots on each connected component.

\begin{prop}  Elements
\begin{equation}
    x^n, \ 0\le n<K, \ \  x^{n}u x^{k}, \  0\le n,k < K
\end{equation}
constitute  a basis of $B_{S}$, and $\dim(B_S)=K^2+K.$
\end{prop}

\begin{prop} \label{AP_def_rel}
The following is a  set  of defining relations in $B_{S}$ on generators $u,x$:
\begin{eqnarray}
    u x^n u & = & \alpha_n u, \ \ n\ge 0, \\
    U_{\alpha}(x) & = & 0.
\end{eqnarray}
\end{prop}
Defining relations of the first type are  shown in Figure~\ref{fig4_2}.

Algebra $B_{S}$ has the bar anti-involution $a\longmapsto \overline{a}$ given by the identity  on the
generators,  $\overline{x}=x,\overline{u}=u$, and   $\overline{x^n u  x^k}=x^k u  x^n$.

The trace form on $B_{S}$ is defined by closing up  a $(1,1)$-cobordism and  evaluating it via $\alpha$. The  trace is given in the above basis by
\begin{equation*}
    \tr(x^n) = \alpha_{n+1},\ \  \tr(x^n u x^k ) = \alpha_{n+k}.
\end{equation*}

\begin{figure}[ht]
\begin{center}
\begin{tikzpicture}
[scale=0.5]
\draw (0,0) circle(1);
\node [scale=1] at (-3,0) {$ux^n u = $};
\draw[yscale=0.3] (-1,0) arc (180:360:1);
\draw[yscale=0.3, dashed] (-1,0) arc (180:0:1);
\draw [fill=black] (-0.15,0.6) circle(0.05);
\node at (0.3,0.6) {$n$};
\draw [yscale=0.3] (0,7) circle(1);
\draw (-1,2.1) to [out=270, in=180]  (0,1.3) to [out=0, in=270] (1,2.1);
\draw [dashed, yscale=0.3] (1,-7) arc(0:180:1);
\draw [yscale=0.3] (-1,-7) arc(180:360:1);
\draw (-1,-2.1) to [out=90, in=180]  (0,-1.3) to [out=0, in=90] (1,-2.1);
\node at (2.5,0) {$=\alpha_n$};

\draw [yscale=0.3] (4.5,4) circle(1);
\draw (3.5,1.2) to [out=270, in=180]  (4.5,0.4) to [out=0, in=270] (5.5,1.2);
\draw [dashed, yscale=0.3] (5.5,-4) arc(0:180:1);
\draw [yscale=0.3] (3.5,-4) arc(180:360:1);
\draw (3.5,-1.2) to [out=90, in=180]  (4.5,-0.4) to [out=0, in=90] (5.5,-1.2);
\node at (7,0) {$=\alpha_n u$};
\end{tikzpicture}
\caption{\label{fig4_2} One of the defining relations in $B_{S}$.}
\end{center}
\end{figure}
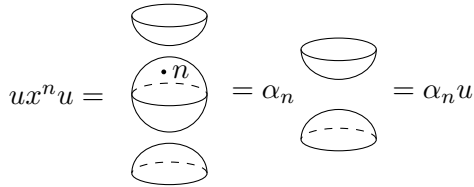

Recall the  hom  space
\begin{equation}\label{q_hom_many}
A(1)=\Hom_{\cobal}(0,1)\cong \Hom_{\scobal}(0,1) \cong \kk[x]/(U_{\alpha}(x))
\end{equation}
of dimension $K$,
with  the commutative algebra structure  given by the pants cobordism.

Denote  by $B(1)\subset B_{S}$ the subalgebra of $B_{S}$ generated  by the  handle endomorphism.
There is an algebra  isomorphism
\begin{equation}
    B(1) \cong A(1)
\end{equation}
given by taking the  handle endomorphism in $B(1)$  to the corresponding handle element of $A(1)$, see Figure~\ref{fig4_3}. The skein relation on powers of the  handle holds on  linear combinations of powers  of handle on a disk as  well as  on an annulus. For this reason  the algebras are isomorphic. The  geometric definitions of multiplications in the two algebras are slightly  different: in $A(1)$ it is given by the pants  cobordism, while in $B(1)$ and $B_{S}$ it is the composition of $(1,1)$-cobordisms.

\begin{figure}[ht]
\begin{center}
\begin{tikzpicture}
\node at (0,0) {$x=$};
\draw [yscale=0.3] (1,1.5) circle(0.25);
\draw (0.75,0.45) to [out=240, in=180] (1,-0.525) to [out=0, in=300] (1.25,0.45);
\draw (0.80,0) to [out=290, in=250] (1.20,0);
\draw (0.9,-0.05) to [out=60, in=120] (1.1,-0.05);
\node at (2.2,0) {$\in A(1)$};
\end{tikzpicture}
\begin{tikzpicture}
\node at (0,0) {$x=$};
\draw [yscale=0.3] (1,1.5) circle(0.25);
\draw [yscale=0.3, dashed]  (1.25,-1.5) arc(0:180:0.25);
\draw [yscale=0.3]  (1.25,-1.5) arc(0:-180:0.25);
\draw (0.75,0.45) to [out=240, in=120] (0.75,-0.45);
\draw (1.25,0.45) to [out=300, in=60] (1.25,-0.45);
\draw (0.80,0) to [out=290, in=250] (1.20,0);
\draw (0.9,-0.05) to [out=60, in=120] (1.1,-0.05);
\node at (2.6,0) {$\in B(1) \subset B_{S}$};
\end{tikzpicture}
\caption{\label{fig4_3} Element $x$ of $A(1)$ on the left  is  a one-holed torus (or a handle with one hole). Element $x$ of $B(1)\subset  B_{S}$ on the right is a  two-holed  torus (a handle with  two  holes). Both have  the same defining relation $U_{\alpha}(x)=0$ in  $A(1)$ and $B(1)$, respectively. Left $x$  is obtained from the right  $x$ by capping  off the bottom circle with a disk.}
\end{center}
\end{figure}

Relatedly, we maintain a slight abuse of notation, also shown  in Figure~\ref{fig4_3}, where $x$ is used to denote handle cobordisms with either one  or two boundary  components, respectively, generating algebras $A(1)$ and $B(1)\subset B_{S}$.

The two-sided ideal $(u)$ of $B_{S}$ is
\begin{equation}
(u) = BuB = B(1)u B(1)\cong B(1)\otimes_{\kk}B(1)^{op},
\end{equation}
where the second isomorphism is that of $B(1)$-bimodules.
It is spanned by cobordisms with two connected viewable components, and  there  is  a  split exact sequence of this 2-sided ideal of $B_{S}$ and the quotient  ring
\begin{equation}
    0 \longrightarrow B(1)u  B(1) \longrightarrow  B_{S} \myrightleftarrows{\rule{0.5cm}{0cm}}  B(1) \lra 0 .
\end{equation}
The quotient by the 2-sided ideal is spanned by powers  of  $x$  and  is naturally isomorphic to the  subring $B(1)$ spanning by cobordisms with one connected component, via the inclusion $B(1)\subset  B_{S}$, which is  a section  of the  surjection above.

\vspace{0.1in}

Recall that $B$ is the quotient of $B_{S}$ given by (\ref{eq_B_quotient}).
$B_{S}$ acts on the space $A(1)$ by left multiplication by cobordisms, see
isomorphisms (\ref{q_hom_many}) for equivalent descriptions  of that  space.  The action takes a cobordism from $0$ to $1$ (that can be assumed to be connected and of genus less than $K$) and composes with a cobordism from $1$ to $1$. Closed components that may result are removed via $\alpha$-evaluation and the genus is reduced to at most $K-1$. The action factors through that of  $B$, since negligible endomorphisms act by $0$.

Passing to gligible quotients results in a short exact sequence
\begin{equation} \label{eq_phi_4}
    0 \longrightarrow B(1)u  B(1) \stackrel{\phi}{\longrightarrow} B \lra  B/\mathrm{im}(\phi) \lra 0 ,
\end{equation}
of a two-sided ideal, algebra, and the quotient algebra. Map $\phi$ is injective, and the quotient $B/\mathrm{im}(\phi)$ is trivial iff $\alpha$ is multiplicative, that is, if the product map  $A(1)\otimes A(1)\lra A(2)$, which corresponds to $\phi$, is an isomorphism, also see~\cite{Kh1}.


\subsection{Examples}
$\quad$
\vspace{0.1in}

{\it  Example: $K=1$.} In  this  case a handle on a component reduces to  a multiple of the component without the handle and $Z(T)$ is either the constant function, $Z(T)=\alpha_0$, or
\begin{equation}
Z(T)=\frac{\alpha_0}{1-\gamma T}= \alpha_0+\alpha_0\gamma T +\alpha_0\gamma^2 T^2+\dots
\end{equation}
The ring $B_{S}$ has a basis $\{1,u\}$ with $u^2=\alpha_0 u$. The trace on $B_{S}$ is
$\tr(1)=\alpha_0\gamma, \tr(u)=\alpha_0$.

The only possible  functions in this case  are $Z_{\alpha}=\alpha_0, \alpha_0\not=0$  and $Z_{\alpha}=\frac{\alpha_0}{1-\gamma T}, \alpha_0,\gamma\not= 0. $

The quotient map $B_S\lra B$ is an isomorphism iff the Gram matrix
\begin{equation}\begin{pmatrix} \alpha^2_0 & \alpha_0 \\
\alpha_0 & \alpha_0\gamma \end{pmatrix}
\end{equation}
of the basis $\{1,u\}$ of $B_S$ is nondegenerate. It has determinant $\alpha_0^2(\alpha_0 \gamma-1)$. We see that the quotient map $B_S\lra B$ is an isomorphism (and the theory is not multiplicative) iff $\gamma\not=\alpha_0^{-1}$.

\vspace{0.1in}

{\it Example: Linear function.} Let $Z(T)=\beta_0+\beta_1 T$, $\beta_1\not=0$. Then $B_S$ has a spanning set $\{1,x,u,xu,ux\}$, which is a basis iff $\beta_1\not=2$, see Section~\ref{subsec_linear} and~\cite{Kh1}. Let us assume the latter case. The multiplication rules in $B_S$ follow from the relations
\begin{equation*}
uxu = \beta_1 u, \ \ xux = \beta_1 x, \ \
  u^2 = \beta_0 u, \ \ x^2=0 ,
 \end{equation*}
and the quotient of $B_S$ by the 2-sided ideal $BuB$ is one-dimensional. The action of $B$ on the left ideal $Bu=\kk u \oplus \kk xu \cong B(1) $ surjects $B$ onto the matrix algebra $\mmat_2(\kk)$ and leads to the direct product decomposition
\begin{equation}
    B \cong \mmat_2(\kk) \times \kk.
\end{equation}
It is given explicitly as follows. Let
$z=xu + ux - \beta_0 x-\beta_1$. Then $z^2=-\beta_1 z$ and $-\beta_1^{-1}z$ is a central idempotent splitting off $\kk$ from $B$. The complementary factor is given by the (non-unital) homomorphism $\mmat_2(\kk) \lra B$,
\begin{eqnarray*}
   & &  \begin{pmatrix}
      1 & 0 \\ 0 & 0
    \end{pmatrix} \lra \beta_1^{-1} ux , \ \
    \begin{pmatrix}
      0 & 1 \\ 0 & 0
    \end{pmatrix} \lra \beta_1^{-1}(u - \beta_0\beta_1^{-1} ux) , \\
    & & \begin{pmatrix}
      0 & 0 \\ 1 & 0
    \end{pmatrix} \lra x , \ \
    \begin{pmatrix}
      0 & 0 \\ 0 & 1
    \end{pmatrix} \lra \beta_1^{-1}(xu - \beta_0 x) .
\end{eqnarray*}

On the other hand, if $\beta_1=2$, the map $\phi$ in (\ref{eq_phi_4}) is an isomorphism, the theory is multiplicative ($A(2)\cong A(1)^{\otimes 2}$), and $B$ is isomorphic to the matrix algebra of size $2$ over $\kk$, see~\cite{Kh1}.

%
%

\section{Constant generating function \texorpdfstring{$\beta$}{b}}\label{sec_constant_gen}


\subsection{State spaces, partitions, and Catalan numbers}
\label{subsec_catalan}

Consider the  evaluation corresponding to the series
\begin{equation}\label{eq_const_gf}
    Z(T)=\beta, \ \beta\in \kk^{\ast},
\end{equation}
which  is just  the constant function, so the associated sequence  $\alpha=(\beta,0,0,\dots )$. Scaling invariance explained in Section~\ref{subsec-scaling} allows us to set  $\beta=1$ without "changing" any categories. We keep $\beta$ arbitrary, but this is just a matter of preference.

\begin{figure}[ht]
\begin{center}
\includegraphics[scale=1.0]{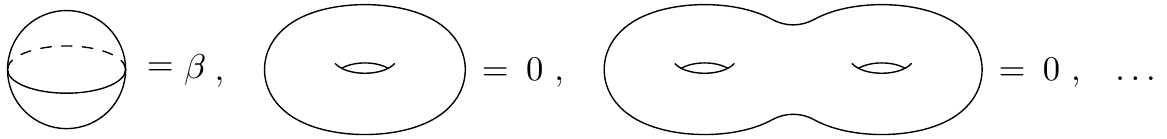}
\caption{\label{fig_constant_1} Constant series  evaluations.}
\end{center}
\end{figure}

For $\alpha$ describing a constant function, see Figure~\ref{fig_constant_1},
any closed surface $S$ which has a component of genus greater than zero evaluates to 0. Otherwise, $S$
(which is then necessarily the disjoint union of 2-spheres) evaluates to
\[
S \mapsto \alpha(S)= \beta^{\ \dim H_0(S, \kk)},
\]
that  is, $\beta$ to the power the number of components of $S$.

Define the genus of a connected surface with  boundary as the genus of a surface obtained by attaching 2-disks to all boundary components.

Consider the space state of $n$ circles $A(n)=A_{\alpha}(n)$ in this theory $\alpha$ as discussed in~\cite[Section 2.7]{Kh1}.  For an arbitrary $\alpha$ the state space is defined by the formula (\ref{eq_A_alpha}), also see~\cite{Kh1}. Recall from the latter reference that the circles are ordered and numbered by $1,2,...,n$; their union is denoted $\sqcup_n  \SS^1$.
All surfaces with a component of genus greater than zero are in the kernel of the bilinear form, so that  $A(n)$ is spanned  by diffeomorphism classes of
viewable surfaces $S$, with each component of genus 0, and the boundary diffeomorphic  to the disjoint union  of $n$ circles.

Such surfaces can be canonically identified with partitions of the set $\{1,2,...,n\}$.
Denote the set of partitions by $D(n)$.  Cardinality of $D(n)$ is known as the Bell number $B_n$ and it has the following generating function:
\[
\sum_{n\geq 0} \frac{B_n}{n!} t^n = \exp(\exp(t)-1)  .
\]
To a partition $\lambda\in D(n)$ there is associated  a viewable surface  $S_{\lambda}$ as above, with  each component of genus $0$. Recall that  by  \emph{viewable surface} we mean  a  surface without closed components.

Consider the vector space $\kk^{D(n)}$ with a basis of vectors $v_{\lambda}$, over all  partitions $\lambda\in  D(n)$.
Form the linear map
\begin{equation}\label{eq_DA_map}
    \kk^{D(n)} \lra A(n), \ \ v_{\lambda} \mapsto [S_{\lambda}]
\end{equation}
into the state space of $n$  circles which takes basis vectors to corresponding surfaces $S_{\lambda}$. This map is surjective, as follows from the discussion above.

For $n\leq 3$ the map (\ref{eq_DA_map}) is
an isomorphism, that is, the  induced bilinear form is nondegenerate on $\kk^{D(n)}$, see~\cite[Section 2.7]{Kh1}.

However, starting from $n=4$ map (\ref{eq_DA_map}) has a nontrivial kernel, and the dimension of the state space (if $\nchar \kk=0$) is equal to the Catalan number
\begin{equation}
c_n = \frac{1}{2n+1}\binom{2n}{n}. \label{catalan}
\end{equation}
\begin{center}
\begin{tabular}{ |c|c|c|c|c|c|c|c|c|c }
 \hline
 $n$ & 0 & 1 & 2 & 3 & 4 & 5 & 6 & 7 \\
 \hline
$B_n$ & 1 & 1 & 2 & 5 & 15 & 52 & 203 & 877  \\
 \hline
 $c_n=\dim A(n)$ & 1 & 1 & 2 & 5 & 14 & 42 & 132 & 429  \\
 \hline
\end{tabular}
\end{center}

\begin{theorem} \label{thm_catalan}
Over a field $\kk$ of characteristic zero
 the  state space $A(n)$ for the theory with the constant generating function  (\ref{eq_const_gf}) has dimension equal to the Catalan number $c_n$:
\[
\dim A(n) = c_n .
\]
It has a basis $\PS^n$ of crossingless surfaces, as described below in Section~\ref{subsec_plane}.
\end{theorem}

One proof of  this theorem is given starting here and through Section~\ref{subsec_meander}. Another proof is contained  in  Section~\ref{subset_osp}, via a connection to representations of  $\mathrm{osp}(1|2)$.

\vspace{0.1in}

Recall notations from~\cite{Kh1},  where  $y_{ij}$ denotes a surface that consists  of a tube connecting circles $i$ and $j$ and $n-2$ disks  that cap off the remaining circles.  More generally, for $1\le i_1<\dots i_r \le n$
denote by $y_J=y_{i_1,i_2,\dots, i_r}$ the surface  that consists of a 2-sphere  with $r$  holes bounding circles $i_1,i_2,\dots, i_r$ and
$n-r$ disks capping off the remaining  $n-r$ circles. Here $J=\{i_1,\dots, i_r\}$.  Figure~\ref{fig_3_3} shows examples of these surfaces for $n=3$ and Figure~\ref{fig_3_4} shows examples for larger $n$.  Note  also that  $A(n)$ is naturally a commutative associative unital algebra under the  multiplication given by composing two diagrams via the disjoint union of $n$ pants cobordisms.

\begin{figure}[ht]
\begin{center}
\includegraphics[scale=1.0]{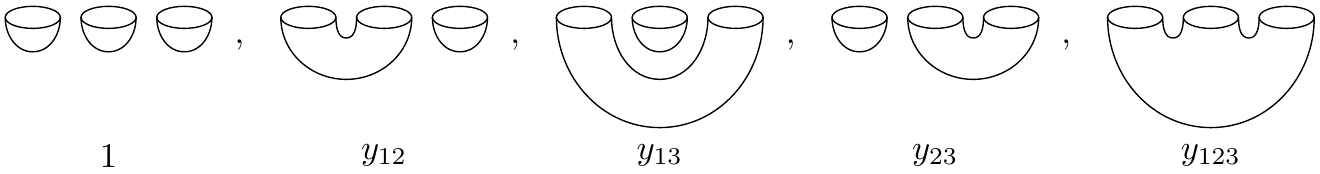}
\caption{A spanning set (in fact, a basis) of $A(3)$: the unit element, tubes $y_{ij}$, and connected surface $y_{123}=y_{12}y_{13}$.}
\label{fig_3_3}
\end{center}
\end{figure}

\begin{figure}[ht]
\begin{center}
\includegraphics[scale=1.0]{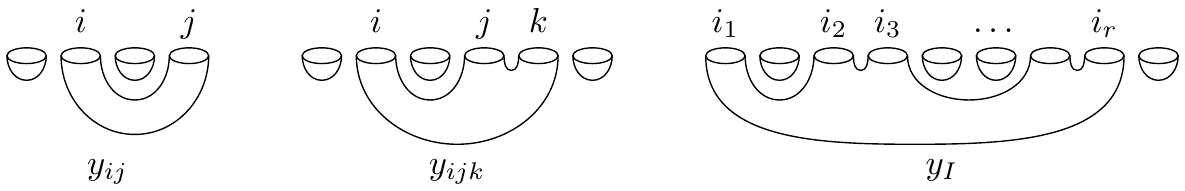}
\caption{Examples of surfaces $y_J$.}
\label{fig_3_4}
\end{center}
\end{figure}

Fifteen elements of this spanning set for $A(4)$ can be separated into five types, as follows and see Figure~\ref{fig_3_5}:
\begin{enumerate}
    \item Unit cobordism $1$.
    \item Six cobordisms $y_{ij}$, $i<j$.
    \item Three cobordisms $y_{ij}y_{kl}$ with $i,j,k,l$ distinct: $y_{12}y_{34},y_{13}y_{24},y_{14}y_{23},$ each a disjoint union of two tubes.
    \item Four cobordisms $y_{ijk}$, each a union of a 3-holed sphere and a disk.
    \item Cobordism $y_{1234}$, which is a 4-holed sphere.
\end{enumerate}

\begin{figure}[ht]
\begin{center}
\includegraphics[scale=1.0]{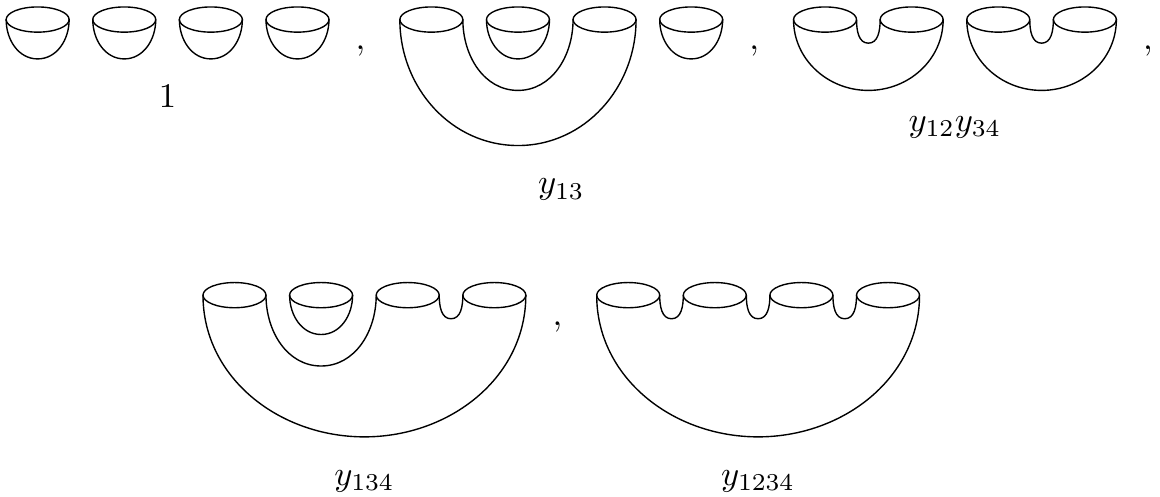}
\caption{Elements $1, y_{1234}$ and examples of elements of types $y_{ij},$ $y_{ij}y_{kl}$, $y_{ijk}$ in $A(4)$.}
\label{fig_3_5}
\end{center}
\end{figure}

Recall from~\cite[Section 2.7]{Kh1} that  there is an $S_4$-invariant skein relation on the eight vectors of the  last three types:
\begin{equation} \label{eq_eight_els}
 (y_{12}y_{34}+y_{13}y_{24}+y_{14}y_{23})-(y_{123}+y_{124}+y_{134}+y_{234}) +\beta y_{1234} = 0
\end{equation}
Among these eight vectors, only $y_{13}y_{24}$ has a diagram with an "intersection" of its surfaces, see Figure~\ref{fig_two_tubes}.  In  fact, the remaining fourteen elements of this  spanning set  all have  "planar" diagrams without overlapping components, see examples in  Figure~\ref{fig_3_5}.

\begin{figure}[!htb]
\begin{center}
\begin{tikzpicture}
[scale=0.07]

\draw [fill=white] (20,0) to [in=180, out=270] (45,-30) to [out=0, in=270] (70,0) to (60,0) to [out=270, in=0] (45,-20) to [out=180, in=270] (30,0) -- cycle;

\draw [fill=white] (0,0) to [in=180, out=270] (25,-30) to [out=0, in=270] (50,0) to (40,0) to [out=270, in=0] (25,-20) to [out=180, in=270] (10,0) -- cycle;

\foreach \i in {0,20,40,60}
{
\draw [fill=white, yscale=0.3] (\i+5,0) circle (5);
}

\node at (35,-40) {$y_{13} y_{24}$};

\end{tikzpicture}
\caption{\label{fig_two_tubes} Two tubes in the diagram of $y_{13}y_{24}$ overlap.}
\end{center}
\end{figure}
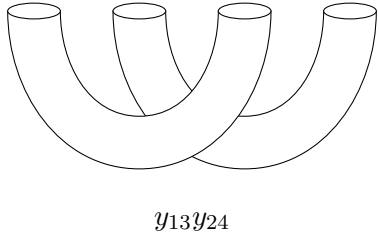

Consequently, element $y_{13}y_{24}$ of the spanning set for $A(4)$ can be written  as a linear  combination  of "crossingless" cobordisms.
Informally, we call a surface with $n$ boundary circles \emph{crossingless} if it can be drawn such that the components do not interlace.

\vspace{0.1in}


\subsection{Planar partitions, crossingless matchings, and crossingless surfaces}
\label{subsec_plane}
\quad
\vspace{0.1in}

{\it Planar partitions and crossingless matchings.}
Recall that $\bbn{n}$ denotes the set of crossingless matchings of $2n$ points on a horizontal line. It has cardinality  $c_n$, the $n$-th Catalan number, see  formula (\ref{catalan}).

Consider the set $\PD^{n}$ of planar partitions of an $n$-element set. These are decompositions of $\{1,\dots, n\}$ into non-empty  subsets such that the configuration  of  these  subsets can be drawn  in the lower half-plane  by connecting points in each subset by arcs  and without arcs  from  different  subsets intersecting. Equivalently, there should exist no quadruple of  numbers $1 < i_1 < i_2 < i_3 < i_4\le  n$ with  $i_1,i_3$ in one  subset and  $i_2,i_4$ in another.

Given a planar  partition $\lambda\in\PD^n$, it can be depicted by connecting points in  each  $m$-element subset (these points lies on the horizontal line with $n$  marked points $p_1,\dots, p_n$) by  $m$ arcs to a central point somewhere in the lower half-plane. In this configuration there are $n$ non-intersecting arcs connecting $n$ points on the horizontal line to $k$ points in the lower half-plane, where $k$ is the number  of sets in the planar partition. Figure~\ref{fig_planar_p} shows an example of the  diagram for the  planar partition
 $(\{1,4,6\},\{2,3\},\{5\})$ in $\PD^6$.

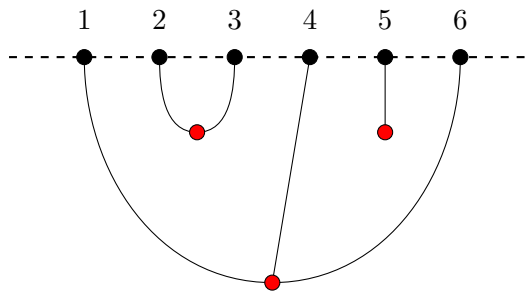
\begin{figure}[!htb]
\begin{center}
\hspace*{-2mm}
\begin{tikzpicture}
[xscale=1, yscale=1]

\foreach \i in {1,2,3,4,5,6}
{
\draw [fill=black]  (\i,0) circle (0.1);
\node at  (\i,0.5) {\i};
}

\draw (1,0) to [out=270, in=180] (3.5,-3) to [out=0, in=270] (6,0);
\draw (4,0) to (3.5,-3);
\draw (2,0) to [out=270, in=180] (2.5,-1) to [out=0, in=270] (3,0);
\draw (5,0) to (5,-1);
\draw [fill=red]  (2.5,-1) circle (0.1);
\draw [fill=red]  (3.5,-3) circle (0.1);
\draw [fill=red]  (5,-1) circle (0.1);

\draw [dashed, thick] (0,0) to (7,0);

\end{tikzpicture}
\caption{Diagram of the planar partition $(\{1,4,6\},\{2,3\},\{5\})$ in $\PD^6$. Points in the lower half-plane where the arcs end are shown in red.}
\label{fig_planar_p}
\end{center}
\end{figure}

To a planar partition $\lambda\in \PD^n$, also called a \emph{non-crossing partition}, we assign a crossingless  matching $\Phi_0(\lambda)\in \bbn{n}$ as follows. Take a planar diagram of $\lambda$ and form a standard retract closed neighbourhood in $\R^2_-$ of the configuration of $n$ arcs and $k$ inner points. This neighbourhood $N(\lambda)$ consists of $k$ connected components. Each component deformation retracts onto the  corresponding tree of the diagram of $\lambda$. The intersection of $N(\lambda)$ with the horizontal line  $\R$ consists of $n$  closed intervals, one for each point $p_1,\dots, p_n$.

The boundary of these intervals constitute $2n$ points $p_1',\dots, p_{2n}'$, with points $p_{2i-1}',p_{2i}'$ being  the boundaries of the  interval that contains  the  point  $p_i$. Boundary of  $N(\lambda)$ consists of  these $n$ intervals  together with $n$  arcs  that lie in  the lower half-plane and constitute a  crossingless matching of points $p_1',\dots, p_{2n}'$. Denote this matching by $\Phi_0(\lambda)$.
This  map
\begin{equation}
    \Phi_0 \ :  \ \PD^n \lra \bbn{n}
\end{equation}
is a bijection between planar  partitions and crossingless matchings.

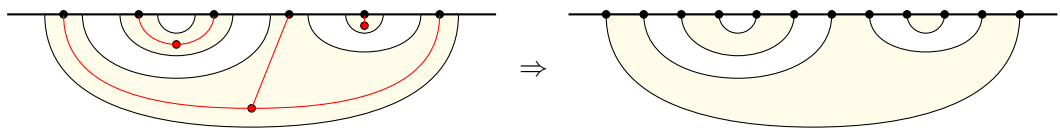
\begin{figure}[!htb]
\begin{center}
\hspace*{-2mm}
\begin{tikzpicture}
[xscale=0.5, yscale=0.5]

\draw [fill=yellow!10] (0,0) to [in=180, out=270] (5.5,-3) to [out=0, in=270] (11,0) to (10,0) to [out=270, in=0] (8.5,-1) to [out=180, in=270] (7,0) to (6,0) to [out=270, in=0] (3.5,-1.7) to [out=180, in=270] (1,0) --cycle ;

\draw [fill=yellow!10] (2,0) to [out=270, in=180] (3.5,-1.1) to [out=0, in=270] (5,0) to (4,0) to [out=270, in=0] (3.5,-0.5) to [out=180, in=270] (3,0) --cycle;

\draw [fill=yellow!10] (8,0) to [out=270, in=180] (8.5,-0.5) to [out=0, in=270] (9,0);

\foreach \i in {}
{
\draw [fill=black] (\i,0) circle (0.1);
}

\draw [fill=red] (5.5,-2.5) circle (0.1);

\draw [fill=red] (3.5,-0.8) circle (0.1);

\draw [fill=red] (8.5,-0.3) circle (0.1);

\draw [fill=black] (0.5,0) circle (0.1);

\draw [fill=black] (2.5,0) circle (0.1);

\draw [fill=black] (4.5,0) circle (0.1);

\draw [fill=black] (6.5,0) circle (0.1);

\draw [fill=black] (8.5,0) circle (0.1);

\draw [fill=black] (10.5,0) circle (0.1);

\draw [red] (0.5,0) to [in=180, out=270] (5.5,-2.5) to [out=0, in=270] (10.5,0) ;

\draw [red] (2.5,0) to [out=270, in=180] (3.5,-0.8) to [out=0, in=270] (4.5,0);

\draw [red] (5.5,-2.5) to (6.5,0);

\draw [red] (8.5,-0.3) to (8.5,0);

\draw [thick] (-1,0) -- (12,0);

\node at (13,-1.5) {$\Rightarrow$};

\end{tikzpicture}
\begin{tikzpicture}
[xscale=0.5, yscale=0.5]

\draw [fill=yellow!10] (0,0) to [in=180, out=270] (5.5,-3) to [out=0, in=270] (11,0) to (10,0) to [out=270, in=0] (8.5,-1) to [out=180, in=270] (7,0) to (6,0) to [out=270, in=0] (3.5,-1.7) to [out=180, in=270] (1,0) --cycle ;

\draw [fill=yellow!10] (2,0) to [out=270, in=180] (3.5,-1.1) to [out=0, in=270] (5,0) to (4,0) to [out=270, in=0] (3.5,-0.5) to [out=180, in=270] (3,0) --cycle;

\draw [fill=yellow!10] (8,0) to [out=270, in=180] (8.5,-0.5) to [out=0, in=270] (9,0);

\foreach \i in {0,1,2,3,4,5,6,7,8,9,10,11}
{
\draw [fill=black] (\i,0) circle (0.1);
}

\draw [thick] (-1,0) -- (12,0);

\end{tikzpicture}

\caption{Left: non-crossing partition $\lambda$ as in Figure~\ref{fig_planar_p}, here shown in red, and its regular neighbourhood $N(\lambda)$, shaded in yellow. Right: Arcs on the boundary  of $N(\lambda)$ constitute a crossingless matching $\Phi_0(\lambda)$.   }
\label{fig_map_phi_0}
\end{center}
\end{figure}

\begin{figure}[!htb]
\begin{center}
\hspace*{-2mm}
\begin{tikzpicture}
[xscale=0.7, yscale=0.7]

\draw [fill=yellow!10] (0,0) to [in=180, out=270] (5.5,-3) to [out=0, in=270] (11,0) to (10,0) to [out=270, in=0] (8.5,-1) to [out=180, in=270] (7,0) to (6,0) to [out=270, in=0] (3.5,-1.7) to [out=180, in=270] (1,0) --cycle ;

\draw [fill=yellow!10] (2,0) to [out=270, in=180] (3.5,-1.1) to [out=0, in=270] (5,0) to (4,0) to [out=270, in=0] (3.5,-0.5) to [out=180, in=270] (3,0) --cycle;

\draw [fill=yellow!10] (8,0) to [out=270, in=180] (8.5,-0.5) to [out=0, in=270] (9,0);

\foreach \i in {0,1,2,3,4,5,6,7,8,9,10,11}
{
\draw [fill=black] (\i,0) circle (0.1);
}

\draw [thick] (-1,0) -- (12,0);

\node at (-0.5,-0.25) {1};
\node at (0.5,-0.25) {2};
\node at (1.5,-0.25) {1};
\node at (2.5,-0.25) {2};
\node at (3.5,-0.25) {1};
\node at (4.5,-0.25) {2};
\node at (5.5,-0.25) {1};
\node at (6.5,-0.25) {2};
\node at (7.5,-0.25) {1};
\node at (8.5,-0.25) {2};
\node at (9.5,-0.25) {1};
\node at (10.5,-0.25) {2};
\node at (11.5,-0.25) {1};

\end{tikzpicture}
\caption{Checkerboard coloring in the complement  of a crossingless matching}
\label{fig_checker}
\end{center}
\end{figure}
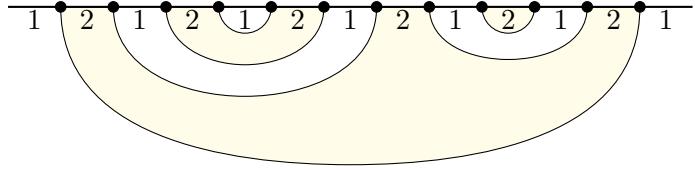

An example of the bijection $\Phi_0$ is depicted in Figure~\ref{fig_map_phi_0}.
To construct the inverse bijection $\Phi_0^{-1}$, start with a crossingless  matching  $b$. The matching decomposes
 the lower half-plane $\R^2_-$ into $n+1$ connected regions. Label these regions by  colors   $1$ and $2$  in a checkerboard fashion so that  the outer region (the unique unbounded region) is labelled  $1$, see Figure~\ref{fig_checker}.  Each region colored  $2$  has a boundary that is a union  of  horizontal intervals and inner arcs in $\R^2_-$.
 Given a region that contains $m$ horizontal intervals, choose a point   inside each interval, a point $u$ inside the region, and connect the $m$ points to the point $u$ by $m$ non-intersecting arcs. The region deformation retracts onto the union of these arcs. Taking these unions over all regions of $b$  colored $2$ gives a diagram of planar partition. This is the planar  partition $\Phi_0^{-1}(b)$.

\vspace{0.1in}

{\it Crossingless surfaces and crossingless matchings.}
For an accurate definition, position $n$ circles, each of diameter $1$, on the $xy$-plane  $\R^2$ so that their  centers are located on the $x$-axis, at points with the $x$-coordinate $2,4,\dots, 2n$, respectively,  and denote  this collection of circles  $C_n$. The circles intersect the $x$-axis at  $2n$ points $2\pm \frac{1}{2},\dots, 2n\pm \frac{1}{2}$.

Place  $\R^2$ in $\R^3$ in the  usual way, by adding the third coordinate $z$  and taking $\R^2$ to be the  plane $z=0$. This plane splits $\R^3$ into  half-spaces $\R^3_+$ and $\R^3_-$. We consider  surfaces $S$ properly  embedded in $\R^3_-$ with the boundary $C_n$. Form  the  intersection $S\cap \R^2_{xz}$ of  $S$ with the $xz$-coordinate plane. Upon a slight deformation  of $S$ while keeping  its boundary on  the  $xy$-plane fixed we can assume that the  intersection $S\cap \R^2_{xz}$ is a one-manifold  which is  the  union of circles  and $n$  arcs with boundary the above $2n$ points on the $x$-axis: $2\pm \frac{1}{2}, 4\pm \frac{1}{2},\dots, 2n\pm \frac{1}{2}$.

These $n$ arcs in the lower half-plane $\R^2_{xz,-}$ constitute a crossingless matching $b\in \bbn{n}$  of $2n$  points.

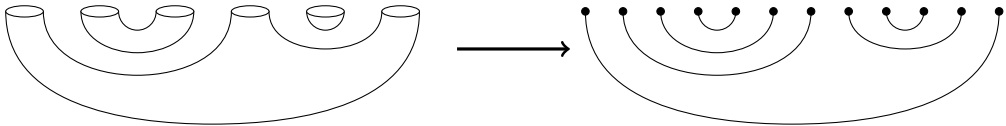
\begin{figure}[!htb]
\begin{center}
\hspace*{-2mm}
\begin{tikzpicture}
[scale=0.5]

\draw [fill=white] (0,0) to [in=180, out=270] (5.5,-3) to [out=0, in=270] (11,0) to (10,0) to [out=270, in=0] (8.5,-1) to [out=180, in=270] (7,0) to (6,0) to [out=270, in=0] (3.5,-1.7) to [out=180, in=270] (1,0) --cycle ;

\draw [fill=white] (2,0) to [out=270, in=180] (3.5,-1.1) to [out=0, in=270] (5,0) to (4,0) to [out=270, in=0] (3.5,-0.5) to [out=180, in=270] (3,0) --cycle;

\draw [fill=white] (8,0) to [out=270, in=180] (8.5,-0.5) to [out=0, in=270] (9,0);

\foreach \i in {0,2,4,6,8,10}
{
\draw [fill=white, yscale=0.3] (\i+0.5,0) circle (0.5);
}

\draw [->, very thick] (12,-1) -- (15,-1);

\end{tikzpicture}
\begin{tikzpicture}
[scale=0.5]

\draw [fill=white] (0,0) to [in=180, out=270] (5.5,-3) to [out=0, in=270] (11,0);
\draw (10,0) to [out=270, in=0] (8.5,-1) to [out=180, in=270] (7,0);
\draw (6,0) to [out=270, in=0] (3.5,-1.7) to [out=180, in=270] (1,0)  ;

\draw [fill=white] (2,0) to [out=270, in=180] (3.5,-1.1) to [out=0, in=270] (5,0);
\draw (4,0) to [out=270, in=0] (3.5,-0.5) to [out=180, in=270] (3,0) ;

\draw [fill=white] (8,0) to [out=270, in=180] (8.5,-0.5) to [out=0, in=270] (9,0);

\foreach \i in {0,1,2,3,4,5,6,7,8,9,10,11}
{
\draw [fill=black] (\i,0) circle (0.1);
}

\end{tikzpicture}
\caption{\label{fig_surf}  Bijection between (isotopy classes of) crossingless surfaces and matchings}
\end{center}
\end{figure}

Vice versa,  to a crossingless matching  $b\in \bbn{n}$ we can associate a 2-manifold $S(b)$. First, color the regions of the lower half-plane $\R^2_-=\R^2_{xz,-}$ which contains the matching by colors   $1$ and $2$  in a checkerboard fashion so that  the outer  color  is $1$. Regions of  $\R^2_-$ colored 2 are bounded. Each one  has a boundary that is a union  of  horizontal intervals and inner arcs in $\R^2_-$. Horizontal intervals  connect points $2i\pm \frac{1}{2}$ for various $i$, $1\le i \le   n$. To associate a surface $S(b)$ to $b$ we  thicken each region $V$ colored $2$ into a 3-dimensional region in $\R^3_-=\R^2_-\times \R$. One way to do that is by  forming  $V\times [0,1]\subset \R^3_-$ and  then  smoothing its corner arcs to get a region bounding  a  smooth surface in $\R^3_-$. Each horizontal interval $[2i-\frac{1}{2},2i+\frac{1}{2}]$ in $V$ first gets multiplied by $I=[0,1]$ and then smoothed  out to a circle of diameter one. In this way $n$ pairs of consecutive points on the boundary of  the matching $b$ turn  into $n$ circles on the plane $\R^2=\partial  \R^3_-$. Each region $V$ of $b$ colored  $2$ turns  into  a 3-dimensional region that bounds  the union of  disks, each of diameter $1$, in $\R^2$ and a surface $S(V)$ in $\R^3_-$. Now to $b$ assign the union $S(b)$ of surfaces $S(V)$ over all regions $V$ labelled $2$.
The boundary of $S(b)$ consists of the union $C_n$ of  $n$ circles.

We refer to $S(b)$ as a \emph{crossingless surface} associated to the matching $b$. Denote by $\PS^n$ the set of crossingless surfaces associated to matchings $b\in \bbn{n}$ and by $\Phi_1$ the corresponding bijection
\begin{equation}\label{eq_psn}
    \Phi_1 \ :  \ \bbn{n} \lra \PS^n .
\end{equation}

This assignment  is inverse to the map depicted in Figure~\ref{fig_surf}.

Isotopy classes of surfaces $S(b)=\Psi_1(b)$ that result from this construction are in a bijection with planar partitions of $n$.
These surfaces are determined by their intersection with the lower half-plane $\R^2_-\subset  \R^3_-$. This intersection is a crossingless matching $b$.

Composing the two bijections above results in the bijection $\Phi$ from the set $\PD^n$ of non-crossing partitions to the set of crossingless surfaces:
\begin{equation}
    \Phi :  \ \PD^n \stackrel{\Phi_0}{\lra} \bbn{n} \stackrel{\Phi_1}{\lra} \PS^n
\end{equation}

Any oriented surface $S$ with a component of genus greater than $0$ and $\partial S \cong \sqcup_n \SS^1$ evaluates to the zero vector, $[S]=0$ in $A(n)$ for our series $\alpha$. Relation (\ref{eq_eight_els}) allows to reduce $y_{13}y_{24}$, which can be viewed as a crossing, to a linear combination of crossingless diagrams. Inductive application of this relation shows that $A(n)$ is spanned by vectors $[S(b)]$ associated to crossingless surfaces $S(b)$, for $b\in \bbn{n}$.
In particular,
\begin{equation}\label{eq_ineq}
\dim A(n) \leq c_n
\end{equation}
for any field $\kk$ and $\beta\in \kk^{\ast}$.

\vspace{0.1in}

\subsection{Meander determinants and size of \texorpdfstring{$A(n)$}{A(n)}}
\label{subsec_meander}
\quad
\vspace{0.1in}

To prove the opposite inequality to (\ref{eq_ineq}) when $\kk$ has zero characteristic, it is enough to show that  the bilinear form is nondegenerate on the subspace with a basis  $\{[S(b)]\}$, $b\in \bbn{n}$.

It turns out that the matrix of  this bilinear form in  the spanning set of crossingless matchings is the same as one of the auxiliary matrices appearing in \cite{KKKo}, namely, the matrix  for  the deformed {\it meander determinant}.

Consider two matchings $a,b\in \bbn{n}$, their surfaces  $S(a),S(b)$ and the closed surface  $S(a,b):=\overline{S}(b)S(a)$ given by  reflecting the surface $S(b)$ about the horizontal plane $\R^2$ and composing with $S(a)$ along the common $n$ circles.

To matchings $a,b $ there is also  associated a  collection $\overline{b}a$ of circles in the  plane, which has a  unique checkerboard coloring of its connected components (regions) by $\{1,2\}$ with the outer component colored $1$. Define $h_1(a,b), h_2(a,b)$ as the number of connected components of colors 1 and 2 respectively.

\begin{figure}[!htb]
\begin{tikzpicture}
\draw [dashed] (0,0) -- (9,0);
\foreach \i in {1,2,3,4,5,6,7,8}
{
\draw [fill=black] (\i,0) circle (0.05) ;
}

\draw (1,0) to[out=270, in=180] (1.5,-1) to[out=0, in=270] (2,0) to[out=90, in=180] (3.5,2) to[out=0, in=90] (5,0) to [out=270, in=180] (6.5,-2) to [out=0, in=270] (8,0) to [out=90,in=0] (7.5,1) to [out=180,in=90] (7,0) to[out=270, in=0] (6.5,-1) to[out=180, in=270] (6,0) to[out=90, in=0] (3.5,3) to[out=180, in=90] (1,0) ;

\draw (3,0) to[out=270, in=180] (3.5,-1) to [out=0, in=270] (4,0) to [out=90, in=0] (3.5,1) to [out=180, in=90] cycle;

\node at (3.5,0.5) {$2$};
\node at (3.5,2.5) {$2$};
\node at (6.5,1.5) {$1$};

\node at (11,0.5) {$h_2(a,b)=2$};
\node at (11,1.5) {$h_1(a,b)=1$};

\node at (8.5,-1.5) {$a$};

\node at (8.5,1.5) {$\overline b$};

\end{tikzpicture}
\caption{Diagram $\overline{b}a$ with $h_1(a,b)=1$. Note that regions colored $2$ have no "holes", that is, each  one is homeomorphic to a disk.}
\end{figure}
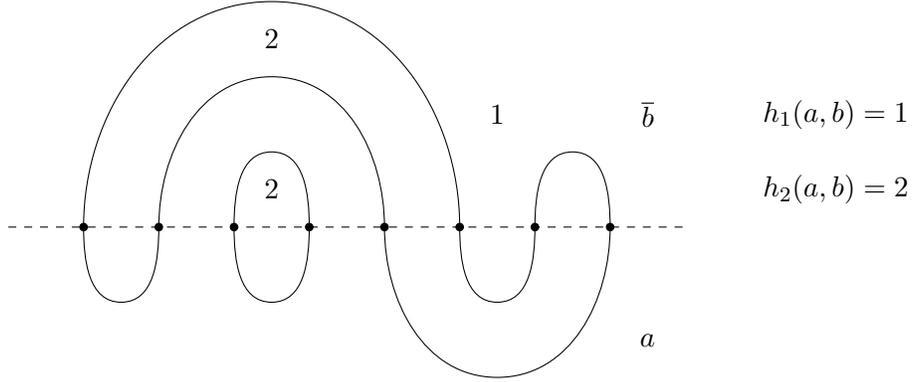

A component of the closed surface $S(a,b)$ has genus 0 if and only if
the corresponding region of the planar diagram $\overline{b}a$ colored  $2$ is a disk.
All  components  of $S(a,b)$ have genus $0$ iff no region of $\overline{b}a$ of color $2$ contains a region of color $1$ inside. Equivalently, all  color $2$ regions are non-nested disks and there is a unique region of color $1$ (the outer region). This  happens  iff $h_1(a,b)=1$.

The number of color 2 components is  $h_2(a,b)$. The bilinear form on $A(n)$ for the constant series $Z_{\alpha}(T)=\beta$ in the spanning set of crossingless surfaces $[S(b)]$ is given by
\begin{equation}\label{eq_matrix}
    ([S(a)],[S(b)]) \ = \ \delta_{1,h_1(a,b)} \beta^{h_2(a,b)},  \ \ a,b\in \bbn{n},
\end{equation}
where
\[
\delta_{i,j} = \begin{cases}
1, & i=j\\
0, & i\neq j.
\end{cases}
\]
Denote by $D_n(\beta)$  this matrix of size $\bbn{n}\times \bbn{n}$.

In \cite{KKKo} the authors study the determinant of $\bbn{n}\times \bbn{n}$ matrix $M(y_1,y_2)$ with the  $(a,b)$-entry
\[
y_1^{h_1(a,b)} y_2^{h_2(a,b)},
\]
where $y_1,y_2$  are formal variables, and show that it can be expressed in terms of Chebyshev polynomials of the second kind. Namely, let
$${\sf M}_n(y)  =
\prod_{h=1}^n { {\sf U}_h(y)^{c_{n,h}-c_{n,h+1}} } \, , \ \  \mathrm{where}
$$
\[
c_{n,h} = \binom{2n}{n-h}-\binom{2n}{n-h-1}, \ \ {\sf U}_h(2 \cos \theta) = \frac{\sin(h+1)\theta}{\sin \theta}.
\]
Chebyshev polynomial ${\sf U}_h(y)$  of the second kind is a polynomial of degree $h$ in $y$  with integer coefficients.
Then (see \cite{KKKo})
\begin{equation}
    \det(M(y_1,y_2)) =  (y_1 y_2)^{{|\bbn n|}/{2}}\cdot {\sf M}_n(\sqrt{y_1 y_2}).
\end{equation}
For any $a,b$
$$ h_1(a,b), h_2(a,b) \geq 1, $$
so that
\begin{equation}
    D_n(\beta) =  \lim_{y_1\to 0} \frac{1}{y_1} M(y_1,\beta).
\end{equation}
We need to prove that the following limit is nonzero, when the ground field has characteristic $0$:
\begin{equation}
\lim_{y_1 \to 0}  \frac{\det(M(y_1,\beta))}{y_1^{|\bbn n|}},
\end{equation}
so that  the determinant below does not vanish:
\begin{equation}
    \det (D_n(\beta)) = \lim_{y_1\to 0} \left(\frac{\beta}{y_1}\right)^{|\bbn{n}|/2} {\sf M}_n(\sqrt{y_1 \beta}) .
\end{equation}

At the point $y=0$ even-indexed Chebyshev polynomials ${\sf U}_{2h}(y)$ do  not vanish, while odd-indexed polynomials ${\sf U}_{2h+1}(y)$ have simple poles, when  $\nchar \kk =0$:
\[
y \to 0 \Rightarrow {\sf U}_{2h}(y) \sim 1, \ \ {\sf U}_{2h+1}(y) \sim y.
\]
The order of vanishing of $M_n(y)$ is
\[
\sum_{i=1, odd}^n (c_{n,i}-c_{n,i+1})= \binom{2n}{n-1}-2\binom{2n}{n-2}+2\binom{2n}{n-3} + ... + (-1)^n 2 \binom{2n}{0}.
\]
\begin{lemma}
\[
\binom{2n}{n} = 2 \sum_{i=1}^n (-1)^{i-1} \binom{2n}{n-i},
\]
\end{lemma}
\begin{proof}
The lemma follows from the identity
\[
(1+t)^{2n} = \binom{2n}{n} t^n + \sum_{i=1}^n \binom{2n}{n-k} \left(t^{n-k}+t^{n+k}\right),
\]
evaluated at $t=-1$.

It follows from the lemma that the order of vanishing can be rewritten as
\[
\binom{2n}{n}-\binom{2n}{n-1} ,
\]
which equals the Catalan number $c_n$.

\end{proof}

This completes  the proof of Theorem~\ref{thm_catalan}. Note that,  over a field $\kk$  of  characteristic $p$,
$\dim A(n)=c_n$ if $p$ does not divide the corrresponding  product  of the values of even Chebyshev polynomials $\mathsf{U}_{2h}(0)$ for $h\leq n/2$  and derivatives of odd Chebyshev polynomials $\mathsf{U}'_{2h+1}(0)$ for $h \leq (n-1)/2$.
Otherwise, $\dim A(n) < c_n$.

$\square$
\vspace{0.2in}

\subsection{Graded case and the Narayana numbers}\label{subset_narayana}
\quad
\vspace{0.1in}

\subsubsection{Non-crossing partitions and Narayana numbers}
Consider points $1,2,...,n$ placed in this  order on a circle. A partition of them into a disjoint union
\[
\{1,2,...,n\} = \{i_1,...,i_t\} \bigsqcup \{j_1,...,j_s\} \bigsqcup ...
\]
is called non-crossing (or planar, see earlier) if the parts avoid interlaps, when drawn via trees in the disk. Instead of the circle and the disk it bounds one can use the $x$-axis and the lower half-plane, see earlier.

The number of non-crossing partitions of $n$ points on a circle with exactly $k$ parts, $1\le k\le n$, is called the Narayana number \cite{P}:
\[
N(n,k) = \frac{1}{n}\binom{n}{k}\binom{n}{k-1}.
\]
Narayana numbers provide a distinguished refinement of Catalan numbers
\[
c_n = \sum_{k=1}^n N(n,k)
\]
and have the following generating function
\[
\sum_{n\ge 0,1\le k\le n} N(n,k)z^n t^{k-1} = \frac{
1-z(t+1)-\sqrt{1-2z(t+1)+z^2(t-1)^2}
}{
2tz
}.
\]

\subsubsection{Graded dimensions of $A(n)$}
The vector space $A(n)$ is spanned by diffeomorphism classes of viewable surfaces (elements of $\PS^n$, see formula (\ref{eq_psn}) with the boundary diffeomorphic to the disjoint union of $n$ disks and each component of genus $0$. Assume that $\nchar \kk=0$. Then surfaces in $\PS^n$ constitute a basis of $A(n)$ and carry a natural degree. The degree of $[S]\in A(n)$, for $S\in \PS^n$, is
\begin{equation}
    \deg[S] \ = \ n -\chi(S),
\end{equation}
where $\chi(S)$ is the Euler characteristic of $S$. In this way $A(n)$ becomes a $2\Z_+$-graded vector space and even  a $2\Z_+$-graded commutative associative algebra. The unit element of $A(n)$, viewed as an algebra, is given by the union of $n$ disks and has degree $0$.  Other basis elements $[S]$ of $A(n)$ have positive even degrees.

This can be written via the action of the grading operator
\[
q^{n-\chi}: A(n) \to A(n) \otimes \kk[q^{\pm 1}], \ \ \
q^{n-\chi} [ S  ] = q^{n-\chi({S})} [ S ] .
\]
 It is natural to consider the trace of $q^{n-\chi}$ as a Laurent polynomial in $q$. Since the homogeneous summands have non-negative gradings, the trace of $q^{n-\chi}$ is a genuine polynomial in $q^2$ rather than a Laurent polynomial. This polynomial  depends on $n$ and describes the graded dimension of $A(n)$ with the above grading.

\begin{proposition} If $\nchar \kk=0$,
the space $A(n)$ is naturally graded. Dimensions of its homogeneous components are
the coefficients of the polynomial $\tr_{A(n)} q^{n-\chi}$ are Narayana numbers:
\[
\tr_{A(n)} q^{n-\chi} = \sum_{k=0}^{n-1} N(n,n-k) q^{2k}.
\]
\end{proposition}
\begin{proof}
Crossingless surfaces correspond to non-crossing partitions.
The connected components of a crossingless surface provide a partition of the set of the boundary components. From the discussion in earlier subsections, crossingless surfaces form a basis of the set $A(n)$.
A crossingless surface $S\in\PS^n$ with $k$ components gives a decomposition
\[
n = i_1 + ... + i_k.
\]
The Euler characteristic of a component with $i$ boundary circles is $2-i$.
Thus,
\[
n-\chi(S) = n - \sum_{j=1}^k (2-i_j) = 2(n-k).
\]
\end{proof}

\begin{remark} Commutative algebra $A(n)$ is naturally a Frobenius algebra, via the trace map of capping off a surface $S$ by $n$ disks and evaluating the resulting closed surface. The trace map does not respect the  grading. To make the trace map homogeneous, one can make $\beta$ a formal variable with $\deg(\beta)=-2$ and change from a field $\kk$ to a polynomial ring $\kk[\beta]$. The resulting pairing on $A(n)$, defined over $\kk[\beta]$, is not perfect, though.
\end{remark}

\vspace{0.1in}

\subsection{A commutative Frobenius algebra in \texorpdfstring{$\Rep(osp(1|2))$}{Rep(osp(1,2)}}\label{subset_osp}
\quad
\vspace{0.1in}

In this Section we give an alternative derivation of Theorem~\ref{thm_catalan} and  relate the category $\udcobal$ for the constant series $\alpha$ with the representation category of Lie superalgebra $osp(1|2)$.

\vspace{0.1in}

 Start with the category $\cC'$ of finite-dimensional representations of  $osp(1|2)$, viewed  as a Lie superalgebra over a field $\kk$ of characteristic $0$, see~\cite{ES,B1,B2,KRe,Z} and~\cite[Theorem  A.3]{EO}. An object of $\cC'$ is a  $\Z/2$-graded representation of $U(osp(1|2))$. The  defining representation $V\cong \kk^{1|2}$ generates a Karoubi-closed tensor subcategory $\cC$ of $\cC'$. One can think of  $\cC$ as "one-half" of the category $\cC'$.
 An  irreducible object of $\cC'$ is either isomorphic too an irreducible object of $\cC$ or to such an object  tensored with the odd one-dimensional representation of $osp(1|2)$. Both  $\cC$ and  $\cC'$ are semisimple $\kk$-linear categories.

 Category $\cC$ is similar  to the category of representations of the Lie group $SO(3)$. Namely, it has one irreducible representation $V_{2n}$ in each odd dimension $2n+1$, $n=0,1,\dots$, just like  $SO(3)$, and the tensor  product decomposition
 \begin{equation*}
 V_{2n}\otimes V_{2m} \cong V_{2|n-m|}\oplus V_{2|n-m|+2}\oplus \dots \oplus  V_{2(n+m)}
 \end{equation*}
 has the same multiplicities as for the corresponding representations of $SO(3)$. In particular, taking $W=V_0\oplus V_2$, the dimension of the space of invariants  $\Hom_{\cC}(V_0,W^{\otimes n})$  equals
 to the corresponding multiplicity  for representations of $SO(3)$. In the latter case, the analogue of $W$ is the  four-dimensional representation $\widetilde{W}$ of $SO(3)$ isomorphic, as a representation of $sl(2)$, to $\widetilde{V}_1\otimes \widetilde{V}_1\cong \widetilde{V}_0\oplus \widetilde{V}_2$, where $\widetilde{V}_1$ is the fundamental representation of $sl(2)$, and $\widetilde{V}_n$ is the irreducible representation of $sl(2)$ of dimension $n+1$. Multiplicities for these representations are the same in the categories of $SO(3)$ and $sl(2)$  representations. The identity representations $\be$ in these  categories are isomorphic to $V_0$ and $\widetilde{V}_0$, respectively. One obtains that
 \begin{equation}\label{eq_to_cat}
     \dim \Hom_{\cC}(\be,W^{\otimes n}) =\dim \Hom_{SO(3)}(\be,\widetilde{W}^{\otimes n})= \dim \Hom_{sl(2)}(\be, \widetilde{V}_1^{\otimes 2n}) = c_n,
 \end{equation}
 where $c_n$ is the $n$-th Catalan number.

Let $E$ be a 2-dimensional $\kk$-vector space with the basis $\{a,b\}$. Consider the exterior algebra
$$A=\Lambda^{\ast} E,$$
thus, $a^2=b^2=0$ and $ab=-ba$ in $A$. Algebra $A$ has the following super-derivations:
$$x=(ab+1)\partial_a,\; y=(ab+1)\partial_b.$$
These super-derivations act on the left in the basis $\{1,a,b,ab+1\}$ of $A$ as follows:
$$
\begin{array}{|c|c|c|c|c|c|}
\hline
& x & y & [x,x]=2x^2&[x,y]=xy+yx &  [y,y]=2y^2\\ \hline
1&0&0&0&0&0\\
a&ab+1&0&2b&-a&0\\
b&0&ab+1&0&b&-2a\\
ab+1&b&-a&0&0&0\\
\hline
\end{array}
$$
 Thus, $x$ and $y$ generate an action of the Lie super-algebra $osp(1|2)$ on $A$. As an $osp(1|2)$-module,
 $A\cong \be \oplus  V_2$, where $\be$ is the trivial module spanned by $1\in A$  and $V_2\cong  \kk^{1|2}$ is irreducible of dimension $(1|2)$ spanned by $a,b,ab+1$.
 In  particular, $A$ is an object of $\cC$.

 The multiplication $A\otimes  A\lra  A$ is  a map of  $osp(1|2)$-modules, since $osp(1|2)$ acts by super-derivations. The unit map $\iota:\be\lra A$ is an $osp(1|2)$-module map as well.
 Being the exterior algebra,  the  algebra  $A$ is a commutative algebra object in the category of super-vector spaces and, consequently,  in the category $\cC$.

 The algebra object
 $A\in \cC$ is Frobenius with respect to the linear form $A\rightarrow \be$ sending $1\in A$  to a non-zero constant $\beta\in\be$
 and $a,b,ab+1$ to zero.

 Assume $\beta =1$. Then the Frobenius comultiplication $A\to A\ot A$ is given by
 \begin{equation}\label{eq_4_lines}
 \begin{array}{l}
 1\mapsto 1\ot ab-a\ot b+b\ot a+ab\ot (1+ab)\\
 a\mapsto a\ot ab+ab\ot a\\
 b\mapsto b\ot ab+ab\ot a\\
 ab\mapsto ab\ot ab
 \end{array}
 \end{equation}
 Composing with the  multiplication,  we  see  that the composition $m\Delta:A\lra  A$, which is the handle endomorphism, is the zero map.
One computes immediately that $\alpha_0=\beta=1$, $\alpha_i =0$ for $i>0$, so that the  generating function for $A$ is $Z(T)=1$.
For  arbitrary invertible $\beta$ one should insert $\beta^{-1}$  after  the  arrow in each line in  the map (\ref{eq_4_lines}) above. Then  $\alpha_0=\beta$ and $\alpha_i=0$ for $i>0$, with $Z(T)=\beta$.

\vspace{0.1in}

Consider the skein category $\scobal$ for the  constant series $Z_{\alpha}(T)=\beta$, its gligible quotient $\Cobal$, and  the  Karoubi envelope $\udcobal$ of the latter.
The  skein category has  the  relations that the  handle  is  equal  to $0$  while the  $2$-sphere evaluates to $\beta$.

There is  a  functor $F_A:\scobal\lra \cC$ taking  the circle object $1$ to $A$ as explained in Section~\ref{subsec_universal}, see (\ref{eq_F_A}).  To apply Proposition~\ref{prop_assume}, note that $\cC$ is semisimple and  that any object of $\cC$ is  a direct summand of $A^{\otimes n}$ for some $n$. To show that $F_A$ is surjective on homomorphisms, it suffices to  check that  the maps
\begin{equation*}
\Hom(n,m)\stackrel{F_A}{\lra} \Hom_{\cC}(A^{\otimes n},A^{\otimes m})
\end{equation*}
induced by $F_A$ are surjective. Furthermore, by duality,
it is  enough to establish surjectivity of  maps
\begin{equation} \label{eq_map_FA}
\Hom(0,n)\stackrel{F_A}{\lra} \Hom_{\cC}(\be,A^{\otimes n}).
\end{equation}
The image $F_A(\Hom(0,n))$  is a subspace of dimension   bounded from below by the rank  of  the matrix $D_n(\beta)$ from Section~\ref{subsec_meander}, see (\ref{eq_matrix}),
\begin{equation*}
    \dim(F_A(\Hom(0,n))) \ \ge \ \mathrm{rk}(D_n(\beta)).
\end{equation*}
We proved in that section that $D_n(\beta)$ has nonzero determinant over a characteristic zero field and has rank equal to the  Catalan number $c_n$. Therefore,
\begin{equation}\label{eq_dim_ge}
    \dim(F_A(\Hom(0,n))) \ \ge \ c_n.
\end{equation}
From equation (\ref{eq_to_cat}) we know that dimension of $\Hom_{\cC}(\be,A^{\otimes n})$ is also  $c_n$. Consequently,
the inequality in (\ref{eq_dim_ge}) is equality and the map (\ref{eq_map_FA}) is an isomorphism. We conclude that $F_A$ is surjective on homomorphisms. By Proposition~\ref{prop_assume} it induces an equivalence from the Karoubi completion of the negligible quotient $\Cobal$ of $\scobal$ to $\cC$.

\vspace{0.1in}

\begin{theorem} \label{thm_dcob_C}
Over a field $\kk$ of characteristic zero,
the category  $\udcobal$ is equivalent, as a symmetric monoidal category, to the above category $\cC$ of representations of $osp(1|2)$:
\begin{equation}
    \udcobal \cong \cC.
\end{equation}
\end{theorem}

\vspace{0.1in}

This gives an alternative proof that $\dim A(n)=c_n$ over a characteristic zero field, which is part of  Theorem~\ref{thm_catalan} .

\begin{remark} (Vera Serganova)
 Let $G$ be a supergroup and  $\mfg=\mfg(G)=\mfg_0\oplus \mfg_1$, where $\mfg_1$ is the odd part of the corresponding Lie superalgebra. Assume that the top exterior power  $\Lambda^{top}\mfg_1=\Lambda^{\dim(\mfg_1)}\mfg_1$ is the trivial representation $V_0$ of $\mfg_0$.
Then the  ring of functions $A=\mathrm{Spec}(G/G_0)$ on $G/G_0$ is a super-commutative Frobenius algebra isomorphic to the exterior algebra $\Lambda^{\ast}\mfg_1$.  $A$ is a commutative Frobenius algebra object in the category of  $\mfg$-modules (the category of super vector spaces with an action of $G$).

As a $G$-module, $A$ is isomorphic to the induced representation
$\mathrm{Ind}^{\mfg}_{\mfg_0} (V_0)$,
\begin{equation}
    A \cong \Lambda^{\ast}\mfg_1 \cong \mathrm{Ind}^{\mfg}_{\mfg_0} (V_0).
\end{equation}
The $G-$invariant trace map
$\epsilon: A\lra \kk$ comes from the
Frobenius
reciprocity via the identity map of $\mfg_0$-modules from the trivial representation  $V_0$ to itself.

 For our case of $G =OSp(1|2)$ and $G_0\cong SL(2)$ this gives an alternative construction of the commutative Frobenius algebra as considered in this section.
\end{remark}

\section{Gram determinants for theories of rank one and two}\label{sec_generating}

Recall that the rank $K=\dim A(1)$ of the theory with $Z(T)=P(T)/Q(T)$ is $\max(\deg P +1,\deg Q)$.

\subsection{Generating function  \texorpdfstring{${\beta}/{(1-\gamma T)}$}{b/(1-gT)} and  the Deligne category}\label{subsec_deligne}

Generating function of a rank one theory has the form
\begin{equation}\label{eq_factor_1}
    Z_{\alpha}(T) = \frac{\beta}{1-\gamma T}  = \beta + \beta\gamma T + \beta \gamma^2 T^2 + \dots
\end{equation}
with $\beta\in \kk^{\ast}$ and $\gamma\in\kk$. When $\gamma=0$, the generating function is constant and the theory is studied in Section~\ref{sec_constant_gen}.

Assume now $\gamma\not= 0$.
Rescaling $T\longmapsto \lambda T$ and $Z(T)\longmapsto \lambda^{-1}Z(\lambda T)$ by invertible $\lambda$ leads to isomorphic  theories, see Section~\ref{subsec-scaling}. Rescaling $T$  to $\gamma^{-1}T$ and $Z(T)$ to $\gamma Z(\gamma^{-1}T)$ reduces the theory to that for the generating function
\begin{equation}
    Z(T) = \frac{\beta\gamma}{1 - T}
\end{equation}
and allows us to restrict to the case $\gamma=1$ and generating function
\begin{equation}\label{eq_beta_Z}
    Z(T) = \frac{\beta}{1- T}=\beta+\beta T + \beta T^2 + \dots  ,
\end{equation}
with the handle relation in this case
shown in Figure~\ref{fig_gen_beta_Z} (note  that the rescaling above changes the handle relation,  in general, by rescaling $x$).

\begin{figure}[!htb]
    \centering
    \begin{tikzpicture}[scale=0.6]
    \draw (0,0) to[out=250, in=180] (1,-2) to [out=0, in=290] (2,0);
    \draw [yscale=0.3] (1,0) circle(1);
    \draw (0.5,-1) to[out=290, in=250] (1.5,-1);
    \draw (0.6,-1.1) to[out=60, in=120] (1.4,-1.1);
    \node[scale=1] at (3.5,-1) {$ \ \ =  \beta$};
    \draw (5,0) to[out=250, in=180] (6,-2) to [out=0, in=290] (7,0);
    \draw [yscale=0.3] (6,0) circle(1);
    \end{tikzpicture}
    \caption{\label{fig_gen_beta_Z} Handle relation for the generating function in (\ref{eq_beta_Z}).}
\end{figure}
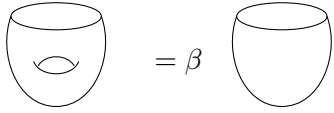

The skein category $\scobal$ for this sequence $\alpha=(\beta,\beta,\beta,\dots)$ is equivalent to the  partition  category $\Pa_{\beta}$, the  category  $\dcobal$ to the Deligne category $\delcatv{\beta}$, and the gligible quotient  $\udcobal$ is equivalent to  the gligible quotient  $\delcatbar{\beta}$.

Going back to the rational function in (\ref{eq_factor_1}) with $\gamma \not=0$, we  obtain three equivalences of  categories and  a commutative diagram
\begin{equation} \label{eq_short_cd_3}
\begin{CD}
\scobal  @>>> \dcobal @>>> \udcobal  \\
 @V{\cong}VV      @V{\cong}VV  @V{\cong}VV  \\
  \Pa_{\beta\gamma} @>>> \delcatv{\beta\gamma} @>>> \delcatbar{\beta\gamma}
\end{CD}
\end{equation}
where in the bottom row appear  the partition category, the Deligne category, and its gligible quotient, respectively, going from left to right, for the  parameter  $t=\beta\gamma$, and $\alpha=(\beta,\beta\gamma,\beta\gamma^2,\dots)$.

When $\nchar \kk=0$ and $\beta\gamma\notin \Z_+\subset \kk$,  the negligible ideal is  zero, the quotient does not change the category, and  there are equivalences
\begin{equation}
    \dcobal\cong \udcobal \cong \delcatbar{\beta\gamma}\cong \delcatv{\beta\gamma},
\end{equation}
between the four categories in the middle and on the right of the diagram (\ref{eq_short_cd_3}), see~\cite[Th\'eor\`eme 2.8]{D1}.

When $\beta\gamma=n$ and $\nchar \kk=0$, the category $\udcobal$ is equivalent to the category of finite-dimensional representations of the symmetric group $S_n$ over $\kk$. If $\nchar \kk=p$, one should replace $S_n$ by $S_m$, where $m$ is the remainder after the division of $n$ by $p$.

\vspace{0.1in}

For the generating function in (\ref{eq_factor_1}) the handle relation is obtained from that in Figure~\ref{fig_gen_beta_Z} by replacing $\beta$ by $\gamma$.

The $n$-circle state space $A_{\alpha}(n)$ of that theory is spanned  by visible cobordisms with  every connected  component of genus zero. Diffeomorphism classes of these cobordisms are in a bijection with set-theoretic partitions of $n$. Properties of the Deligne category imply  that  for $\gamma\not=0$ and   $\beta\gamma$ not the image of an integer in $\kk$ the set of  these genus zero surfaces is a basis of $A_{\alpha}(n)$, so that in this case
\begin{equation}\label{eq_dim_B}
    \dim A_{\alpha}(n) \ = \ B_n,
\end{equation}
where $B_n$ is the Bell number, see Section~\ref{subsec_catalan}.
Recall from Section~\ref{sec_constant_gen} that for $\gamma=0$ genus zero surfaces span $A_{\alpha}(n)$ and constitute a basis of the latter when $\nchar \kk=0$ and $\beta\not=0$, leading to the formula (\ref{eq_dim_B}) in this case as well.

\vspace{0.1in}

{\it Bilinear form data.}
We next compute the bilinear  form on this  spanning set, keeping $\beta,\gamma$ as formal variables, for small  values of $n$.
Table~\ref{table_det_lin} below describes determinants of Gram matrices $G_n$ of size $B_n\times B_n$,  with rows and columns enumerated by set partitions $\lambda$ and  $\mu$ of $n$. To each set partition $\lambda$  we assign a surface $S(\lambda)$ that matches $n$ boundary circles into genus zero connected  components  via  parts of the partition.  These surfaces span $A_{\alpha}(n)$ for all values of $\beta,\gamma\in \kk$ and  constitute a basis in the generic case.

Two surfaces $S(\lambda),S(\mu)$ share  the common boundary of $n$  circles and can be glued into a closed surface $\overline{S(\mu)}S(\lambda)$. Evaluating this surface via the generating function $Z_{\alpha}(T)$ in (\ref{eq_factor_1}) gives a monomial  in $\beta$ and $\gamma$ that we put as the $(\lambda,\mu)$-entry of  the matrix $G_n$ and then compute its determinant.

\begin{table}[!htb]
$$
\begin{array}{|c|c|c|}
\hline
n & B_n & \det \\
\hline
1&1&\beta\\
2&2& {\beta}^{2} \left( \beta\,\gamma-1 \right) \\
3 & 5 & {\beta}^{5}  \left( \beta\,\gamma-1
 \right) ^{4}  \left( \beta\,\gamma-2 \right) \\
 4 & 15 & {\beta}^{15} \, \gamma \left( \beta\,\gamma-1 \right) ^{14}  \left( \beta\,
\gamma-2 \right) ^{7} \left( \beta\,\gamma-3 \right)   \\
5 & 52 & {\beta}^{52}{\gamma}^{10} \left( \beta\,\gamma-1 \right) ^{51} \left(
\beta\,\gamma-2 \right) ^{36} \left( \beta\,\gamma-3 \right) ^{11}
 \left( \beta\,\gamma-4 \right)
\\
6 & 203 & {\beta}^{203}{\gamma}^{73} \left( \beta\,\gamma-1 \right) ^{202}
 \left( \beta\,\gamma-2 \right) ^{171} \left( \beta\,\gamma-3 \right)
^{81} \left( \beta\,\gamma-4 \right) ^{16} \left( \beta\,\gamma-5
 \right)
 \\
7 & 877 & {\beta}^{877}{\gamma}^{490} \left( \beta\,\gamma-1 \right) ^{876}
 \left( \beta\,\gamma-2 \right) ^{813} \left( \beta\,\gamma-3 \right)
^{512} \left( \beta\,\gamma-4 \right) ^{162} \left( \beta\,\gamma-5
 \right) ^{22} \left( \beta\,\gamma-6 \right)
\\
\hline
\end{array}
$$
\caption{\label{table_det_lin} Determinants of the bilinear form on $A(n)$ for the generating function $Z(T)=\frac{\beta}{1-\gamma T}$.}
\end{table}

For instance, for $n=2$ there are two  partitions, and the Gram matrix is $\begin{pmatrix} \beta^2 & \beta \\ \beta & \beta\gamma \end{pmatrix}$, with  the determinant $\beta^2(\beta\gamma-1)$.

\vspace{0.1in}

{\it Exponents of $\gamma$.}
The  main surprising  feature of the above table is given by terms that  are powers of $\gamma$.
Nonzero powers of $\gamma$ in the above determinant for $\beta/(1-\gamma T)$ start on line four. Adding zero as the power of $\gamma$ on line three gets us the sequence $(0,1, 10, 73,490)$  matching the Sloan sequence A200580, see \url{https://oeis.org/A200580}. It relates to the supercharacters of the finite group  of unipotent upper-triangular matrices over the 2-element field~\cite{BT,DS} and can also be expressed as the combination $-2 B_{n+2} + (n+4)B_{n+1}$ of two Bell numbers.

\vspace{0.1in}

{\it Exponents of determinant factors $\beta\gamma-k$.}
The exponents of factors of $\det \ G_n$ are related to representation theory of the symmetric group.
Namely, for $n$ shown in the table, the difference between the Bell  number $B_n$ (the dimension of $A(n)$ in the generic case) and the exponent $b_{k,n}$ of $\beta\gamma-k$ in $\det \  G_n$, for $k\ge 1$, is equal to the dimension of invariants in
$V^{\otimes n}$,
where $V$ is the natural $k$-dimensional representation of $S_k$:
\begin{equation}\label{eq_B_b}
    B_n - b_{k,n} = \dim \left(V^{\otimes n}\right)^\textrm{inv}.
\end{equation}
We expect this pattern to hold for all $n$, see Figure~\ref{fig_predict}.

The dimension of invariants can be calculated via  characters. Conjugacy classes in $S_k$ are parametrized by cycle types of permutations, which are in a bijective correspondence with Young diagrams with $k$ boxes. The diagram $\lambda = (\lambda_1, \lambda_2, ..., \lambda_{l(\lambda)})$ corresponds to the conjugacy  class of permutations with cycle lengths given by $\lambda$. The trace of such permutation on $V$ equals to $n-l(\lambda)$. Hence, the multiplicity of the trivial representation in $V^{\otimes n}$ is
\[
\dim \left(V^{\otimes n}\right)^\textrm{inv} =  \sum_{|\lambda| = k} \frac{1}{z_\lambda} \left(m_1(\lambda)\right)^n,
\]
where
\[
\lambda = (1^{m_1(\lambda)} 2^{m_2(\lambda)}...).
\]
For $k=1$ representation $V$ of $S_1$ is trivial, and (\ref{eq_B_b}) specializes to
\[
B_n - b_{1,n} = \dim (V^{\otimes n})^\mathrm{inv} = 1, \ \ n>0.
\]

\begin{small}
\begin{table}[!htb]
\begin{tabular}{|c|c|c|c|c|c|c|c|c|c||c|c|c|}
\hline
factor & Group & $A(0)$ & $A(1)$ & $A(2)$ & $A(3)$ & $A(4)$ & $A(5)$ & $A(6)$ & $A(7)$ & $A(8)$ & $A(9)$ & $A(10)$\\
\hline

 $\beta \gamma-1$ & $S_1$ &      0 &  0 &  1 &  4 &  14 &  51 &  202 &  876 &  4139 &  21146 &  115974 \\
  $\beta \gamma-2$ & $S_2$ &      0 &  0 &  0 &  1 &  7 &  36 &  171 &  813 &  4012 &  20891 &  115463 \\
  $\beta \gamma-3$ & $S_3$ &      0 &  0 &  0 &  0 &  1 &  11 &  81 &  512 &  3046 &  17866 &  106133 \\
  $\beta \gamma-4$ & $S_4$ &       0 &  0 &  0 &  0 &  0 &  1 &  16 &  162 &  1345 &  10096 &  72028 \\
  $\beta \gamma-5$ & $S_5$ &         0 &  0 &  0 &  0 &  0 &  0 &  1 &  22 &  295 &  3145 &  29503 \\
    $\beta \gamma-6$ & $S_6$ &         0 &  0 &  0 &  0 &  0 &  0 &  0 &  1 &  29 &  499 &  6676 \\
   $\beta \gamma-7$ & $S_7$ &            0 &  0 &  0 &  0 &  0 &  0 &  0 &  0 &  1 &  37 &  796 \\
   $\beta \gamma-8$ & $S_8$ &             0 &  0 &  0 &  0 &  0 &  0 &  0 &  0 &  0 &  1 &  46 \\
   $\beta \gamma-9$ & $S_9$ &             0 &  0 &  0 &  0 &  0 &  0 &  0 &  0 &  0 &  0 &  1 \\
   $\beta \gamma-10$ & $S_{10}$ &             0 &  0 &  0 &  0 &  0 &  0 &  0 &  0 &  0 &  0 &  0
\\
     \hline
\end{tabular}
\caption{\label{fig_predict} Prediction for the exponents of linear factors are given in the last three columns, for $n=8,9,10$. }
\end{table}
\end{small}

{\it Exponents of $\beta$.} Denote by $b_{0,n}$ the exponent of $\beta$ in $\det\, G_n$, see Table~\ref{table_det_lin}.
The sequence $(1,2,5,15,52,203,877)$ of exponents of $\beta$ in that table is the Bell numbers sequence A000110, see \url{https://oeis.org/A000110}, and we expect this pattern to hold for all $n$, so that $b_{0,n}=B_n$.
Furthermore,
\[
b_{0,n} - b_{1,n} =
1, \ \  n>0.
\]
which matches the data in Table~\ref{table_det_lin} (difference in exponents of $\beta$ and $\beta\gamma-1$ is $1$), so we expect
\[
b_{0,n} = b_{1,n}+1 = B_n, \ \ n>0.
\]

\subsection{Gram determinants for rank two theories }
\label{subsec_denom_two}

\quad

{\bf 1.} Consider the generating function
\begin{equation}\label{eq_gfunc_2}
    Z(T) = \frac{\beta}{(1-\gamma T)^2}, \ \ \beta,\gamma\in \kk^{\ast}.
\end{equation}
This theory has $K=2$ and the handle relation $(x-\gamma)^2=0$. Sequence $\alpha$ for this theory has no abelian realizations.

The space $A(n)$ has a spanning set consisting of viewable surfaces with $\sqcup_n\SS^1$ as the boundary and each component of genus at most one. Elements of the spanning set are set partitions of $n$ carrying labels $0,1$ (the genus of a component), and their count is the generalized (two-colored) Bell number $B_n^{(2)}$, see~\cite{Macd,Stan}, with the generating exponential function
\begin{equation}
    \sum_{n\ge 0} B_n^{(2)}\frac{t^n}{n!} =
   \exp\left(2(\exp(x)-1)\right).
\end{equation}
The first few values of $B_n^{(2)}$ are listed in  Tables~\ref{table_dim_1} and~\ref{fig_table_linear}.

\begin{table}
\begin{tabular}{|c|c|c|}
\hline
  $n$   &  $B^{(2)}_n$ & $\det$ \\
  \hline
  1   &  2 & $-\beta^2$\\
  2 & 6 & $-\beta^{10} \gamma^{12}$ \\
  3 & 22 & $-\beta^{50} \gamma^{66}$ \\
  4 & 94 & $-\beta^{266} \gamma^{376}$ \\
  5 & 454 & $-\beta^{1522} \gamma^{2270}$ \\
  \hline
\end{tabular} \caption{\label{table_dim_1} Two-colored Bell numbers and Gram determinants for the function $Z(T)=\beta/(1-\gamma T)^2$. }
\end{table}

Values of the determinant of the Gram matrix for this spanning set, from computer computations, are shown in Table~\ref{table_dim_1}. For each $n\le 5$ the determinant is the negative product of powers of $\beta$ and $\gamma$.

\vspace{0.1in}

The following guess works for the
powers of $\beta$ in the table: it is the total number of components of all elements of the spanning set, that is, of two-colored partitions of $n$.
More explicitly, the exponent of $\beta$ matches the sequence $B_{n+1}^{(2)}-2B_n^{(2)}$.

Consider a sequence with the terms given by half of the power of $\beta$ in the $n$-th Gram determinant plus $B^{(2)}_n$, see Table~\ref{table_dim_1}.
The sequence of exponents has the form $(1,3,11,47,227,1215,\dots)$, with the first few terms matching the sequence A035009 in the Sloan encyclopedia,
see \url{https://oeis.org/A035009}. Multiplying terms of the latter sequence by $2$ recovers the sequence $(B_n^{(2)})_n$, with the index shifted by one.

The degree of $\gamma$ in the table matches the product of numbers in the first two columns of the table, that is, $n B_n^{(2)}$.

\vspace{0.1in}

{\bf 2.} Table~\ref{table_dim_2} below shows the Gram determinants of the same spanning set of cobordisms with components of genus at most one for the generating function
\begin{equation}
Z(T)=\frac{\beta_0+\beta_1 T}{(1-\gamma T)^2}
\end{equation}
that deforms the function in (\ref{eq_gfunc_2}) without changing the parameter $K=2=\dim A(1)$  of the theory, that is, its rank.

\begin{table}[!htb]
\begin{tabular}{|c|c|c|}
\hline
  $n$   &  $\dim$ & $\det$ \\
  \hline
  1   &  2 & $ -\left(\beta_0 \gamma +\beta_1\right){}^2$\\
  2 & 6 & $-\gamma ^2 \left(\beta_0 \gamma +\beta_1\right){}^{10}$ \\
  3 & 22 & $-\gamma ^{16} \left(\beta_0 \gamma +\beta_1\right){}^{50}$ \\
  4 & 94 & $-\gamma ^{110} \left(\beta_0 \gamma +\beta_1\right){}^{266}$ \\
  5 & 454 & $-\gamma ^{748} \left(\beta_0 \gamma +\beta_1\right){}^{1522}$ \\
  \hline
\end{tabular}
 \caption{\label{table_dim_2} Dimensions and determinants for the function $Z(T)=(\beta_0+\beta_1 T)/(1-\gamma T)^2$. The difference with the previous table is $\beta_0\gamma+\beta_1$ taking place of $\beta$. }
\end{table}

\vspace{0.1in}

{\bf 3.} Gram determinants for the generating function
\begin{equation}
Z(T) = \frac{\beta}{(1-T)(1-\gamma T)}
\end{equation}
are given in Table~\ref{table_dim_3}.

\begin{table}[!htb]
\begin{tabular}{|c|c|}
\hline
  $n$   &  $\det$ \\
  \hline
   1  & $-\gamma\,{\beta}^{2}$ \\
   \hline
   2 & $-{\gamma}^{4}{\beta}^{8} \left( \gamma+\beta-1 \right)  \left( \beta\,
{\gamma}^{2}-\gamma+1 \right)$\\
\hline
3 & $-{\gamma}^{17}{\beta}^{34} \left( \gamma+\beta-1 \right) ^{7} \left(
\beta\,{\gamma}^{2}-\gamma+1 \right) ^{7} \left( \beta+2\,\gamma-2
 \right)  \left( \beta\,{\gamma}^{2}-2\,\gamma+2 \right)
$\\
\hline
4 & $\begin{array} {r@{}l@{}} -{\gamma}^{80}{\beta}^{158} \left( \gamma+\beta-1 \right) ^{42}
 \left( \beta\,{\gamma}^{2}-\gamma+1 \right) ^{42} \left( \beta+2\,
\gamma-2 \right) ^{11} \left( \beta\,{\gamma}^{2}-2\,\gamma+2 \right)
^{11} \cdot \\
\cdot \left( \beta+3\,
\gamma-3 \right) \left( \beta\,{\gamma}^{2}-3\,\gamma+3 \right)   \end{array}
$\\
\hline
5 & $\begin{array} {r@{}l@{}} -{\gamma}^{417}{\beta}^{804} \left( \gamma+\beta-1 \right) ^{251} \left( \beta\,{\gamma}^{2}-\gamma+1 \right) ^{
251} \left( \beta+2\,
\gamma-2 \right) ^{91}  \left( \beta\,{\gamma}^{2}-2\,\gamma+2 \right) ^{91}   \cdot \\ \left( \beta+3\,\gamma-3 \right) ^{16}
\left( \beta\,{\gamma}^{2}
-3\,\gamma+3 \right) ^{16} \left( \beta+4\,\gamma-4 \right) \left( \beta\,{\gamma}^{2}-4\,\gamma+4
 \right)
 \end{array}$
\\
\hline
\end{tabular}
\caption{\label{table_dim_3} Determinants of the bilinear form on $A(n)$ for the generating function $Z(T)=\frac{\beta}{(1-T)(1-\gamma T)}$.}
\end{table}

\vspace{0.1in}

{\bf 4.} Consider the most general generating function, over an algebraically closed $\kk$, for a theory of rank two ($K=2$) :
\begin{equation}
Z(T)=\frac{\beta_0+\beta_1 T}{(1-\gamma_1 T)(1-\gamma_2 T)}.
\end{equation}
Its partial  fraction decomposition is given by
\[
Z(T) =  \frac{\beta_0\gamma_1+\beta_1}{\gamma_1-\gamma_2} \frac{1}{1-\gamma_1 T} +  \frac{\beta_0\gamma_2+\beta_1}{\gamma_2-\gamma_1}\frac{1}{1-\gamma_2 T}.
\]
Table~\ref{table_dim_4} shows values of the Gram determinant for the same spanning set of $A(n)$.

\begin{table}[!htb]
\begin{tabular}{|c|c|c|}
\hline
  $n$   &  $\dim$ & $\det$ \\
  \hline
  1   &  2 & $-\left(\beta_0 \gamma_1+\beta_1\right) \left(\beta_0 \gamma_2+\beta_1\right)$ \\
  \hline
  2 & 6 & $-\left(\beta_0 \gamma_1+\beta_1\right){}^4 \left(\beta_0 \gamma_1^2+\beta_1 \gamma
   _1-\gamma_1+\gamma_2\right) \left(\beta_0 \gamma_2+\beta_1\right){}^4 \left(\beta_0
   \gamma_2^2+\beta_1 \gamma_2-\gamma_2+\gamma_1\right)$ \\
   \hline
  3 & 22 & $ \begin{array} {r@{}l@{}} -\left(\beta_0 \gamma_1+\beta_1\right){}^{17} \left(\beta_0 \gamma_1^2+\beta_1 \gamma
   _1-\gamma_1+\gamma_2\right){}^7 \left(\beta_0 \gamma_1^2+\beta_1 \gamma_1-2 \gamma_1+2
   \gamma_2\right) \\
   \left(\beta_0 \gamma_2+\beta_1\right){}^{17}
   \left(\beta_0 \gamma_2^2+\beta_1
   \gamma_2-\gamma_2+\gamma_1\right){}^7
   \left(\beta_0 \gamma
   _2^2+\beta_1 \gamma_2-2 \gamma_2+2 \gamma_1\right)  \end{array}$ \\
   \hline
  4 & 94 & $ \begin{array} {r@{}l@{}} -\gamma_1 \gamma_2 \left(\beta_0 \gamma_1+\beta_1\right){}^{79} \left(\beta_0 \gamma
   _1^2+\beta_1 \gamma_1-\gamma_1+\gamma_2\right){}^{42} \left(\beta_0 \gamma_1^2+\beta_1
   \gamma_1-2 \gamma_1+2 \gamma_2\right){}^{11}
   \\
   \left(\beta_0 \gamma_1^2+\beta_1 \gamma
   _1-3 \gamma_1+3 \gamma_2\right)
   \left(\beta_0 \gamma_2+\beta_1\right){}^{79}  \left(\beta_0 \gamma_2^2+\beta
   _1 \gamma_2-\gamma_2+\gamma_1\right){}^{42}
   \\
   \left(\beta_0 \gamma
   _2^2+\beta_1 \gamma_2-2 \gamma_2+2 \gamma_1\right){}^{11} \left(\beta
   _0 \gamma_2^2+\beta_1 \gamma_2-3 \gamma_2+3 \gamma_1\right) \end{array}$ \\
   \hline
  5 & 454 & $ \begin{array} {r@{}l@{}} -\gamma_1^{15} \gamma_2^{15} \left(\beta_0 \gamma_1+\beta_1\right){}^{402} \left(\beta_0
   \gamma_1^2+\beta_1 \gamma_1-\gamma_1+\gamma_2\right){}^{251}
   \left(\beta_0 \gamma
   _1^2+\beta_1 \gamma_1-2 \gamma_1+2 \gamma_2\right){}^{91} \\
   \left(\beta_0 \gamma_1^2+\beta
   _1 \gamma_1-3 \gamma_1+3 \gamma_2\right){}^{16}
   \left(\beta_0 \gamma_1^2+\beta_1 \gamma
   _1-4 \gamma_1+4 \gamma_2\right)
   \\
   \left(\beta_0 \gamma_2+\beta_1\right){}^{402}   \left(\beta_0 \gamma_2^2+\beta_1 \gamma
   _2-\gamma_2+\gamma_1\right){}^{251} \left(\beta_0 \gamma_2^2+\beta
   _1 \gamma_2-2 \gamma_2+2 \gamma_1\right){}^{91}
   \\
   \left(\beta_0 \gamma
   _2^2+\beta_1 \gamma_2-3 \gamma_2+3 \gamma_1\right){}^{16}
   \left(\beta
   _0 \gamma_2^2+\beta_1 \gamma_2-4 \gamma_2+4 \gamma_1\right)
  \end{array}$ \\
  \hline
\end{tabular}
 \caption{\label{table_dim_4} Dimensions and determinants for the function $Z(T)=(\beta_0+\beta_1 T)/((1-\gamma_1 T)(1-\gamma_2 T))$.  }
\end{table}

 Powers of $\beta_0\gamma_i+\beta_1$, $i=1,2$ are given by the sequence $(1,4,17,79,402)$, which matches the beginning of the Sloan sequence $A289924$, see  \url{https://oeis.org/A289924}.
 The differences between total dimensions, shown in the middle column, and these exponents, for the simplest factors $\beta_0\gamma_i+\beta_1$,  give the sequence $(1, 2, 5, 15, 52)$, matching  Bell numbers.

 For the factor $\left(\beta_0 \gamma_1^2+\beta_1 \gamma_1-\gamma_1+\gamma_2\right)$
 and its image under the index transposition $\gamma_1\leftrightarrow \gamma_2$  the differences are given by the sequence $(2, 5, 15, 52, 203)$ matching  the shifted sequence of Bell numbers.

 For the factor $\left(\beta_0 \gamma_1^2+\beta_1 \gamma_1-2\gamma_1+2\gamma_2\right)$ and its  transposition under $\gamma_1\leftrightarrow \gamma_2$,
 the sequence of differences is $(2, 6, 21, 83, 363)$, which can be represented as $\frac{1}{2}(B_{n+2}-B_{n+1}+B_{n})$.

We expect that various patterns observed above hold for all $n$.

\section{Polynomial generating functions}
\label{sec_poly_gen}

In this section the categories $\udcobal$ are investigated for polynomial generating functions $Z(T)=Z_{\alpha}(T)$, i.e. $\alpha_i=0$
for $i\gg 0$. The case of the constant function $Z(T)=\beta$ was studied
in Section~\ref{sec_constant_gen}, so we will assume that the degree of the polynomial $Z(T)$ is at least one. There are two main cases:
\begin{itemize}
    \item Function $Z(T)$ is linear. Its tensor envelopes are closely related to the unoriented Brauer category $\Rep(O_t)$.
    \item Function $Z(T)$ has degree at least two. Its series $\alpha$  has no abelian realizations, see condition (3) in Theorem~\ref{necess_thm}.  Experimental data, discussed below, indicates that $\dim A_{\alpha}(n)$ depends only on the degree of the polynomial $Z(T)$.
\end{itemize}

\vspace{0.1in}

Note that a theory of degree $K$ has a skein relation that reduces the $K$-th power of a handle to a linear combination of lower degree powers. Consequently, the state space $A(n)$ has a spanning set given by partitions with an integer between $0$ and $K-1$ (inclusive) assigned to each component.

Generalized Bell numbers $B_n^{(k)}$ count set partition of $n$ together with an assignment of an integer between $0$ and  $k-1$ (inclusive) to each part of  the partition~\cite{Macd, Stan}. Elements of the latter set are in a bijection with  diffeomorphism classes of viewable surfaces with $n$ fixed  boundary components and at most $k-1$ handles on each component.

For a given $k$, generalized Bell numbers have the following exponential generating function:
\begin{equation}
\exp(k(\exp(t) - 1)) = \sum_{n=0}^{\infty} B_n^{(k)} \frac{t^n}{n!}.
\end{equation}

When $Z(T)$ is a polynomial $P(T)$, rank $K$ of the theory is $1+\deg P$.

\subsection{Linear generating function and the unoriented Brauer category}\label{subsec_linear}

Here we consider the case
\begin{equation}
    Z(T)  = \beta_0  + \beta_1 T
\end{equation}
of a linear generating function, with $\beta_0,\beta_1\in \kk$ and $\beta_1\not=0$. Evaluations of connected surfaces for this  $\alpha$ are shown  in Figure~\ref{fig_6_1}. Alternatively, one can treat this theory as  defined  over a ring that  contains  $\kk[\beta_0,\beta_1]$, in which case $\beta_1$ may not   be invertible.

\begin{figure}[ht]
\begin{center}
\begin{tikzpicture}
\draw (0,0) circle(1);
\draw [yscale=0.3] (-1,0) arc[radius=1, start angle=180, end angle=360];
\draw [yscale=0.3, dashed] (-1,0) arc[radius=1, start angle=180, end angle=0];
\node at (1.5,0) {$=\beta_0$};
\draw [yscale=0.7] (4,0) circle(1);
\draw (3.5,0) to [out=290, in=250] (4.5,0);
\draw (3.7,-0.1) to [out=70, in=110] (4.3,-0.1);
\node at (6.,0) {$=\beta_1 \neq 0$};
\draw (8,0) to [out=90, in=180] (8.5,0.4) to [out=0, in=180] (9,0.3) to [out=0, in=180] (9.5,0.4) to [out=0, in=90] (10,0) to [out=270, in=0] (9.5,-0.4) to [out=180, in=0] (9,-0.3) to [out=180, in=0] (8.5,-0.4) to [out=180, in=270] (8,0);
\draw (8.2,0) to [out=290, in=250] (8.7,0);
\draw (8.3,-0.1) to [out=70, in=110] (8.6,-0.1);
\draw (9.2,0) to [out=290, in=250] (9.7,0);
\draw (9.3,-0.1) to [out=70, in=110] (9.6,-0.1);
\node at (10.5,0) {$=0$};
\node at (9,-1) {$g\geq 2$};
\end{tikzpicture}
\caption{Evaluation is zero  beyond  genus one, $\beta_1\not= 0$.}
\label{fig_6_1}
\end{center}
\end{figure}
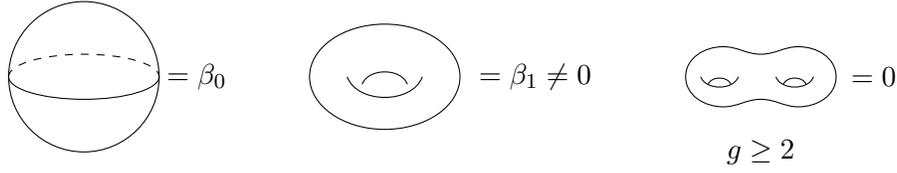

Scaling by  $\lambda=\mu^2$  as in Section~\ref{subsec-scaling}  changes $Z(T)$ to $ \lambda^{-1}\beta_0+\beta_1 T.$ Consequently, if every invertible element  of $\kk$ is  a square, we can reduce  to one of the two  cases:
\begin{equation*}
    (1) \ \  Z_0(T)=\beta_1 T ,\ \ \ \ \ \
    (2) \ \  Z_1(T) = 1 +\beta_1 T.
\end{equation*}

{\it  Skein relations.}
For now, consider the general case, with both  $\beta_0,\beta_1\in \kk$,  $\beta_1\not= 0$. Recall that  a dot on a connected component is a shorthand  for  a handle. Due to the   particular evaluation we are considering, two or more dots on  a component  evaluate the entire  diagram  to  zero, see  Figure~\ref{fig_two_more}.

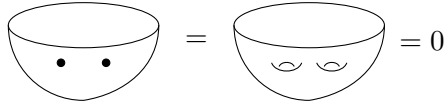
\begin{figure}[ht]
\begin{center}
\begin{tikzpicture}
\draw [yscale=0.3] (1,0) circle(1);
\draw (0,0) to [out=270, in=180] (1,-1) to [out=9, in=270] (2,0);
\draw [fill=black] (0.7,-0.5) circle(0.05);
\draw [fill=black] (1.3,-0.5) circle(0.05);
\node at (2.5,-0.2) {$=$};
\draw [yscale=0.3] (4,0) circle(1);
\draw (3,0) to [out=270, in=180] (4,-1) to [out=9, in=270] (5,0);
\draw (3.5,-0.5) to [out=290, in=250] (3.9,-0.5);
\draw (3.6,-0.55) to [out=70, in=110] (3.8,-0.55);
\draw (4.1,-0.5) to [out=290, in=250] (4.5,-0.5);
\draw (4.2,-0.55) to [out=70, in=110] (4.4,-0.55);
\node at (5.5,-0.2) {$=0$};
\end{tikzpicture}
\caption{\label{fig_two_more} Two or more  dots  on  a connected component  evaluate to zero, so for a spanning set one can reduce to each connected component carrying at most  one dot. Handles in diagrams  can be converted to dots for convenience.}
\end{center}
\end{figure}

The  state space of  two circles $A(2)=A_{\alpha}(2)$ for  this theory was  considered in~\cite[Section 6.2]{Kh1}. It has a spanning set of six vectors, shown in Figure~\ref{fig_6_2_1}: a pair of disks, each  with no dots or  a single dot, and a tube, either dotless or with a dot.
\begin{figure}[ht]
\begin{center}
\includegraphics[scale=0.85]{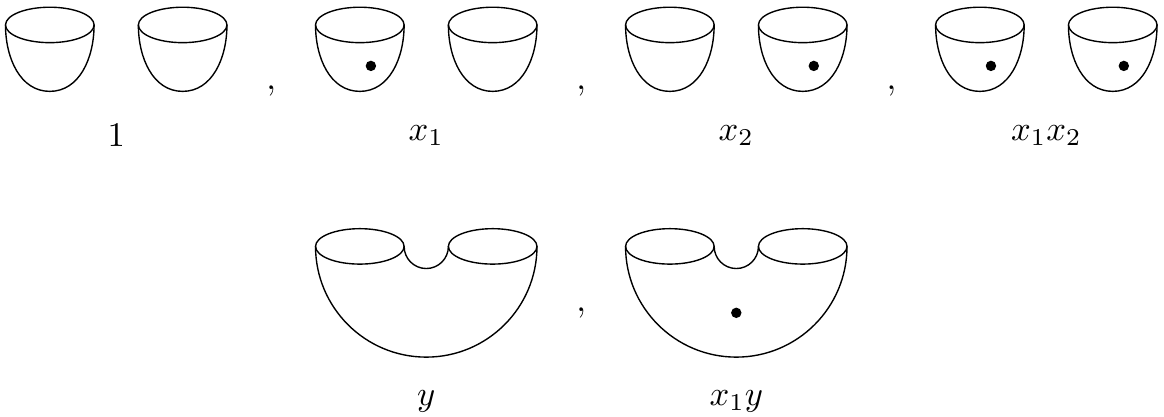}
\caption{ Six vectors that span $A_{\alpha}(2)$. Notation $x_i$ denotes a dot on the $i$-th cup, while $y$ stands for the tube.}
\label{fig_6_2_1}
\end{center}
\end{figure}

One  skein  relation on these  six vectors, shown in Figure~\ref{fig_6_3}, holds for  all values of  the parameters and allows to exclude the vector $x_1x_2$ from the list. To verify this relation observe that the pairing of the both sides with any vector from Figure~\ref{fig_6_2_1} except for the vector 1 is
zero; and it is easy to see that the pairings of both sides with vector 1
are $\beta_1^2$. Thus the difference of the left hand side and the right hand side is a negligible morphism, so it vanishes in categories $\cobal$ and  $\udcobal$.
\begin{figure}[ht]

\begin{center}
\begin{tikzpicture}
[scale=0.7]
\draw [yscale=0.3] (0,0) circle(0.3);
\draw (-0.3,0) to [out=270, in=180] (0,-1) to [out=0, in=270] (0.3,0);
\draw [fill=black] (0,-0.6) circle(0.05);
\draw [yscale=0.3] (2.5,0) circle(0.3);
\draw (2.2,0) to [out=270, in=180] (2.5,-1) to [out=0, in=270] (2.8,0);
\draw [fill=black] (2.5,-0.6) circle(0.05);
\node at (4,-0.5) {$=\beta_1$};
\draw [yscale=0.3] (5,0) circle(0.3);
\draw [yscale=0.3] (7,0) circle(0.3);
\draw (4.7,0) to [out=270, in=180] (6,-1) to [out=0, in=270] (7.3,0);
\draw (5.3,0) to [out=270, in=180] (6,-0.5) to [out=0, in=270] (6.7,0);
\draw [fill=black] (6,-0.7) circle(0.05);
\node at (8,-0.6) {$ \ \ \ \ \ \ \ \ \  \ \ \ \ \ \Longleftrightarrow \ \ \ \ $};
\end{tikzpicture}
\begin{tikzpicture}
[scale=0.7]
\draw [yscale=0.3] (5,0) circle(0.3);
\draw (4.7,0) to [out=270, in=180] (5,-1) to [out=0, in=270] (5.3,0);
\draw [fill=black] (5,-0.6) circle(0.05);
\draw [yscale=0.3] (7.5,0) circle(0.3);
\draw (7.2,0) to [out=270, in=180] (7.5,-1) to [out=0, in=270] (7.8,0);
\draw [fill=black] (7.5,-0.6) circle(0.05);
\node at (3.5,-0.5) {$=\beta_1^{-1}$};
\draw [yscale=0.3] (0,0) circle(0.3);
\draw [yscale=0.3] (2,0) circle(0.3);
\draw (-0.3,0) to [out=270, in=180] (1,-1) to [out=0, in=270] (2.3,0);
\draw (0.3,0) to [out=270, in=180] (1,-0.5) to [out=0, in=270] (1.7,0);
\draw [fill=black] (1,-0.7) circle(0.05);
\end{tikzpicture}
\caption{\label{fig_6_3} Skein relation on diagrams with two boundary circles. One  can either exclude two cups with dots or  exclude a tube with dots. If not  working over a field, only transformation  on the left may be allowed.}
\end{center}
\end{figure}

The  Gram determinant for the remaining five vectors is $\beta_1^6(\beta_1-2)$. In  particular, if $\beta_1\not=2$, these five vectors constitute a basis of  $A(2)$  and $\dim A(2) = 5.$

The case $\beta_1=2$  gives a multiplicative theory discussed at the  end  of~\cite[Section 2.5]{Kh1}. Note that when $\beta_1=2$, we have  $\nchar\kk\not=2$ since $\beta_1\not=0$.

Let us now assume  that  $\beta_1\not= 2$.
Figure~\ref{fig_6_3} relation allows to reduce a component of genus one with more than one boundary circle to components of genus one with just  one boundary circle each.

Consider the state space  $A(3)=A_{\alpha}(3)$ for three circles. These reductions give a spanning set of cobordisms with each component of genus at most one (equivalently, of genus $0$ with at most one dot), and genus one components with only one boundary circle each. Furthermore, a direct computation shows that a genus zero component with three boundary circles is a linear combination of other cobordisms from this spanning set, see Figure~\ref{fig_6_4}.

\begin{figure}[ht]
\begin{center}
\begin{tikzpicture}
[scale=0.4]
\draw [yscale=0.3] (0,0) circle(1);
\draw [yscale=0.3] (3,0) circle(1);
\draw [yscale=0.3] (6,0) circle(1);
\draw (-1,0) to [out=270, in=180] (3,-2) to [out=0, in=270] (7,0);
\draw (1,0) to [out=270, in=180] (1.5,-1) to [out=0, in=270] (2,0);
\draw (4,0) to [out=270, in=180] (4.5,-1) to [out=0, in=270] (5,0);

\node[black, scale=1] at (9,-1) {$=\beta_0 \beta_1^{-3}$};

\draw [yscale=0.3] (12,0) circle(1);
\draw (11,0) to [out=270, in=180] (12,-2) to [out=0, in=270] (13,0);
\draw [fill=black] (12,-1) circle(0.05);
\draw [yscale=0.3] (15,0) circle(1);
\draw (14,0) to [out=270, in=180] (15,-2) to [out=0, in=270] (16,0);
\draw [fill=black] (15,-1) circle(0.05);
\draw [yscale=0.3] (18,0) circle(1);
\draw (17,0) to [out=270, in=180] (18,-2) to [out=0, in=270] (19,0);
\draw [fill=black] (18,-1) circle(0.05);

\node[black, scale=1] at (20.5,-1) {$-\beta_1^{-2}$};
\draw [thick] (22.5,0.5) to [out=250, in=110] (22.5,-2.5);
\draw [yscale=0.3] (24,0) circle(1);
\draw (23,0) to [out=270, in=180] (24,-2) to [out=0, in=270] (25,0);
\draw [fill=black] (24,-1) circle(0.05);
\draw [yscale=0.3] (27,0) circle(1);
\draw (26,0) to [out=270, in=180] (27,-2) to [out=0, in=270] (28,0);
\draw [fill=black] (27,-1) circle(0.05);
\draw [yscale=0.3] (30,0) circle(1);
\draw (29,0) to [out=270, in=180] (30,-2) to [out=0, in=270] (31,0);
\node at (32,-1) {$+$};
\draw[->, thick] (33,-0.3) arc(150:-150:0.8);
\draw [thick] (35,0.5) to [out=290, in=70] (35,-2.5);
\node [red] at (29,-3) {3 terms};
\node[black] at (37,-1) {$+$};

\node[black] at (12.5,-6) {$+$};

\node[black] at (14.,-6) {$\beta_1^{-1}$};

\draw [thick] (15.5,-4.5) to [out=250, in=110] (15.5,-7.5);
\draw [yscale=0.3] (17,-16.67) circle(1);
\draw [yscale=0.3] (20,-16.67) circle(1);
\draw (16,-5) to [out=270, in=180] (18.5,-7) to [out=0, in=270] (21,-5);
\draw (18,-5) to [out=270, in=180] (18.5,-6) to [out=0, in=270] (19,-5);
\draw [yscale=0.3] (23,-16.67) circle(1);
\draw (22,-5) to [out=270, in=180] (23,-7) to [out=0, in=270] (24,-5);
\node at (25,-6) {$+$};
\draw[->, thick] (26,-5.3) arc(150:-150:0.8);
\draw [thick] (28,-4.5) to [out=290, in=70] (28,-7.5);
\draw [fill=black] (23,-6) circle(0.05);
\node [red] at (22,-8) {3 terms};

\end{tikzpicture}
\caption{\label{fig_6_4} Sphere with 3 holes as a linear combination of seven terms. The expression on the RHS is invariant under the permutation action of $S_3$ and terms permuted under the action are grouped together into sets of three. Terms in the first group of 3 differ in dot placement. }
\end{center}
\end{figure}

Relation in Figure~\ref{fig_6_4} implies the relation in Figure~\ref{fig_6_3} by capping off one of the three boundary circles with a one-holed torus. Inductive application of these relations, together with the one in Figure~\ref{fig_two_more}, allows to reduce any connected component with three or more boundary circles to a linear combination of surfaces where
\begin{itemize}
\item each component has at most one boundary circle,
\item all components have genus zero or one,
\item  each genus one component bounds one circle.
\end{itemize}

\begin{prop}\label{span1}
The space $A(n)$ has a spanning set $\mathscr{A}(n)$ that consists of viewable cobordisms  with $n$ boundary circles, with each connected component having one or two boundary circles, all genus one components with one boundary circle and no components of genus two or higher. The cardinality $a_n$ of the above set $\mathscr{A}(n)$ of cobordisms satisfies the following recurrent relation
\begin{equation} \label{eq_recurrent_rel}
a_n = 2a_{n-1}+(n-1)a_{n-2},
\end{equation}
and has the following generating function
\[
\sum_{n=0}^\infty \frac{a_n}{n!} t^n = \exp\left(
2t+\frac{t^2}{2}
\right)
\]
\end{prop}

\begin{proof}
Using the relations in Figures~\ref{fig_6_3} and
~\ref{fig_6_4} we can express any cobordism as a linear combination of the elements of $\mathscr{A}(n)$. Let us show that the relation
\eqref{eq_recurrent_rel} holds. Consider the first boundary circle
of a cobordism from $\mathscr{A}(n)$. If it is the only
boundary circle of its connected component $C$  then the cobordism is a union
of this component of genus zero or one and an element of $\mathscr{A}(n-1)$, giving $2a_{n-1}$ possibilities.
Otherwise, $C$ has genus zero
and two boundary circles. There are $n-1$ options for the second
circle, giving $(n-1)a_{n-2}$ possibilities in this case and proving
\eqref{eq_recurrent_rel}. The derivation of the generating
function from the recurrence relation is standard.
\end{proof}

\begin{corollary} \label{ineq cor}
 $\dim A(n)\le a_n$ for $a_n$ as above, for any field $\kk$ and $\beta_0\in \kk,\beta_1\in \kk^{\ast}$.
\end{corollary}

\begin{remark}
The sequence $(a_n)_{n\ge 0}$ is the Sloan sequence A005425, see~\url{https://oeis.org/A005425},
\[ (a_0,a_1,a_2,\dots ) =
(1,1,2,5,14,43,142,499,1850,7193,...). \]
\end{remark}

Numerical data for $\dim A(n)$ for generic values of $\beta_0,\beta_1$ and
the Gram determinant for the spanning set $\mathscr{A}(n)$ with
$n\le 7$ is given in Table~\ref{fig_table_linear}.
The third column shows $\dim A(n)$ for generic values of $\beta_0,\beta_1$. The last column
shows the values of the  Gram determinant for the set of vectors in the above spanning set $\mathscr{A}(n)$. Observe that the determinants do not depend on $\beta_0$ and can vanish only when $\beta_1$ is in the image of $\Z$ in $\kk$. Note that non-vanishing of the determinants implies that $\dim A(n)=a_n$ for $n\le 7$ and generic $\beta$'s. We are going to show that the same is true
for any $n$.

\begin{table}[!htb]
$$
\begin{array}{|c|c|c|c|}
\hline
n & B_n^{(2)} & \dim A(n) & \det \\
\hline
1& 2 & 2 &-\beta_1^2\\
2 & 6 & 5 & \left(\beta_1-2\right) \beta_1^8  \\
3 & 22 & 14 & -\left(\beta_1-2\right){}^6 \beta_1^{30} \\
4 & 94 & 43 & \left(\beta_1-3\right){}^2 \left(\beta_1-2\right){}^{27} \beta_1^{113} \\
5 & 454 & 142 & -\left(\beta_1-3\right){}^{20} \left(\beta_1-2\right){}^{110} \beta_1^{440}
\\
6 & 2430 & 499 & \left(\beta_1-4\right){}^5 \left(\beta_1-3\right){}^{134} \left(\beta_1-2\right){}^{435}
   \beta_1^{1774} \left(\beta_1+2\right)
\\
7 & 14214 & 1850 & -\left(\beta_1-4\right){}^{70} \left(\beta_1-3\right){}^{756} \left(\beta
   _1-2\right){}^{1722} \beta_1^{7406} \left(\beta_1+2\right){}^{14}
\\
\hline
\end{array}
$$
\caption{\label{fig_table_linear} Determinants of the bilinear form on $A(n)$ for the generating function $Z(T)=\beta_0 + \beta_1 T$. Notice the appearance of the term $\beta_1+2$ in the last two lines.}
\end{table}

Consider the Deligne orthogonal category $\Rep(O_t)$, $t\in \kk$, see e.g. \cite[Section 9]{D1}. Let $V\in \Rep(O_t)$ be the generating object corresponding to one element set in \cite[Definition 9.2]{D1}. By definition, we have
\begin{eqnarray*}
\inv\left( V^{\otimes 2n} \right) & := & \ \dim \Hom_{\Rep(O_t)}(\be,V^{\otimes 2n}) =(2n-1)!!, \\
\inv\left( V^{\otimes 2n+1} \right) & = &
\dim \Hom_{\Rep(O_t)}(\be,V^{\otimes {2n+1}}) =0,
\end{eqnarray*}
where $(2n-1)!!=(2n-1)\dots 3 \cdot 1$ is the odd factorial.

\begin{prop} \label{prop 7.3}
$a_n$ is the  dimension of invariants of the $n$-th tensor power of the object $\be\oplus \be \oplus V\in \Rep(O_t)$,
\begin{equation}
    a_n \ = \ \dim \mathrm{Hom}_{\Rep(O_t)}(\be,(\be^{2}\oplus V)^{\otimes    n}).
\end{equation}
\end{prop}

\begin{proof}
Let us compute the exponential generating function of dimensions of invariants
\begin{multline}
\sum_{n\geq 0} \dim \left(\inv \left((\be^{2}\oplus V)^{\otimes    n}\right)\right)  \frac{u^n}{n!} = \dim (\inv \exp(2u+Vu) )
= \dim (\inv \exp(2u)) \exp(Vu) = \\
 = \exp(2u) \sum_{n\geq 0} \frac{u^n}{n!} \dim (\inv \  V^{\otimes n}) = \exp(2u) \sum_{n \geq 0} \frac{u^{2n}}{{2n}!} (2n-1)!! = \exp\left(2u+\frac{u^2}{2}\right).
\end{multline}
Here for a $G$-representation $V$ we denote by
\[
\exp(V u) = \sum_{n\geq 0} \frac{u^n}{n!} V^{\otimes n} \in K(G)[[u]],
\]
where $K(G)$ is the representation ring (tensored with $\Q$) of the group $G$.
\end{proof}

Proposition \ref{prop 7.3} motivates the following construction.
Let $A\in \Rep(O_t)$ be the commutative Frobenius algebra obtained from a
symmetrically self-dual object $V\in \Rep(O_t)$ by the construction in Example
\ref{ortho example} in Section~\ref{subset_frob_obj}. The generating function of the $\alpha-$invariant of algebra $A$ is $\alpha_0+\alpha_1T$ where $\alpha_0$ can be chosen arbitrarily and $\alpha_1=t+2$. It is also easy to see that the
skein relation $x^2=0$ holds for the handle morphism of the algebra $A$.
Thus by the universal property
from Section \ref{subsec_universal} there is a symmetric tensor functor
$F_\alpha: \dcobal \to \Rep(O_{\beta_1-2})$ sending the circle object to $A=\be^{2}\oplus V$ (we assume here that $\beta_1\ne 0$ since the skein
relation is different in the case $\beta_1=0$).
The following simple result is crucial:

\begin{proposition} \label{full prop}
Assume $\beta_1\ne 0$. Then the functor
$F_\alpha$ is full and essentially surjective.
\end{proposition}

\begin{proof} By definition, the image of the functor $F_\alpha$ contains
$A=\be^2\oplus V$. This implies the essential surjectivity, as any object
of $\Rep(O_t)$ is a direct summand of a direct sum of tensor powers of $V$.

Let us show that the functor $F_\alpha$ is full.
Since the object $A$ is self-dual it is sufficient to show that
any morphism from $\Hom(\be,A^{\otimes n})$ is in the image of the functor $F_\alpha$. Using the decomposition
$$A^{\otimes n}=(\be^2\oplus V)^{\otimes n}=\bigoplus_{S\subset [1,\ldots ,n]}
\bigotimes_{i=1}^nX_i^S$$
where $X_i^S=\be^2$ if $i\in S$ and $X_i^S=V$ if $i\not \in S$, we see that
the space $\Hom(\be,A^{\otimes n})$ is spanned by the tensor products
of morphisms from $\Hom(\be,A)$ and the pairing $\be \to V\otimes V\to A\otimes A$. Thus it is sufficient to check that $F_\alpha$ is surjective on
$\Hom(\be,A^{\otimes n})$ for $n=1,2$. This is clear for $n=1$ since
the space $\Hom(\be,A)$ is two dimensional and the image of the functor
$F_\alpha$ is at least two dimensional (by the first row of table ~\ref{fig_table_linear}), as $F_\alpha$ does not annihilate
non-negligible morphisms. The same argument (based on the second row of table ~\ref{fig_table_linear}) works for $n=2$ provided that $\beta_1\ne 2$.
Finally in the case $n=2$, $\beta_1=2$ and $\nchar \kk \ne 2$ one verifies
by an explicit computation that the unique up to scaling negligible morphism in $\Hom(\be,A^{\otimes 2})$ is not annihilated by $F_\alpha$.
\end{proof}

\begin{remark} One verifies that the functor $F_\alpha$ annihilates
relations in Figures ~\ref{fig_6_3} and ~\ref{fig_6_4}. Let $\dcobalb$
be the quotient of $\dcobal$ by these relations. It is clear that the inequality
from Corollary \ref{ineq cor} holds in the category $\dcobalb$. Thus
Proposition~\ref{full prop} implies that
the functor $\dcobalb \to \Rep(O_{\beta_1-2})$ is also faithful.
Hence, there is an equivalence of tensor categories
$\dcobalb \simeq \Rep(O_{\beta_1-2})$.
\end{remark}

Combining Propositions~\ref{full prop} and~\ref{prop_assume} results
the following:

\begin{theorem} \label{linear quot}
Assume $Z_{\alpha}=\beta_0+\beta_1 T$ with $\beta_1\ne 0$.
The functor $F_\alpha$ induces an equivalence of tensor categories
$\udcobal \simeq \uRep (O_{\beta_1-2})$, where
$\uRep(O_{\beta_1-2})$ is the gligible quotient of the Deligne category $\Rep(O_{\beta_1-2})$.
\end{theorem}

Here is a special case. Assume that $\nchar \kk=0$. By a theorem
of H.~Wenzl (see e.g. \cite[Th\'eor\`em 9.7]{D1}) we have $\Rep(O_{\beta_1-2})=
\uRep(O_{\beta_1-2})$ when $\beta_1\not \in \Z$.
It follows from Proposition \ref{prop 7.3} that in this case $\dim A(n)=a_n$, and the set $\mathscr{A}(n)$ is linearly independent. We have the following implications for
the determinant $\det_n$ of the bilinear form on $A(n)$:

\begin{proposition}
The polynomial $\det_n$ is nonzero and depends only on
$\beta_1$ (and not on $\beta_0$); moreover its irreducible factors are of the form $\beta_1-s$, $s\in \Z$.
\end{proposition}

\begin{proof} It is clear that $\det_n$
is a polynomial in variables $\beta_0$ and
$\beta_1$ with integer coefficients. As explained above this polynomial can vanish only when $\beta_1\in \Z$, so that $\det_n$ does not depend on $\beta_0$,  by elementary algebraic geometry.
\end{proof}

In the case $\nchar \kk=0$ and $t\in \Z$, the gligible quotients $\uRep(O_{t})$ are computed in \cite[Th\'eor\`em 9.6]{D1}.
Recall that
$$\uRep(O_{t})\cong\Rep(G,\varepsilon),$$
where $G$ is one of the super groups $O(n)$ (if $t=n\ge 0$), $Sp(2m)$
(if $t=-2m$ is negative and even), $OSp(1,2m)$ (if $t=1-2m$ is negative and odd)
and $\varepsilon \in G$ is a suitable involution. Thus we get the following
examples illustrating Theorem \ref{linear quot}:

\begin{example}\label{linear Z example}
($\nchar \kk=0$)

(1) Assume $Z_\alpha =\beta_0+2T$. Then $\udcobal$ is
the category $Vec$ and the circle object corresponds to the Frobenius algebra
$\kk [x]/(x^2)$ with $\epsilon(1)=\beta_0$ and $\epsilon(x)=1$.
Note that in this case $\dim A(n)=2^n$.

(2) Assume $Z_\alpha =\beta_0+3T$. Then $\udcobal$ is
the category $\Rep(\Z/2)$ and the circle object corresponds to the Frobenius algebra
$A=\kk [x]/(x^3)$ with $\epsilon(1)=\beta_0$, $\epsilon(x)=0$, $\epsilon(x^2)=1$,
and the group $\Z/2$ acting on $A$ via $x\mapsto -x$.
Thus the character of the $\Z/2-$representation $A$ takes values 3 and 1 on the elements $0, 1\in \Z/2$, and $\dim A(n)=\frac{3^n+1}2$.

(3) Assume $Z_\alpha =\beta_0-2T$. Then $\udcobal$ is
the category $\Rep(Sp(4))$ (with the modified commutativity constraint), and the circle object corresponds to the Frobenius algebra
$H^\ast(\Sigma_2,\kk)$, where $\Sigma_2$ is a oriented closed connected surface of genus two. Here $Sp(4)$ acts trivially on $H^{even}(\Sigma_2,\kk)$
and via the natural representation on $H^1(\Sigma_2,\kk)$. The commutativity
constraint in $\Rep(Sp(4),\varepsilon)$ is modified in a way making the natural
representation into an odd vector space, so the algebra $H^\ast(\Sigma_2,\kk)$ is commutative in the category $\Rep(Sp(4),\varepsilon)$.
\end{example}

We will see later (Proposition \ref{poly lead prop}) that the leading coefficient of the polynomial $\det_n$ is $\pm 1$. It follows
that $\det_n$ is nonzero even if $\nchar \kk >0$; moreover its
roots lie in the prime subfield of $\kk$. Thus there is
an equivalence $\udcobal \cong \Rep(O_{\beta_1-2})$ provided that $\beta_1$
is not an element of the prime subfield.

\begin{corollary}
Assume $\nchar \kk>0$ and $t$ is not in the prime subfield of $\kk$. Then the category $\Rep(O_t)$ is non-degenerate (i.e.
has no nonzero negligible morphisms) and non-semisimple.
\end{corollary}

In the case $\nchar \kk>0$ and $\beta_1$ is in the prime subfield
we expect that the categories $\udcobal$ are equivalent to the
fusion categories associated with super groups $(G,\varepsilon)$
as above (i.e. gligible quotients of suitable tilting modules
categories). In particular, the categories $\udcobal$ should
have finitely many simple objects up to isomorphism.

We discuss now the multiplicities of the roots of polynomials
$\det_n$. Here are some patterns that can be observed in
Table \ref{fig_table_linear}:
\begin{itemize}
    \item The differences $\dim A(n)-u_n$, where  $u_n$ is the exponent of $\beta_1-2$, are given by powers of two: $(2,4,8,16,32,64)$.
    \item The differences $\dim A(n) - w_n$, where $w_n$ is the exponent of $\beta_1-3$, are given by $(2,5,14,41,122,365)$. These exponents match the sequence $(3^n+1)/2.$
\end{itemize}

Comparing this patterns with Example \ref{linear Z example} (1) and (2) we arrive at the following

\begin{conjecture}\label{lin conj}
Let $s\ne 0$ be an integer.
The exponent of the factor $\beta_1-s$ in the
polynomial $\det_n$ is given by
$$a_n-\dim A_{\alpha(s)} (n)$$
where $\alpha(s)=(\beta_0,s,0,0,\dots)$, so that the generating function $Z_{\alpha(s)}(T)=\beta_0+sT$ (for
arbitrary $\beta_0$ and $\nchar \kk=0$).
\end{conjecture}

\begin{remark}
By definition, the bilinear form on the space $\kk \mathscr{A}(n)$ has a null space of
dimension $a_n-\dim A_{\alpha(s)} (n)$. Thus, a standard
argument (see e.g. \cite[Lemma 8.4]{KacR}) implies that the exponent
of the factor $\beta_1-s$ in $\det_n$ is greater or equal to
$a_n-\dim A_{\alpha(s)} (n)$.
\end{remark}

Note that, according to Theorem \ref{linear quot}, the dimensions
$\dim A_{\alpha(s)} (n)$ are given by the dimensions of invariants
of the (super) groups $G=O(k), Sp(2k), OSp(1|2k)$ in the representation
$(\be^2\oplus V)^{\otimes n}$ where $V$ is the defining representation of $G$ and $s=k+2,$ $2-2k,$ $3-2k$, respectively.
We tabulated the exponents predicted by Conjecture \ref{lin conj} in Table \ref{pred table}. Here are some observations
about Tables \ref{pred table} and \ref{fig_table_linear}:

\begin{itemize}
    \item The exponents for $\beta_1-1$ and $\beta_1-5$
 coincide. The same applies to the exponents for $\beta_1+1$ and $\beta_1-7$ and to the exponents for $\beta_1+3$ and $\beta_1-9$ etc. This is explained by the coincidence of the multiplicities for tensor products for $OSp(1,2k)$ and
$O(2k+1)$, see~\cite{RS}.
\item The irreducible factors of $\det_{2n}$ and $\det_{2n+1}$
coincide for any $n\ge 0$.
\item The irreducible factors which appear in $\det_{2n}$ and do not appear at $\det_{2n-1}$ are

$\beta_1-n-1$ (for $n\ge 1$), $\beta_1+2n-4$ (for $n\ge 3$),
and $\beta_1+n-5$ (for even $n\ge 4$).
\end{itemize}

Conjecture \ref{lin conj} does not predict the exponent of the factor
$\beta_1$ in $\det_n$. We propose the following

\begin{conjecture}\label{lin conj2}
The exponent of the factor $\beta_1$ in $\det_n$ is given by
$$2na_{n-1}+a_n-c_{n+1}$$
where $c_n$ is the Catalan number.
\end{conjecture}

It would be interesting to find a categorical interpretation of this conjecture.
The numerical data for it are given in the last row
of Table \ref{pred table}. Term $na_{n-1}$ in the above conjecture is the total number of connected components of genus one in the set of cobordisms $\mathscr{A}(n)$.

\vspace{0.1in}

\begin{tiny}
\begin{table}[!htb]
\begin{tabular}{|c|c|c|c|c|c|c|c|c||c|c|c|c|c|}
\hline
factor & Group  & $A(1)$ & $A(2)$ & $A(3)$ & $A(4)$ & $A(5)$ & $A(6)$ & $A(7)$ & $A(8)$ & $A(9)$ & $A(10)$ & $A(11)$ & $A(12)$ \\
\hline

$\beta_1-2$ & $O(0)$ & 0& 1& 6& 27& 110& 435& 1722& 6937& 28674& 122085& 536030& 2426259 \\
$\beta_1-3$ & $O(1)$ & 0& 0& 0& 2& 20& 134& 756& 3912& 19344& 93584& 449504& 2164634 \\

$\beta_1-4$ & $O(2)$ &   0 &  0 &  0 &  0 &  0 &  5 &  70 &  630 &  4620 &  30219 &  184338 &  1076229 \\

$\beta_1-5$ & $O(3)$ &    0 &  0 &  0 &  0 &  0 &  0 &  0 &  14 &  252 &  2862 &  26004 &  207350\\

$\beta_1-6$ & $O(4)$ &   0& 0& 0& 0& 0& 0& 0& 0& 0& 42& 924& 12705\\

$\beta_1-7$ & $O(5)$  &  0 &  0 &  0 &  0 &  0 &  0 &  0 &  0 &  0 &  0 &  0 &  132\\

$\beta_1-8$ & $O(6)$ & 0& 0& 0& 0& 0& 0& 0& 0& 0& 0& 0& 0 \\
\hline
     $\beta_1+2$ & $Sp(4)$   &  0 &  0 &  0 &  0 &  0 &  1 &  14 &  133 &  1050 &  7491 &  50226 &  323796 \\
     $\beta_1+4$ & $Sp(6)$    &  0 &  0 &  0 &  0 &  0 &  0 &  0 &  1 &  18 &  216 &  2112 &  18370 \\
     $\beta_1+6$ & $Sp(8)$  &  0 &  0 &  0 &  0 &  0 &  0 &  0 &  0 &  0 &  1 &  22 &  319 \\

    $\beta_1+8$ & $Sp(10)$  &  0 &  0 &  0 &  0 &  0 &  0 &  0 &  0 &  0 &  0 &  0 &  1 \\
     \hline
     $\beta_1-1$ & $OSp(1,2)$   &  0 &  0 &  0 &  0 &  0 &  0 &  0 &  14 &  252 &  2862 &  26004 &  207350\\
     $\beta_1+1$ & $OSp(1,4)$    &  0 &  0 &  0 &  0 &  0 &  0 &  0 &  0 &  0 &  0 &  0 &  132\\
     $\beta_1+3$ & $OSp(1,6)$ &   0 &  0 &  0 &  0 &  0 &  0 &  0 &  0 &  0 &  0 &  0 &  0 \\
     \hline
     $\beta_1$ & ?  &  2 &  8 &  30 &  113 &  440 &  1774 &  7406 &  31931 &  141864 &  648043 &  3038464 &  14601327 \\
     \hline
\end{tabular}
\caption{\label{pred table}Prediction for the exponents of linear factors. Column for $A(0)$ is not shown, it contains zeros only. For a naive interpolation, the question mark in the bottom row on the left may be replaced by $Sp(2)$, but  Conjecture~\ref{lin conj2} gives more complicated rules for the entries of this row than Conjecture~\ref{lin conj} (if extended to $s=0$) that should governs the other rows of the table. }
\end{table}
\end{tiny}

Table \ref{fig_table_linear} suggests that the leading coefficient
of the polynomial $\det_n$ is $(-1)^n$. Using this together with
Conjectures \ref{lin conj} and \ref{lin conj2} we can predict
the polynomials $\det_n$. For example, the prediction for $\det_8$ is
$$(\beta_1-5)^{14}(\beta_1-4)^{630}(\beta_1-3)^{3912}(\beta_1-2)^{6937}(\beta_1-1)^{14}\beta_1^{31931}(\beta_1+2)^{133}(\beta_1+4).$$
One verifies that the degree of this polynomial agrees with
Corollary \ref{lin deg cor} below.

\vspace{0.1in}

\vspace{0.1in}

\subsection{Polynomials of degree two and three}
\label{subsec_deg_2_3}
Consider a polynomial generating function of degree  two,
\begin{equation}
Z(T) = \beta_0  + \beta_1 T + \beta_2 T^2
\end{equation}
The dimension of $A(n)$ is bounded from above by $B_n^{(3)}$, since all surfaces with a component of genera at least 3 are in the kernel of the bilinear form.
However the computation shows that the actual dimension is strictly less than $B_n^{(3)}$ starting from $n=2$, see the data in Table~\ref{fig_table_T2} for the quadratic $Z(T)$.

\begin{table}
\begin{tabular}{|c|c|c|c|}
\hline
  $n$   & $B_n^{(3)}$ & $\dim$ & $\det$ \\
  \hline
  0  &  1  & 1 & $1$\\
  1   &  3 & 3& $-\beta_2^{3}$\\
  2 & 12 & 11& $-\beta_2^{20}$ \\
  3 & 57 & 46 & $\beta_2^{118}$\\
  4 & 309 & 213 & $\beta_2^{696}$ \\
  5 & 1866 & 1073 & $-\beta_2^{4225}$\\
  \hline
\end{tabular}
\caption{\label{fig_table_T2} Computation of  dimensions and the determinant for $Z(T)=\beta_0 + \beta_1 T + \beta_2 T^2$.}
\end{table}

Table~\ref{fig_table_T3} shows the determinants for a generic polynomial of degree three.

\begin{table}
\begin{tabular}{|c|c|c|c|}
\hline
  $n$   & $B_n^{(4)}$ & $\dim$ & $\det$\\
  \hline
  0  &  1  & 1 & 1 \\
  1   &  4 & 4 & $\beta_3^{4}$\\
  2 & 20 & 19 & $-\beta_3^{35}$\\
  3 & 116 & 102 & $-\beta_3^{266}$\\
  4 & 756 & 604 & $\beta_3^{2007}$\\
  5 & 5428 & 3884 & $\beta_3^{15540}$\\
  \hline
\end{tabular}
\caption{\label{fig_table_T3} Computation of  dimensions and the determinant for $Z(T)=\beta_0 + \beta_1 T + \beta_2 T^2+\beta_3 T^3$.}
\end{table}

\vspace{0.1in}

\subsection{Polynomials of arbitrary degree}
\label{subsec_any_deg}

Now consider the case of an arbitrary polynomial generating function:
\[
Z = \beta_0+\beta_1 T + ... + \beta_m T^m, \ \ m \geq 1.
\]
Let $\mathscr{A}^m(n)$ be the set of viewable cobordisms with $n$ boundary circles such that for each component $S$ of genus $g$ with $\ell$ boundary circles the following inequality holds:
\begin{equation}\label{eq_inequality}
g+\ell \leq m+1.
\end{equation}

Note that $\mathscr{A}^1(n)$ is precisely the set $\mathscr{A}(n)$ from Proposition~\ref{span1}.
Let us consider the matrix of the bilinear form on the space
$A(n)$ computed at the elements of the set $\mathscr{A}^m(n)$,
and let $\det_n^{(m)}$ denote its determinant.
It is clear that $\det_n^{(m)}$ is a polynomial in variables
$\beta_0, \beta_1, \ldots, \beta_m$. In the next Proposition
we are going to compute the leading term of this polynomial.
Let $d_n^{(m)}$ be the total number of connected components of all elements of
the set $\mathscr{A}^m(n)$.

\begin{proposition}\label{poly lead prop}
The polynomial $\det_n^{(m)}$ is of the form
\[
\pm (\beta_m)^{d_n^{(m)}}+\mbox{lower terms}
\]
where each lower term monomial has either
less than $d_n^{(m)}$ factors or precisely
$d_n^{(m)}$ factors but involves some
$\beta_i$ with $i<m$.
\end{proposition}

\begin{proof}
The expansion of the determinant $\det_n^{(m)}$ is a sum
over all permutations $\pi$ of the set
$\mathscr{A}^m(n)$ of terms
\[
t_\pi=\pm \prod_{a\in \mathscr{A}^m(n)}b_{a,\pi(a)}
\]
where $b_{a,\pi(a)}$ is the pairing of $a$
and $\pi(a)$, hence some monomial in $\beta_i$'s. The number of factors in the monomial $b_{a,\pi(a)}$ is precisely the number of connected components of the surface
obtained from $a$ and $\pi(a)$ by gluing along the boundary. Thus it is clear that the number of factors is less or equal to the number of connected components of $a$. Moreover, we have equality only if the partition of the boundary circles determined by the connected components of $\pi(a)$ is a refinement of the partition determined by $a$.

Thus, the total number of factors
in $t_\pi$ is less or equal than the total
number of connected components of all elements $a\in \mathscr{A}^m(n)$, and
every monomial in the polynomial $\det_n^{(m)}$ has $\le d_n^{(m)}$ factors.
The term $t_\pi$ has precisely $d_n^{(m)}$ factors if and only if the permutation
$\pi$ has the following property:

(*) for any $a$ the partition of the boundary circles determined by the connected components of $\pi(a)$ is a refinement of the partition determined by $a$.

Note that
there exists $r>0$ such that $\pi^r(a)=a$. Thus the condition (*) is equivalent to the following property:

(**) for any $a$ the partition of the boundary circles determined by the connected components of $\pi(a)$ coincides the partition determined by $a$.

Now let $\pi_0$ be the following permutation:

$\pi_0(a)$ is obtained from $a$ by replacing each connected component of genus $g$ with $l$ boundary circles by the connected component of genus $g'=m+1-g-l$ with the same
boundary circles. This transformation preserves inequality (\ref{eq_inequality}) and defines an involution $\pi_0$ on $\mathscr{A}^m(n)$.

Then every connected component of the closed surface $\overline{a}\pi_0(a)$ given by gluing
$a$ and $\pi_0(a)$ along the boundary has genus $g+g'+l-1=m$, and  the term $t_{\pi_0}=\pm (\beta_m)^{d_n^{(m)}}$. It is also
clear that for any other $\pi$ satisfying (**) the term $t_\pi$ will be either zero (if one of the components of $\overline{a}\pi_0(a)$ has genus greater than $m$) or will involve $\beta_i$ with $i<m$ (if one of the components of the gluing has genus $<m$).
This completes the proof of the proposition.
\end{proof}

\begin{remark} The sign of the leading term
is the sign of the permutation $\pi_0$; since $\pi_0$ is an involution, the sign can be computed from the number of fixed points.
\end{remark}

\begin{corollary}
 The set $\mathscr{A}^m(n)$ is linearly independent in $A(n)$, for generic values of $\beta_i$'s.
\end{corollary}

Using the standard methods one computes the exponential generating functions for the sizes of the sets $\mathscr{A}^m(n)$ and for the sequence $d_n^{(m)}$:
\[
\sum_{n\geq 0} \frac{|\mathscr{A}^m(n)|}{n!} t^n = \exp\left(
\sum_{g=0}^m \frac{t^{m+1-g}}{(m+1-g)!}(g+1)
\right),
\]
\[
\sum_{n\geq 0} \frac{d^{(m)}_n}{n!} t^n =\left(
\sum_{g=0}^m \frac{t^{m+1-g}}{(m+1-g)!}(g+1)
\right) \exp\left(
\sum_{g=0}^m \frac{t^{m+1-g}}{(m+1-g)!}(g+1)
\right).
\]

In particular, for $m=1$ we get

\begin{corollary}\label{lin deg cor}
 The degree $d_n=d_n^{(1)}$ of the polynomial $\det_n=\det^{(1)}_n$ satisfies
 \[
 \sum_{n\geq 0} \frac{d_n}{n!} t^n =\left(2t+\frac{t^2}2\right)\exp\left(2t+\frac{t^2}2\right).
 \]
 Equivalently, $d_n=\frac12n(a_n+2a_{n-1})$.
\end{corollary}

\begin{conjecture}\label{poly conj}
$\mathscr{A}^m(n)$ spans $A(n)$.
\end{conjecture}

\begin{proposition}
Assume that Conjecture \ref{poly conj} holds for some $m>1$.
Then
\[
\det_n^{(m)}=
\pm (\beta_m)^{d_n^{(m)}}.
\]
Thus $\mathscr{A}^m(n)$ is a basis of $A(n)$ for any $\beta_0,
\beta_1,\ldots,\beta_m$ with $\beta_m\ne 0$ and in any characteristic.
\end{proposition}

\begin{proof} The set of zeroes of $\det_n^{(m)}$ should be
invariant under the scaling $(\beta_0,\beta_1,\beta_2,\ldots,\beta_m)\mapsto (\lambda^{-1}\beta_0,\beta_1,\lambda \beta_2\ldots,\lambda^{m-1}\beta_m)$(see section \ref{subsec-scaling}). Now the result follows from
Proposition \ref{poly lead prop} since the leading term
is multiplied by $(\lambda)^{d_n^{(m)}(m-1)}$ under the scaling
and the potential lower terms are multiplied by lower power
of $\lambda$.
\end{proof}

\begin{remark}
The argument above does not work for $m=1$ since $\beta_1$
does not change under the scaling. However it gives an alternative proof to the known fact that $\det_n=\det_n^{(1)}$ does not depend on $\beta_0$.
\end{remark}

\begin{theorem}
The conjecture \ref{poly conj} holds for $m\leq 2$.
\end{theorem}

\begin{proof}
For $m=1$ it was established earlier.
To prove it for $m=2$, let us first introduce the following notation.

The symmetric group $S_n$ acts on $A(n)$ via the permutation cobordisms that permute $n$ circles. Suppose given a cobordism $y$ which is stabilized by a parabolic subgroup $S_{\lambda}\subset S_n$, for a decomposition $\lambda=(\lambda_1,\dots, \lambda_k)$ of $n$, so that
$\sigma y =y$ for $y\in S_{\lambda}$. To $y$ and $\lambda$ assign the element $\sum_{\lambda}y$ of $A(n)$ given by
\begin{equation} \label{eq_orbit}
\sum_{\lambda}y \ := \ \sum_{\sigma \in S_n/S_{\lambda}}\sigma y.
\end{equation}
That is, pick a representative $\tau$ in each coset $S_n/S_{\lambda}$, form $\tau y$ and sum over cosets.

\begin{figure}[!htbp]
\begin{center}
$ \ \ \ \ \ \ \ \ \ \ $ $ \ \ \ \ \ \ \ $
\begin{tikzpicture}[scale=0.5]
\node at (2,-1) {$1$};
\node[scale=1] at (-1.5,-1) {$\beta_2^3$};
\draw[thick][yscale=0.3] (0,0) circle(0.5);
\draw[thick,yscale=0.3] (2,0) circle(0.5);
\draw[thick,yscale=0.3] (4,0) circle(0.5);
\draw[thick] (0.5,0) to [out=270, in=270] (1.5,0);
\draw[thick] (2.5,0) to [out=270, in=270] (3.5,0);
\draw[thick] (-0.5,0) to [out=270, in=180] (2,-2) to[out=0, in=270] (4.5,0);
\node[scale=1] at (7.5,-1) {$=A + B + C$};
\end{tikzpicture}
\newline\newline
\begin{tikzpicture}[scale=0.5]
\node[scale=1] at (-1.5,-1) {};
\draw[thick,yscale=0.3] (0,0) circle(0.5);
\draw[thick,yscale=0.3] (2,0) circle(0.5);
\draw[thick,yscale=0.3] (4,0) circle(0.5);
\draw[thick] (-0.5,0) to [out=270, in=180] (1,-2) to[out=0, in=270] (2.5,0);
\draw[thick] (0.5,0) to [out=270, in=270] (1.5,0);
\draw[thick] (3.5,0) to [out=270, in=180] (4,-2) to[out=0, in=270] (4.5,0);
\node [scale=1] at (-3.,-1.4) {$A=\beta_2^2 \sum\limits_{(2,1)}$};
\node at (4,-1) {$2$};
\node at (1,-1) {$1$};
\end{tikzpicture}
$ \ \ \ \ \ $
\begin{tikzpicture}[scale=0.5]
\node[scale=1] at (-1.5,-1) {};
\draw[thick][yscale=0.3] (0,0) circle(0.5);
\draw[thick][yscale=0.3] (2,0) circle(0.5);
\draw[thick][yscale=0.3] (4,0) circle(0.5);
\draw[thick] (-0.5,0) to [out=270, in=180] (0,-2) to[out=0, in=270] (0.5,0);
\draw[thick] (1.5,0) to [out=270, in=180] (2,-2) to[out=0, in=270] (2.5,0);
\draw[thick] (3.5,0) to [out=270, in=180] (4,-2) to[out=0, in=270] (4.5,0);
\node [scale=1] at (-3.3,-1.4) {$B=-\beta_2 \sum\limits_{(1,2)}$};
\node at (4,-1) {$2$};
\node at (2,-1) {$2$};
\node at (0,-1) {$1$};
\end{tikzpicture}
\newline\newline
\begin{tikzpicture}[scale=0.5]
\node[scale=1] at (-1.5,-1) {};
\draw[thick][yscale=0.3] (0,0) circle(0.5);
\draw[thick][yscale=0.3] (2,0) circle(0.5);
\draw[thick][yscale=0.3] (4,0) circle(0.5);
\draw[thick] (-0.5,0) to [out=270, in=180] (0,-2) to[out=0, in=270] (0.5,0);
\draw[thick] (1.5,0) to [out=270, in=180] (2,-2) to[out=0, in=270] (2.5,0);
\draw[thick] (3.5,0) to [out=270, in=180] (4,-2) to[out=0, in=270] (4.5,0);
\node [scale=1] at (-2.8,-1.4) {$C=\beta_1 \sum\limits_{(3)}$};
\node at (4,-1) {$2$};
\node at (2,-1) {$2$};
\node at (0,-1) {$2$};
\end{tikzpicture}
\end{center}
\caption{\label{eq_rel_3_holes_2} Relation in $A(3)$ for $Z(T) = \beta_0 + \beta_1 T+\beta_2 T^2$. Numbers $1$ and $2$ show the number of handles (dots) on the component.   Summation means symmetrization with respect to permutations of  boundary components parametrized by cosets of the stabilizer of the surface in $S_3$. Sums $A,B,C$ have 3, 3, 1 terms respectively (7 terms in the right hand side in total).
}
\end{figure}

For $m=2$, the following relations hold in $A(3)$ and $A(4)$, see Figures~\ref{eq_rel_3_holes_2} and~\ref{eq_rel_4_holes_2}. Figure~\ref{eq_rel_3_holes_2} relation reduces a 3-holed torus to a linear combination of other cobordisms, with each summand $A,B,C$ also invariant under the permutation action of $S_3$. Each of these three terms is associated to a surface that has an obvious $S_{\lambda}$-invariance, for the decomposition $\lambda$ shown under the sum sign. The term is the sum over surfaces in its orbit, as described above in (\ref{eq_orbit}).

Figure~\ref{eq_rel_4_holes_2} relation has a similar presentation. For term $B$ there the stabiliser of the surface is the dihedral group $D_{(4)}\subset S_4$ generated by the permutations $(12),(34)$, and $(13)(24)$. The corresponding subgroup is denoted $(2,2)'$, it contains $S_{(2,2)}$ as an index two subgroup. This relation implies Figure~\ref{eq_rel_3_holes_2} relation by capping off a circle by a handle. Capping off by a disk results in a trivial relation.

\vspace{0.1in}

We can exclude components with four boundary circles ($\ell=4$) using Figure~\ref{eq_rel_4_holes_2} relation. To obtain the relations that simplify a genus $i$ surface with $4-i$ boundary components for $i=1,2$, see inequality (\ref{eq_inequality}),  cap $i$ boundary components by handles (one-holed tori) in  Figure~\ref{eq_rel_4_holes_2} relation, resulting in Figure~\ref{eq_rel_3_holes_2} and Figure~\ref{eq_rel_2_holes_2} left relations.  For $i=3$, there is also the relation that a one-holed connected surface of genus three is $0$ in $A(1)$, see Figure~
\ref{eq_rel_2_holes_2} right.

These relations show that any connected component of genus $g$ with $\ell$ boundary circles and $g+\ell > 2+1$ (since $m=2$) simplifies to a linear combination of surfaces in the set $ \mathscr{A}^2(n)$. Consequently, this set spans $A(n)$, establishing Conjecture~\ref{poly conj} for $m=2$.
\end{proof}

\begin{remark}
Originally, relation (\ref{eq_rel_4_holes_2}) was computed in Sage by finding the kernel of the $(|\mathscr{A}^{2}(4)|+1)\times (|\mathscr{A}^{2}(4)|+1)$-matrix of the quadratic form restricted to the elements of $\mathscr{A}^2(4)$ and the four-holed sphere.
\end{remark}

\begin{figure}[!htbp]
\begin{center}
$ \ \ \ \ \ \ \ \ \ \ $ $ \ \ \ \ \ \ \ $
\begin{tikzpicture}[scale=0.5]
\node[scale=1] at (-1.5,-1) {$\beta_2^5$};
\draw[thick][yscale=0.3] (0,0) circle(0.5);
\draw[thick][yscale=0.3] (2,0) circle(0.5);
\draw[thick][yscale=0.3] (4,0) circle(0.5);
\draw[thick][yscale=0.3] (6,0) circle(0.5);
\draw[thick] (-0.5,0) to [out=270, in=270] (6.5,0);
\draw[thick] (0.5,0) to [out=270, in=270] (1.5,0);
\draw[thick] (2.5,0) to [out=270, in=270] (3.5,0);
\draw[thick] (4.5,0) to [out=270, in=270] (5.5,0);
\node[scale=1] at (11,-1) {$=A + B + \sum
\limits_{i=1}^3 C_i + \sum\limits_{i=1}^4 D_i$};
\end{tikzpicture}
\newline\newline
\begin{tikzpicture}[scale=0.5]
\node[scale=1] at (-1.5,-1) {};
\draw[thick][yscale=0.3] (0,0) circle(0.5);
\draw[thick][yscale=0.3] (2,0) circle(0.5);
\draw[thick][yscale=0.3] (4,0) circle(0.5);
\draw[thick][yscale=0.3] (6,0) circle(0.5);
\draw[thick] (-0.5,0) to [out=270, in=180] (2,-2) to[out=0, in=270] (4.5,0);
\draw[thick] (0.5,0) to [out=270, in=270] (1.5,0);
\draw[thick] (2.5,0) to [out=270, in=270] (3.5,0);
\draw[thick] (5.5,0) to [out=270, in=180] (6,-2) to[out=0, in=270] (6.5,0);
\node[scale=1] at (-3.,-1.3) {$A=\beta_2^4 \sum\limits_{(3,1)}$};
\node at (6,-1) {$2$};
\end{tikzpicture}
$ \ \ \ \ \ \ \ \ \ \ $
\begin{tikzpicture}[scale=0.5]
\node[scale=1] at (-1.5,-1) {};
\draw[thick][yscale=0.3] (0,0) circle(0.5);
\draw[thick][yscale=0.3] (2,0) circle(0.5);
\draw[thick][yscale=0.3] (4,0) circle(0.5);
\draw[thick][yscale=0.3] (6,0) circle(0.5);
\draw[thick] (-0.5,0) to [out=270, in=180] (1,-2) to[out=0, in=270] (2.5,0);
\draw[thick] (0.5,0) to [out=270, in=270] (1.5,0);
\draw[thick] (4.5,0) to [out=270, in=270] (5.5,0);
\draw[thick] (3.5,0) to [out=270, in=180] (5,-2) to[out=0, in=270] (6.5,0);
\node [scale=1] at (-3.,-1.3) {$B=\beta_2^4\sum\limits_{(2,2)'}$};
\node at (5,-1) {$1$};
\node at (1,-1) {$1$};
\end{tikzpicture}
\newline\newline
\begin{tikzpicture}[scale=0.5]
\node[scale=1] at (-1.5,-1) {};
\draw[thick][yscale=0.3] (0,0) circle(0.5);
\draw[thick][yscale=0.3] (2,0) circle(0.5);
\draw[thick][yscale=0.3] (4,0) circle(0.5);
\draw[thick][yscale=0.3] (6,0) circle(0.5);
\draw[thick] (-0.5,0) to [out=270, in=180] (1,-2) to[out=0, in=270] (2.5,0);
\draw[thick] (0.5,0) to [out=270, in=270] (1.5,0);
\draw[thick] (5.5,0) to [out=270, in=180] (6,-2) to[out=0, in=270] (6.5,0);
\draw[thick] (3.5,0) to [out=270, in=180] (4,-2) to[out=0, in=270] (4.5,0);
\node [scale=1] at (-3.5,-1.3) {$C_1=-\beta_2^3\sum\limits_{(2,2)}$};
\node at (6,-1) {$2$};
\node at (4,-1) {$2$};
\node at (1,-1) {};
\end{tikzpicture}
$ \ \ \ \ \ \ \ \ \ \ $
\begin{tikzpicture}[scale=0.5]
\node[scale=1] at (-1.5,-1) {};
\draw[thick][yscale=0.3] (0,0) circle(0.5);
\draw[thick][yscale=0.3] (2,0) circle(0.5);
\draw[thick][yscale=0.3] (4,0) circle(0.5);
\draw[thick][yscale=0.3] (6,0) circle(0.5);
\draw[thick] (-0.5,0) to [out=270, in=180] (1,-2) to[out=0, in=270] (2.5,0);
\draw[thick] (0.5,0) to [out=270, in=270] (1.5,0);
\draw[thick] (5.5,0) to [out=270, in=180] (6,-2) to[out=0, in=270] (6.5,0);
\draw[thick] (3.5,0) to [out=270, in=180] (4,-2) to[out=0, in=270] (4.5,0);
\node [scale=1] at (-3.5,-1.3) {$C_2=-\beta_2^3\sum\limits_{(2,1,1)}$};
\node at (6,-1) {$2$};
\node at (4,-1) {$1$};
\node at (1,-1) {$1$};
\end{tikzpicture}
\newline\newline
\begin{tikzpicture}[scale=0.5]
\node[scale=1] at (-2.,-1) {};
\draw[thick][yscale=0.3] (0,0) circle(0.5);
\draw[thick][yscale=0.3] (2,0) circle(0.5);
\draw[thick][yscale=0.3] (4,0) circle(0.5);
\draw[thick][yscale=0.3] (6,0) circle(0.5);
\draw[thick] (-0.5,0) to [out=270, in=180] (1,-2) to[out=0, in=270] (2.5,0);
\draw[thick] (0.5,0) to [out=270, in=270] (1.5,0);
\draw[thick] (5.5,0) to [out=270, in=180] (6,-2) to[out=0, in=270] (6.5,0);
\draw[thick] (3.5,0) to [out=270, in=180] (4,-2) to[out=0, in=270] (4.5,0);
\node [scale=1] at (-4.,-1.3) {$C_3=\beta_1 \beta_2^2 \sum\limits_{(2,2)}$};
\node at (6,-1) {$2$};
\node at (4,-1) {$2$};
\node at (1,-1) {$1$};
\end{tikzpicture}
$ \ \ \ \ \ \ \ \ \ \ $
\begin{tikzpicture}[scale=0.5]
\node[scale=1] at (-1.5,-1) {};
\draw[thick][yscale=0.3] (0,0) circle(0.5);
\draw[thick][yscale=0.3] (2,0) circle(0.5);
\draw[thick][yscale=0.3] (4,0) circle(0.5);
\draw[thick][yscale=0.3] (6,0) circle(0.5);
\draw[thick] (-0.5,0) to [out=270, in=180] (0,-2) to[out=0, in=270] (0.5,0);
\draw[thick] (1.5,0) to [out=270, in=180] (2,-2) to[out=0, in=270] (2.5,0);
\draw[thick] (5.5,0) to [out=270, in=180] (6,-2) to[out=0, in=270] (6.5,0);
\draw[thick] (3.5,0) to [out=270, in=180] (4,-2) to[out=0, in=270] (4.5,0);
\node [scale=1] at (-3.,-1.3) {$D_1=\beta_2^2 \sum\limits_{(1,3)}$};
\node at (6,-1) {$2$};
\node at (4,-1) {$2$};
\node at (2,-1) {$2$};
\end{tikzpicture}
\newline\newline
\begin{tikzpicture}[scale=0.5]
\node[scale=1] at (-1.5,-1) {};
\draw[thick][yscale=0.3] (0,0) circle(0.5);
\draw[thick][yscale=0.3] (2,0) circle(0.5);
\draw[thick][yscale=0.3] (4,0) circle(0.5);
\draw[thick][yscale=0.3] (6,0) circle(0.5);
\draw[thick] (-0.5,0) to [out=270, in=180] (0,-2) to[out=0, in=270] (0.5,0);
\draw[thick] (1.5,0) to [out=270, in=180] (2,-2) to[out=0, in=270] (2.5,0);
\draw[thick] (5.5,0) to [out=270, in=180] (6,-2) to[out=0, in=270] (6.5,0);
\draw[thick] (3.5,0) to [out=270, in=180] (4,-2) to[out=0, in=270] (4.5,0);
\node [scale=1] at (-3.5,-1.3) {$D_2=2 \beta_2^2 \sum\limits_{(2,2)}$};
\node at (6,-1) {$2$};
\node at (4,-1) {$2$};
\node at (2,-1) {$1$};
\node at (0,-1) {$1$};
\end{tikzpicture}
$ \ \ \ \ \ \ \ \ \ \ $
\begin{tikzpicture}[scale=0.5]
\node[scale=1] at (-2,-1) {};
\draw[thick][yscale=0.3] (0,0) circle(0.5);
\draw[thick][yscale=0.3] (2,0) circle(0.5);
\draw[thick][yscale=0.3] (4,0) circle(0.5);
\draw[thick][yscale=0.3] (6,0) circle(0.5);
\draw[thick] (-0.5,0) to [out=270, in=180] (0,-2) to[out=0, in=270] (0.5,0);
\draw[thick] (1.5,0) to [out=270, in=180] (2,-2) to[out=0, in=270] (2.5,0);
\draw[thick] (5.5,0) to [out=270, in=180] (6,-2) to[out=0, in=270] (6.5,0);
\draw[thick] (3.5,0) to [out=270, in=180] (4,-2) to[out=0, in=270] (4.5,0);
\node [scale=1] at (-4.5,-1.3) {$D_3=-3\beta_1\beta_2\sum\limits_{(1,3)}$};
\node at (6,-1) {$2$};
\node at (4,-1) {$2$};
\node at (2,-1) {$2$};
\node at (0,-1) {$1$};
\end{tikzpicture}
\newline\newline
\begin{tikzpicture}[scale=0.5]
\node[scale=1] at (-3.,-1) {};
\draw[thick][yscale=0.3] (0,0) circle(0.5);
\draw[thick][yscale=0.3] (2,0) circle(0.5);
\draw[thick][yscale=0.3] (4,0) circle(0.5);
\draw[thick][yscale=0.3] (6,0) circle(0.5);
\draw[thick] (-0.5,0) to [out=270, in=180] (0,-2) to[out=0, in=270] (0.5,0);
\draw[thick] (1.5,0) to [out=270, in=180] (2,-2) to[out=0, in=270] (2.5,0);
\draw[thick] (5.5,0) to [out=270, in=180] (6,-2) to[out=0, in=270] (6.5,0);
\draw[thick] (3.5,0) to [out=270, in=180] (4,-2) to[out=0, in=270] (4.5,0);
\node [scale=1] at (-5.,-1.3) {$D_4=(3\beta_1^2 - \beta_0\beta_2)\sum\limits_{(4)}$};
\node at (6,-1) {$2$};
\node at (4,-1) {$2$};
\node at (2,-1) {$2$};
\node at (0,-1) {$2$};
\end{tikzpicture}
\end{center}
\caption{\label{eq_rel_4_holes_2} Relation in $A(4)$ for $Z(T) = \beta_0 + \beta_1 T+\beta_2 T^2$. Numbers $1$ and $2$ show the number of handles (dots) on the component.   Summation means symmetrization with respect to permuting the boundary components, as described in the proof. Sums $A,B,C_1,C_2,C_3,D_1,D_2,D_3,D_4$ have 4, 3, 6, 12, 6, 4, 6, 4, 1 terms respectively (46 terms in the right hand side in total).}
\end{figure}

\begin{figure}[!htb]
\begin{center}
\begin{tikzpicture}[scale=0.5]
\node[scale=1] at (-2.,-1) {};
\draw[thick][yscale=0.3] (0,0) circle(0.5);
\draw[thick][yscale=0.3] (2,0) circle(0.5);
\draw[thick][yscale=0.3] (5,0) circle(0.5);
\draw[thick][yscale=0.3] (7,0) circle(0.5);
\draw[thick] (-0.5,0) to [out=270, in=180] (1,-2) to[out=0, in=270] (2.5,0);
\draw[thick] (0.5,0) to [out=270, in=270] (1.5,0);
\draw[thick] (6.5,0) to [out=270, in=180] (7,-2) to[out=0, in=270] (7.5,0);
\draw[thick] (4.5,0) to [out=270, in=180] (5,-2) to[out=0, in=270] (5.5,0);
\node [scale=1] at (-1.,-1.3) {$\beta_2$};
\node [scale=1] at (3.5,-1.3) {$ = $};
\node at (7,-1) {$2$};
\node at (5,-1) {$2$};
\node at (1,-1) {$2$};
\end{tikzpicture}
$ \ \ \ \ \ \ \ \ \ \ \ \ $
\begin{tikzpicture}
\draw[thick][yscale=0.3] (0,0) circle(0.5);
\draw[thick] (-0.5,0) to [out=270, in=180] (0,-1) to[out=0, in=270] (0.5,0);
\node at (0,-0.5) {$3$};
\node [scale=1] at (1.,-0.5) {$=0$};
\end{tikzpicture}
\end{center}
\caption{\label{eq_rel_2_holes_2} Relations in $A(2)$ and $A(1)$ for $Z(T) = \beta_0 + \beta_1 T+\beta_2 T^2$. Numbers $2$ and $3$ show the number of handles (dots) on the component.
}
\end{figure}
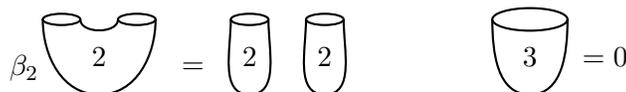

\section{Cobordisms of fractional genus and other decorations}\label{sec_frac_genus}

{\it Fractional genus.}
Recall the defining relations in Proposition~\ref{AP_def_rel} on generators $x$ and $u$ of the algebra $B_S$.
If $\alpha_{n}\not= 0$ for some even $n\ge 0$, then the element
\begin{equation}\label{eq_may_be_fraction}
\alpha_{n}^{-1}x^{n/2} u x^{n/2}
\end{equation}
is an idempotent  in $B_S$. This is an obvious way to  get an idempotent in $B_S$ unless the power series $Z(T)$ has nontrivial coefficients  only at odd powers  of  $T$. With a minor effort, a version of the above  idempotent can be produced in the latter case as well. Namely, for odd $n$ and with $\alpha_n\not=0$, we can try to make sense of  the expression (\ref{eq_may_be_fraction}). For that one needs "cobordism" $x^{1/2}$, which should  be a "genus $1/2$" surface,  with some boundary components. Let us consider an even more general  case of a "genus $1/\ell$" surface for some $\ell>1$. We simply  introduce a fractional dot $x^{1/\ell}$ with the relation that its $\ell$-th power is the handle, see Figures~\ref{fig4_4}   and~\ref{fig4_3_1}.

\begin{figure}[!htb]
\begin{center}
\hspace*{-2mm}
\begin{tikzpicture}
[xscale=0.7, yscale=0.7]

\draw (0,0) to [out=270,in=180] (1,-1.5) to [out=0,in=270] (2,0);
\draw [yscale=0.3] (1,0) circle(1);
\draw [fill=black] (1,-0.7) circle(0.05);
\node [scale=0.7] at (1,-1) {$m/l$};
\node at (3,-0.5) {$ = x^{m/l}$};

\end{tikzpicture}
\caption{\label{fig4_3_1} Dot of fractional order $m/\ell.$}
\end{center}

\end{figure}
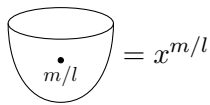
\begin{figure}[!htb]
\begin{center}
\hspace*{-2mm}
\begin{tikzpicture}
[xscale=0.7, yscale=0.7]

\draw (0,0) to [out=270,in=180] (3,-1.5) to [out=0,in=270] (6,0);
\draw [yscale=0.2] (3,0) circle(3);
\draw [fill=black] (1,-0.7) circle(0.05);
\node [scale=0.7] at (1,-1) {$1/l$};
\draw [fill=black] (2,-0.85) circle(0.05);
\node [scale=0.7] at (2,-1.15) {$1/l$};
\draw [fill=black] (3,-0.9) circle(0.05);
\node [scale=0.7] at (3,-1.3) {$1/l$};
\node at (7,-0.5) {$=$};
\node at (2,-2) {$l$ times};

\draw (8.5,0) to [out=270,in=180] (9.5,-1.5) to [out=0,in=270] (10.5,0);
\draw [yscale=0.3] (9.5,0) circle(1);
\draw [fill=black] (9.5,-0.7) circle(0.1);

\node at (12,-0.5) {$=$};

\draw [yscale=0.3] (15,0) circle(1);
\draw (14,0) to [out=250, in=180] (15,-2) to [out=0, in=290] (16,0);

\draw (15.1,-0.5) to [out=230, in=130] (15.1,-1.5);
\draw (15.0,-0.7) to [out=290, in=70] (15.0,-1.3);

\end{tikzpicture}
\begin{tikzpicture}
[xscale=0.7, yscale=0.7]

\draw [yscale=0.3] (0,0) circle(1);
\draw [yscale=0.3] (0,-10) circle(1);
\draw (-1,0) -- (-1,-3);
\draw (1,0) -- (1,-3);
\draw [fill=black] (-0.5,-1) circle(0.05);
\draw [fill=black] (-0.5,-2) circle(0.05);
\node at (0, -1) {$\frac{1}{2}$};
\node at (0, -2) {$\frac{1}{2}$};

\node at (2,-1.5) {$=$};

\draw [yscale=0.3] (4,0) circle(1);
\draw [yscale=0.3] (4,-10) circle(1);
\draw (3,0) -- (3,-3);
\draw (5,0) -- (5,-3);
\draw [fill=black] (3.7,-1.5) circle(0.05);
\node at (4.1, -1.5) {$1$};

\node at (6,-1.5) {$=$};

\draw [yscale=0.3] (8,0) circle(1);
\draw [yscale=0.3] (8,-10) circle(1);
\draw (7,0) -- (7,-3);
\draw (9,0) -- (9,-3);

\draw (8.1,-1.0) to [out=230, in=130] (8.1,-2.0);
\draw (8.0,-1.2) to [out=290, in=70] (8.0,-1.8);

\end{tikzpicture}
$\ \ \ \ \ \ \ \ $
\begin{tikzpicture}
[xscale=0.5, yscale=0.7]

\draw [yscale=0.3] (0,0) circle(1);
\draw [yscale=0.3] (0,-10) circle(1);
\draw (-1,0) to [out=270, in=90] (-1.5,-1.5) to [out=270, in=90] (-1,-3);
\draw (1,0) to [out=270, in=90] (1.5,-1.5) to [out=270, in=90] (1,-3);
\draw (-0.8,-1.4) to [out=300, in=240] (0.8,-1.4);
\draw (-0.6,-1.5) to [out=30, in=150] (0.6,-1.5);
\draw[fill=black] (-0.3,-0.8) circle(0.05);
\node at (0.3,-0.8) {$\frac{1}{2}$};
\draw[fill=black] (-0.3,-2.2) circle(0.05);
\node at (0.3,-2.2) {$\frac{1}{2}$};

\node at (2,-1.5) {$=$};

\draw [yscale=0.3] (4,0) circle(1);
\draw [yscale=0.3] (4,-10) circle(1);
\draw (3,0) to [out=270, in=90] (2.5,-1.5) to [out=270, in=90] (3,-3);
\draw (5,0) to [out=270, in=90] (5.5,-1.5) to [out=270, in=90] (5,-3);
\draw (3.2,-1.) to [out=300, in=240] (4.8,-1.);
\draw (3.4,-1.1) to [out=30, in=150] (4.6,-1.1);
\draw (3.2,-2) to [out=300, in=240] (4.8,-2);
\draw (3.4,-2.1) to [out=30, in=150] (4.6,-2.1);

\end{tikzpicture}
\caption{\label{fig4_4} Top: Fractional dot $x^{1/\ell}$ can freely float along a connected component. Its $\ell$-th power is the handle. Bottom: examples of relations on dots and handles.}
\end{center}
\end{figure}
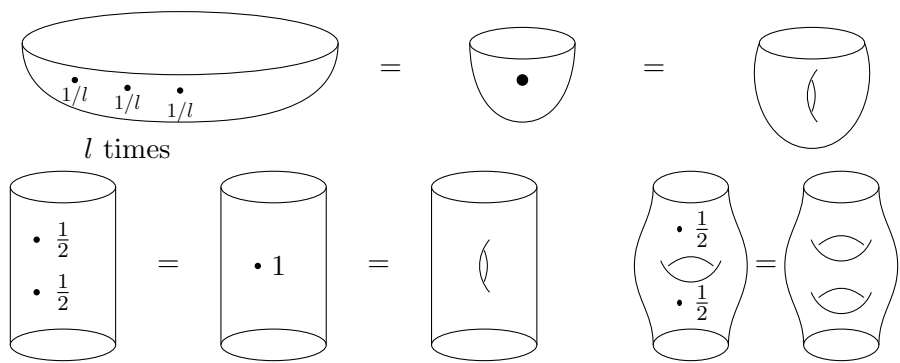

Fractional dots  $a/\ell$ and $b/\ell$, $a,b\in \Z_+$, floating on the same component, can merge into the fractional dot $(a+b)/\ell$. Vice versa, the latter can split into  $a/\ell$ and $b/\ell$ dots. Dot $\ell/\ell=1$ converts into $x$  and equals  a handle on the component.

Formally, one can introduce the category of fractional cobordisms $\Cobtwo^{\ell}$.  Its objects are non-negative integers $n\ge 0$ and morphisms from $n$ to $m$ are the  diffeomorphism classes rel boundary  of oriented cobordisms from $n$ to $m$ circles with dots floating on the  components and labelled by elements of the commutative semigroup $H=\frac{1}{\ell}\Z_+=\{0,1/\ell,2/\ell, \dots\}.$ Dots can merge, adding their labels, and the label $1$ dot equals the handle. Dot $0$ can  be erased. A connected closed cobordism of genus $g$ with a  dot $\frac{m}{\ell}$ reduces to a 2-sphere decorated by the dot  $\frac{m+g\ell}{\ell}\in H$.

Universal constructions for 2-dimensional cobordisms, as described here and in~\cite{Kh1,KS}, extend in a straightforward way to $\Cobtwo^{\ell}$ for any $\ell \ge 2$ (the original theory corresponds to $\ell=1$). Parameters of the theory are $\alpha_{n/\ell} \in \kk$, over all $n\ge 0$, encapsulated by the power series in  $T^{1/\ell}$,
\begin{equation}
    Z_{\alpha}(T^{1/\ell}) = \sum_{n\ge 0}\alpha_{n/\ell} T^{n/\ell}.
\end{equation}
State spaces $A_{\alpha}(k)$ of $k$ circles are defined as in~\cite{Kh1}, and the rationality result is  proved in the same way.
\begin{prop} Vector spaces $A_{\alpha}(k)$ are finite-dimensional for all $k\ge 0$ iff $A_{\alpha}(1)$ is finite-dimensional iff $Z_{\alpha}(T^{1/\ell})$ is a rational function,
\begin{equation}\label{eq_Z_frac}
    Z_{\alpha}(T^{1/\ell})=\frac{P(T^{1/\ell})}{Q(T^{1/\ell})},
\end{equation}
for coprime polynomials $P$ and $Q$, with $Q(0)\not=0$.
\end{prop}
Generalizations of the Deligne category extend to this case as well, and each rational function as in (\ref{eq_Z_frac}) gives rise to several categories by direct analogy with~\cite{RS}. These categories have finite-dimensional hom spaces and  include the analogue of the partition category, the Deligne category, and their quotients by ideals of negligible morphisms.

\vspace{0.1in}

We should warn the reader that cobordisms of fractional genus, as above, are simply  decorated  surfaces that are morphisms  of $\Cobtwo^{\ell}$. They don't carry any of the rich structure associated with the usual surfaces, such as the mapping class group, moduli spaces of complex structures, and so  on.

\vspace{0.1in}

{\it Commutative monoid decorations.}
More generally, one can take any commutative monoid $H$ (with the binary operation written additively as $+$), together with a monoid homomorphism $\psi:\Z_+ \lra H$.
To $\psi$ one can assign the category  $\Cobtwo^{\psi}$ of  oriented decorated 2D cobordisms. As before, objects of this category are non-negative integers $n\in \Z_+$, while the morphisms are oriented 2D cobordisms (modulo rel boundary  diffeomorphisms) decorated by dots labelled by elements of $H$. Dots labelled $a,b\in H$ floating on the same component can merge into a dot labelled $a+b$,  see Figure~\ref{fig4_5}. Dot labelled $0\in H$ can be erased. A handle on a component equals a dot on  that component labelled by  $\psi(1)$.

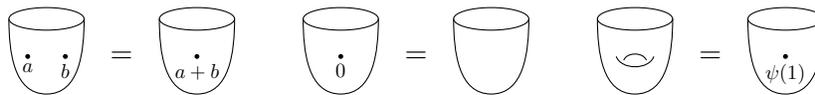
\begin{figure}[ht]
\begin{center}
\begin{tikzpicture}
[scale=0.5]
\draw [yscale=0.3] (1,0) circle(1);
\draw (0,0) to [out=270, in=180] (1,-2) to [out=0, in=270] (2,0);
\draw [fill=black] (0.5,-1) circle(0.05);
\draw [fill=black] (1.5,-1) circle(0.05);
\node [below, scale=0.7] at (1.5,-1) {$b$};
\node [below, scale=0.7] at (0.5,-1) {$a$};
\node at (3,-1) {$=$};
\draw [yscale=0.3] (5,0) circle(1);
\draw (4,0) to [out=270, in=180] (5,-2) to [out=0, in=270] (6,0);
\draw [fill=black] (5,-1) circle(0.05);

\node [below, scale=0.7] at (5.,-1) {$a+b$};
\end{tikzpicture}
$\ \ \ \ \ $
\begin{tikzpicture}
[scale=0.5]
\draw [yscale=0.3] (1,0) circle(1);
\draw (0,0) to [out=270, in=180] (1,-2) to [out=0, in=270] (2,0);
\draw [fill=black] (1,-1) circle(0.05);
\node [below, scale=0.7] at (1,-1) {$0$};
\node at (3,-1) {$=$};
\draw [yscale=0.3] (5,0) circle(1);
\draw (4,0) to [out=270, in=180] (5,-2) to [out=0, in=270] (6,0);
\end{tikzpicture}
$\ \ \ \ \ $
\begin{tikzpicture}
[scale=0.5]
\draw [yscale=0.3] (1,0) circle(1);
\draw (0,0) to [out=270, in=180] (1,-2) to [out=0, in=270] (2,0);
\draw (0.5,-1) to [out=290, in=250] (1.5,-1);
\draw (0.7,-1.1) to [out=70, in=110] (1.3,-1.1);
\node at (3,-1) {$=$};
\draw [yscale=0.3] (5,0) circle(1);
\draw (4,0) to [out=270, in=180] (5,-2) to [out=0, in=270] (6,0);
\draw [fill=black] (5.,-1) circle(0.05);
\node [below, scale=0.7] at (5.,-1) {$\psi(1)$};
\end{tikzpicture}
\caption{\label{fig4_5} Dots $a$ and $b$ merge  into $a+b$ dot; dot $0$ can be erased; handle equals the dot labelled $\psi(1)$.}
\end{center}
\end{figure}

Closed $\psi$-cobordisms are disjoint unions of their connected components, classified by their genus, which is an element of $H$. The analogue of evaluation is a  map  of sets
$\alpha: H\lra \kk$
which can be written via formal power series
\begin{equation}
   Z_{\alpha} = \sum_{h\in H} \alpha_h \, h
\end{equation}
and viewed as an element of  the  dual vector space $(\kk H)^{\ast}$.

In $\Cobtwo^{\psi}$ connected cobordisms from  $0$ to $1$ are parametrized by  elements  $h\in H$  and correspond to a 2-disk with a dot labelled $H$. Consequently, the space $A_{\alpha}(1)$ is the  $\kk H$-submodule of $(\kk H)^{\ast}$ generated by the functional $Z_{\alpha}$. It is finite-dimensional iff $\alpha$ is a representative function on $H$, see \cite{Kh1}.

\vspace{0.1in}

Notice that $\psi$ does not have to be injective. However, one can specialize to the case when $H$ is a free monoid and $\psi$ is injective, and then get any desired defining relations on generators of $H$ by restricting to suitable subspaces of $(\kk H)^{\ast}$. Taking large $H$, however, may move the emphasis from 2D cobordisms and representative functions on them to, for the most part, studying representative functions on $\kk H$, with only a  meagre input from cobordisms.

\vspace{0.1in}

Another potentially interesting specialization is to the periodic genus.  For that specialization genus  does not need to be fractional. Consider the  quotient  of the  cobordism  category by the relation that the $M$-th power of the handle is identity, that is, can  be removed. This corresponds to working with the monoid and map
\begin{equation}
    H = \Z/M\Z , \ \ \psi; \Z_+\lra H, \ \psi(1)=1,
\end{equation}
that is, modding out $\Z_+$ by $M\Z_+$. In the language of $\alpha$-evaluations, one is looking at "$M$-periodic" power series, that is, $\alpha_{M+n}=\alpha_n$ for all $n\ge 0$. Equivalently, the power series
\begin{equation}\label{eq_Z_cycl}
    Z_{\alpha}(T) = \frac{P(T)}{1-T^M}, \ \ \deg(P(T))< M ,
\end{equation}
is determined by the coefficients of $P(T)$. Such  evaluations necessarily extend to the category of cobordisms with integral genus, see the  next remark. Fractional version of (\ref{eq_Z_cycl}) also makes sense, with the power series
\begin{equation}\label{eq_Z_cycl_2}
    Z_{\alpha}(T^{1/\ell}) = \frac{P(T^{1/\ell})}
    {1-T^{M/\ell}},
    \ \ \deg(P(x))< M ,
\end{equation}
and  not necessarily integral $M/\ell$. In this  theory $M/\mathrm{gcd}(M,\ell)$ handles on  a  component  can be erased.

\vspace{0.1in}

\emph{Remark:}
Paper~\cite{KR2} discusses several
rank two Frobenius extensions $R_{\ast}\subset A_{\ast}$  used in various flavors of $SL(2)$ link homology. Here $R_{\ast}$ is a ground commutative ring and $A_{\ast}$ is a Frobenius $R_{\ast}$-algebra, which is, in particular, free of rank two over $R_{\ast}$. Extension $(R_{\mathcal{D}},A_{\mathcal{D}})$ considered in~\cite{KR2} makes use of the   \emph{anti-handle}, a formal inverse of the handle cobordism, denoted by $\star^{-1}$ in that  paper. This extension essentially describes  Lee's homology theory and also gives a monoidal functor from the category  $\Cobtwo^{\psi}$ to the category of free $R_{\mathcal{D}}$-modules, where
\begin{equation}
    H=\Z, \ \ \psi: \Z_+ \lra \Z
\end{equation}
is the usual homomorphism from the monoid of  non-negative integers (under addition) to integers. The functor assigns $A_{\mathcal{D}}^{\otimes n}$ to the  union of $n$ circles and the structure maps of that Frobenius algebra (unit, counit, multiplication, comultiplication) to the basic cobordisms: cup, cap, pants, copants. Multiplication by  the handle endomorphism of $A_{\mathcal{D}}$ is invertible and allows to introduce the antihandle (dot labelled $-1$) as the inverse of the handle endomorphism.

A similar localization appears in~\cite{KR1} in the context of evaluations of unoriented $SL(3)$ foams, where one can invert the discriminant and work with suitable decorations on foams.

\vspace{0.1in}

Allowing connected sums of cobordisms in $\Cobtwo$ with $\mathbb{RP}^2$ (which results in unorientable cobordisms) corresponds to working with the monoid and the map
\begin{equation}
    H=\langle 1,b\rangle /(3b=b+1), \ \ \psi:\Z_+ \lra H, \psi(1)=1,
\end{equation}
with dot labelled $b$ corresponding to the connected sum with $\mathbb{RP}^2$. In this monoid there is no cancellation, and $b+b\not= 1$ although $b+b+b=b+1$. Topologically, connected sum with three  copies of  $\mathbb{RP}^2$ is diffeomorphic to the connected sum with one  $\mathbb{RP}^2$ and the torus, but connected sum with two copies of $\mathbb{RP}^2$ (equivalently, with the Klein bottle) is not diffeomorphic to adding a handle, when applied to an orientable connected component.
Monoid $H$ surjects onto $\frac{1}{2}\Z_+$ by sending $b$ to $\frac{1}{2}$, intertwining homomorphisms $\psi$ for these monoids.

\vspace{0.1in}

{\it Extending to commutative algebras.}
Decorations of two-dimensional cobordisms by elements  of  a commutative monoid can be further generalized. Observe that  a map $\psi$ as above from  the "genus" semigroup $\Z_+$ into a commutative monoid $H$, together with the "trace" or evaluation $\alpha:H\lra \kk$ gives rise to two maps, that we still denote $\psi$ and $\alpha$,
\begin{equation}
    \kk[x]\cong \kk \Z_+ \stackrel{\psi}{\lra} \kk H \stackrel{\alpha}{\lra} \kk.
\end{equation}
The first  map is a homomorphism of commutative  $\kk$-algebras, the second (evaluation) map is $\kk$-linear. One  can generalize from $\kk H$  to a commutative $\kk$-algebra $B$  together with a homomorphism  $\psi$ and a $\kk$-linear  trace map
\begin{equation}\label{eq_data_k}
    \kk \Z_+ \cong \kk[x] \stackrel{\psi}{\lra} B \stackrel{\alpha}{\lra} \kk.
\end{equation}
Instead of the  set-theoretic category $\Cobtwo$ which does not have  a linear structure one starts with  the $\kk$-linear category $\kk\Cobtwo$ with the same objects as $\Cobtwo$ and morphisms -- finite $\kk$-linear combinations  of morphisms in $\Cobtwo$. One then modifies $\kk\Cobtwo$ to the category
$\Cobtwo^{\psi}$ as follows. Category $\Cobtwo^{\psi}$ has objects $n\in \Z_+$. Morphisms are $\kk$-linear combinations of oriented 2D cobordisms as before with dots  labelled by elements of $B$ floating on components and the following relations, see also Figure~\ref{fig4_6}:

\begin{itemize}
    \item Dots are subject to the obvious addition and product rules for elements of $B$,
    \item A handle on a cobordism can  be  replaced by a dot labelled $\psi(x)$,
    \item A closed surface of genus $g$  with dot labelled $b$ evaluates to  $\alpha(b\psi(x)^g)\in \kk.$
\end{itemize}

\begin{figure}[ht]
\begin{center}
\begin{tikzpicture}
\draw (0,0)--(0,1)--(2,1)--(2,0)--cycle;
\draw [fill=black] (0.3,0.5) circle(0.05);
\node at (0.3,0.3) {$a$};
\draw [fill=black] (1.6,0.5) circle(0.05);
\node at (1.5,0.3) {$b$};
\node at (2.5,0.5) {$=$};
\draw (3,0)--(3,1)--(5,1)--(5,0)--cycle;
\draw [fill=black] (4,0.5) circle(0.05);
\node at (4,0.3) {$ab$};
\end{tikzpicture}
$ \ \ \ \ \  \ \ $
\begin{tikzpicture}
\draw (1,0)--(1,1)--(2,1)--(2,0)--cycle;

\node at (2.5,0.5) {$=$};
\draw (3,0)--(3,1)--(5,1)--(5,0)--cycle;
\draw [fill=black] (4,0.5) circle(0.05);
\node at (4.5,0.3) {$\psi(x)$};
\draw [fill=white] (1.5,0.6) to [out=180, in=0] (0.6,0.9) to [out=180, in=90] (0.4,0.5) to [out=270, in=180] (0.6,0.1) to [out=0, in=180] (1.5,0.4);
\draw (0.5,0.5) to [out=290, in=250] (1,0.5);
\draw (0.6,0.45) to [out=60, in=120] (0.9,0.45);
\end{tikzpicture}

$\ \ \ \ $

\begin{tikzpicture}
\draw (0.0,0.5) to [out=290, in=250] (0.5,0.5);
\draw (0.1,0.45) to [out=60, in=120] (0.4,0.45);
\draw (1.0,0.5) to [out=290, in=250] (1.5,0.5);
\draw (1.1,0.45) to [out=60, in=120] (1.4,0.45);
\draw (2.0,0.5) to [out=290, in=250] (2.5,0.5);
\draw (2.1,0.45) to [out=60, in=120] (2.4,0.45);
\draw [yscale=0.3] (1.5,1.5) circle (1.7);
\draw [fill=black] (2.7,0.6) circle(0.05);
\node at (2.7,0.3) {$b$};
\node at (4,0.5) {$=$};
\node [blue] at (1,-0.5) {$g$ handles};
\draw [yscale=0.3] (6.5,1.5) circle (1.7);
\draw [fill=black] (6,0.5) circle(0.05);
\node at (7,0.4) {$\psi(x)^g b$};
\node at (10,0.5) {$=\ \ \alpha\left(b \psi(x)^g\right)$};
\end{tikzpicture}
\caption{\label{fig4_6} Skein relations in category $\Cobtwo^{\psi}$. Note that $a,b$ are elements of an algebra $B$ rather than of a monoid $H$ (where the binary operation is denoted $+$), which leads to the addition in Figure~\ref{fig4_5} becoming multiplication here.}
\end{center}
\end{figure}

In  this way, a surface of  genus $g$ with  dots labelled  $a_1, \dots, a_k$ floating on it reduces to a genus zero surface with the same boundary and a single dot labelled $\psi(x)^g a_1\dots  a_k$.
A closed connected component reduces to a 2-sphere with a dot $b$. It then evaluates to $\alpha(b)\in \kk$. Thus, any closed component evaluates to an element of $\kk$. In this way a dotted cobordism reduces to a viewable cobordism with all components of genus  zero and at most a single dot on each connected  component, labelled by some element of $B$.

Pick a basis $\{b\}_{b\in C}$ of $B$ that contains $1\in  B$. A morphism  from $n$ to $m$ in $\Cobtwo^{\psi}$ reduces to a linear  combination  of genus zero viewable cobordisms with dots on each components labelled by elements of the basis $C$ of  $B$. In fact, such dotted cobordisms constitute a basis in the hom space $\Hom_{\Cobtwo^{\psi}}(n,m)$. Vice  versa, category $\Cobtwo^{\psi}$ can be defined via these  bases and multiplication rules that come from composition of cobordisms, converting a handle to $\psi(x)$, multiplication  in the basis $C$ of  $B$ and evaluation map $\alpha\in B^{\ast}$. Basis elements are
parametrized by a choice of a  set-theoretical  partition  $\lambda\in D^m_n$ of $n+m$ boundary circles together with an assignment of an element of the basis $C$ of $B$ to each component.

\vspace{0.1in}

Next, assume that $\alpha\in B^{\ast}=\Hom_{\kk}(B,\kk)$ is a representative function and exclude the  trivial case $\alpha=0$. This means that the hyperplane $\ker(\alpha)\subset B$  contains an ideal $I\subset B$  of finite codimension, $\dim(B/I)<\infty$, see~\cite{Mn}, for instance. Assume that $I$ is the largest such ideal. Element $\alpha$ generates  a subrepresentation $B\alpha\subset  B^{\ast}$ of $B$ that factors through the action of $B/I$ and is isomorphic to a free rank one $B/I$-module.

\vspace{0.1in}

\emph{Remark:} Algebra $B'=B/I$ is a commutative Frobenius $\kk$-algebra, with the nondegenerate trace $\alpha$ and a preferred element, which is the image of $\psi(x)$ under the quotient map $B\lra B/I$. The constructions that follow can alternatively be done with such a commutative algebra $B'$ (necessarily finite-dimensional over $\kk$), equipped with  a nondegenerate trace and a preferred element.

\vspace{0.1in}

Next, we quotient $\Cobtwo^{\psi}$ by the relations that a dot labelled by $z\in I$ is zero. Such a dot can be expanded as a linear combination in the basis $C$. A possible convenient basis can  be  formed by choosing a basis  a basis $C_I$ of $I$ and extending it to a basis of $B$ that contains $1$ (the  latter  condition is also  for convenience). Let us denote such a basis by $C=C_I\sqcup C',$ with $1\in C'$ and $C'$ descending to a basis of the quotient $B/I$. Set $C'$ is finite.

\vspace{0.1in}

Denote the quotient category by $\scob^{\psi}_{\alpha}$. It is the analogue of the category $\PCobal$ from~\cite{KS}. One  can check that  a  basis of $\Hom_{\scob^{\psi}_{\alpha}}(n,m)$ is given by choosing a set-theoretic partition $\lambda\in D^m_n$ for $n+m$ boundary circles and assigning  an  element of $C'$ to  each component  of  the  partition. In particular, hom spaces in the  category $\scobal^{\psi}$ are finite-dimensional. The space of homomorphisms
\begin{equation}
A^{\psi}_{\alpha}(1):=\Hom_{\scobal^{\psi}}(0,1)
\end{equation}
is a commutative algebra under the pants cobordism, naturally isomorphic  to the Frobenius algebra $B/I$ above. Now form the additive
Karoubi closure
\begin{equation}
    \dcobal^{\psi} \ := \ \Kar(\scobal^{\psi})
\end{equation}
to get a $\kk$-linear idempotent-complete rigid symmetric monoidal category with finite-dimensional hom spaces. This is the analogue  of the Deligne  category for the  data $(B,\psi, \alpha)$ as in (\ref{eq_data_k}), with a representative functional $\alpha$ (trivial case $\alpha=0$ gives the zero category).

Categories $\Cobtwo^{\psi}$,   $\scobal^{\psi}, $  and $\dcobal^{\psi}$ have a trace map  given on a decorated $(n,n)$-cobordism $x$  representing an element in $\Hom(n,n)$  by closing $x$  via $n$ annuli into a closed cobordism $\widehat{x}$ and  evaluating via $\alpha$:
\begin{equation}
    \tr(x) = \alpha(\widehat{x}).
\end{equation}
Denote by  $J^{\psi}_{\alpha}$ the two-sided ideal in $\scobal^{\psi}$ of \emph{negligible morphisms} for the trace $\tr$. Let
\begin{equation}
       \Cob_{\alpha}^{\psi}= \mathrm{SCob}^{\psi}_{\alpha} / J^{\psi}_{\alpha}
\end{equation}
be the quotient  category by this ideal. Likewise, let $\udcobal^{\psi}$ be the quotient of the Deligne category $\dcobal^{\psi}$ by the negligible ideal for $\tr$.

\vspace{0.1in}

For a representative $\alpha$, as before,
these  categories can be organized into the following diagram of categories and functors, with a commutative square on the right.
\begin{equation} \label{eq_seq_cd_4}
\begin{CD}
\Cobtwo @>>> \kk\Cobtwo @>>> \Cobtwo^{\psi} @>>> \mathrm{SCob}^{\psi}_{\alpha}   @>>> \dcobal^{\psi} \\
@.  @.  @.   @VVV     @VVV  \\
 @.   @.    @.  \Cob^{\psi}_{\alpha}  @>>>
 \udcobal^{\psi}
\end{CD}
\end{equation}
This diagrams of categories is fully analogous to the ones described in (\ref{eq_seq_cd_1}) above and  in~\cite{KS,KQR,Kh2}. The four categories in the commutative square have finite-dimensional hom spaces. The two categories on the far right are idempotent complete.

\vspace{0.1in}

\emph{Remark:}
The construction above is likely  to be  more interesting when the algebra $B$ is not very  large.
One may, for  instance, take $B=\kk[x,y]/(g(x,y))$, the quotient of the  ring of polynomials in two variables by a polynomial that depends nontrivially on both  $x$ and  $y$, and define  $\psi:\kk[x]\lra B$  by $\psi(x)=x$.

\vspace{0.1in}

\emph{Remark:} The category of thin flat 2-dimensional cobordisms in~\cite{KQR} has commuting \emph{hole} and \emph{handle} cobordisms. Similar to the discussion above, dot-decorated version of that category can be introduced, with elements of a commutative monoid $H$ floating  on components of cobordism. One fixes two elements of $H$, to equate to a  handle and a hole, respectively. Equivalently, a homomorphism $\psi:\Z_+\times \Z_+\lra H$ can be fixed for that.

If, instead, elements of a commutative $\kk$-algebra $B$ are made to float on cobordism's components, one should choose two elements of $B$, to equate to the handle and the hole, respectively. In the version of the thin cobordism category~\cite{KQR} where side boundaries are colored by colors $\{1,\dots, r\}$, there are $r$ different holes, one for each color of its boundary. Then to combine $B$ with the handles and holes, one chooses $r+1$ elements in $B$.

\end{document}